\documentclass[11pt]{amsart}
\usepackage{amsmath}
\usepackage{amssymb}
\usepackage{amsthm}
\usepackage{amsfonts}
\usepackage{latexsym}
\usepackage{graphicx}
\usepackage[pdftex,bookmarks=true,bookmarksnumbered=true]{hyperref}
\usepackage{color}
\usepackage{comment}
\RequirePackage{doi}
\let\visiblecomments y 

\def\setof#1{\mbox{$\{\,#1\,\}$}}
\newcommand{\inte}{\operatorname{int}}
\newcommand{\pto}{\nrightarrow}

\renewcommand{\epsilon}{\varepsilon}
\renewcommand{\rho}{\varrho}
\renewcommand{\phi}{\varphi}
\newcommand{\DS}{\displaystyle}
\newcommand{\N}{{\mathbb N}}
\newcommand{\Z}{{\mathbb Z}}

\newcommand{\R}{{\mathbb R}}

\newcommand{\mbbP}{{\mathbb P}}
\renewcommand{\P}{{\mathbb P}}
\newcommand{\cA}{{\mathcal A}}
\newcommand{\cB}{{\mathcal B}}
\newcommand{\cC}{{\mathcal C}}
\newcommand{\cD}{{\mathcal D}}
\newcommand{\cE}{{\mathcal E}}

\newcommand{\cK}{{\mathcal K}}

\newcommand{\cM}{{\mathcal M}}

\newcommand{\cS}{{\mathcal S}}

\newcommand{\cV}{{\mathcal V}}

\newcommand{\cX}{{\mathcal X}}

\newcommand{\dom}{\operatorname{dom}}
\newcommand{\im}{\operatorname{im}}
\newcommand{\Inv}{\operatorname{Inv}}

\newcommand{\Exit}{\operatorname{Mo}}
\newcommand{\Bd}{\operatorname{Bd}}
\newcommand{\bd}{\operatorname{bd}}
\newcommand{\Cl}{\operatorname{Cl}}
\newcommand{\cl}{\operatorname{cl}}
\newcommand{\Crit}{\operatorname{Crit}}
\newcommand{\Tail}{\operatorname{Tail}}
\newcommand{\Head}{\operatorname{Head}}

\def\cse#1{\mbox{$\langle#1\rangle_\epsilon$}}
\def\card#1{\# #1}
\newcommand{\bcap}{\mathbin{\bar{\cap}}}
\newcommand{\btimes}{\mathbin{\bar{\times}}}
\title{Creating Semiflows on Simplicial Complexes from Combinatorial Vector Fields}

\author{Marian Mrozek}
\address{Marian Mrozek, Division of Computational Mathematics,
  Institute of Computer Science and Computational Mathematics,
  Faculty of Mathematics and Computer Science,
  Jagiellonian University, ul.~St. \L{}ojasiewicza 6, 30-348~Kra\-k\'ow, Poland.
}
\email{marian.mrozek@uj.edu.pl}
\author{Thomas Wanner}
\address{Thomas Wanner, Department of Mathematical Sciences,
  George Mason University,
  Fairfax, Virginia 22030, USA.
}
\email{twanner@gmu.edu}
\thanks{The research of M.M.\ was partially supported by the Polish National
        Science Center under Ma\-estro Grant No.~2014/14/A/ST1/00453 
        and under Opus Grant No. 2019/35/B/ST1/00874. 
        T.W.\ was partially supported by NSF grants DMS-1114923 and DMS-1407087, and by
        the Simons Foundation under Award~581334. Both authors gratefully
        acknowledge the support of the Hausdorff Center for Mathematics in
        Bonn for providing an excellent environment to work together during
        the 2017 Special Hausdorff Program on Applied and Computational
        Algebraic Topology.}


\date{\today}

\begin{document}
\newtheorem{definition}{Definition}[section]
\newtheorem{theorem}[definition]{Theorem}
\newtheorem{proposition}[definition]{Proposition}
\newtheorem{corollary}[definition]{Corollary}
\newtheorem{lemma}[definition]{Lemma}
\newtheorem{remark}[definition]{Remark}
\newtheorem{conjecture}[definition]{Conjecture}
\newtheorem{problem}[definition]{Problem}
%
\begin{abstract}
Combinatorial vector fields on simplicial complexes introduced by
Robin Forman constitute a combinatorial analogue of classical flows.
They have found numerous and varied applications in recent
years. Yet, their formal relationship to classical dynamical systems has been
less clear. In this paper we prove that for every combinatorial
vector field on a finite simplicial complex~${\mathcal X}$ one can construct
a semiflow on the underlying polytope~$X$
which exhibits the same dynamics. The equivalence of the
dynamical behavior is established in the sense of Conley-Morse
graphs and uses a tiling of the topological space~$X$ which makes it
possible to directly construct isolating blocks for all involved
isolated invariant sets based purely on the combinatorial information.
\end{abstract}
\maketitle
%

%
\setcounter{tocdepth}{1}
\vspace*{-3mm}
\tableofcontents
\vspace*{-3mm}

\newpage
\section{Introduction}
\label{sec1}
%
%
%
%
Combinatorial vector fields on CW complexes were introduced in 1998 by
R.\ Forman~\cite{forman:98a} as a tool in the construction 
of a discrete analogue of classical Morse theory. 
Originally needed only in the gradient setting of Morse theory, 
they were further studied as an analogue of a flow 
in \cite{forman:98b} where Forman presented a combinatorial counterpart
of Conley's result ~\cite{conley:78a} on the decomposition 
of a flow into chain recurrent and gradient dynamics. 
Forman's study of the combinatorial analogues of the concepts in dynamics 
covered several further directions \cite{forman:98c, forman:02a,forman:02b,forman:02c,forman:03}.
At the outset, these results seem to have been loosely motivated
by the corresponding classical results. Since then his results 
have been used successfully in
their own right in a number of applications, such as visualization
and computer graphics~\cite{Comic:etal:2011, defloriani:etal:2016,
delgadoFriedrichs:etal:2015, lewiner:etal:04a, sahner:etal:2008},
networks and sensor networks analysis \cite{Dey:etal:2017,
kannan:etal:2019, zhang:goupil:2018,Weber:atal:2017}, homology
computation~\cite{harker:etal:14a, mischaikow:nanda:13a},
astrophysics~\cite{Makarenko:etal:2016, sousbie:2011},
neuroscience~\cite{Dias:etal:2015}, algebra~\cite{joellenbeck:welker:2009}
and computational geometry~\cite{bauer:edelsbrunner:17a}.

A fundamental question arises whether these combinatorial analogues 
may be tied in some formal way to their classical counterparts.
More specifically, there are two interesting, mutually inverse, questions:
\begin{enumerate}
  \item {\bf Flow modeling:} Given a flow on a smooth manifold $M$, can one model
  its dynamics by a combinatorial vector field on a triangulation of $M$ or approximate it 
  in some sense by a sequence of triangulations of $M$ and combinatorial vector fields?
  \item {\bf Flow reconstruction:}  Given a combinatorial vector field on a CW complex,
  does it model the dynamics of a classical flow or semiflow on the polytope of the complex?
\end{enumerate}
Questions of this type are of
inherent interest, as it seems natural to use a combinatorial vector
field or one of its generalizations such as generalized Morse matching \cite{Freij:2009} or
combinatorial multivector field \cite{mrozek:17a,Lipinski:atal:2020},  
as a discretization tool for the rigorous study of differential
equations, see for example~\cite{mischaikow:etal:16a, mrozek:17a},
as well as a  tool to investigate dynamical systems known only from
finite samples~\cite{AlBrMe2015, DJKKLM2017, MiMrReSz1999,
morita:inatsu:kokubu:2019}.

Surprisingly, so far there are few answers to such questions. 
Regarding the flow modeling question Gallais~\cite{Gallais:2010} 
proved in the gradient situation of Morse theory 
that given a smooth manifold $M$ with a Morse function $F:M\to\R$
there is a triangulation of $M$ and a gradient combinatorial vector field
on this triangulation whose critical cells are in one-to-one correspondence 
with critical points of $F$. This was strenghtened by Benedetti~\cite{benedetti:16a}
who proved that the result applies to $r$th barycentric subdivision of 
every PL triangulation with a suitably chosen integer $r$ (see \cite[Section 7.5]{Knudson:2015}
for an overview of these results). 
The  general, non-gradient case seems to be significantly more challenging
and, to our best knowledge, remains untouched.
The problem here is the diversity and complexity of general dynamics for which
a correspondence to finite dynamics is not sufficient and must be replaced
by an approximation scheme. Moreover, as indicated in \cite{mrozek:17a},
a more general concept of combinatorial vector field may be needed.

In this paper we address the equally significant flow reconstruction question. 
The question is particularly important in the dynamical systems known only from finite samples.
We prove that for a given combinatorial vector field~$\cV$ on a simplicial complex~$\cX$ 
there exists a continuous-time dynamical system $\phi$ on the underlying polytope~$X$
such that for every Morse decomposition of the combinatorial vector field $\cV$
there is a Morse decomposition of $\phi$ with the same Conley-Morse graph. 
This result requires some ideas of asymptotic dynamics
excluded from the original Forman's study of combinatorial dynamics
and added to the theory in~\cite{kaczynski:etal:16a,mrozek:17a,batko:etal:20a}, 
in particular the concepts of isolated invariant sets, Morse decompositions and 
their topological invariants: Conley index and Conley-Morse graph.

As we already pointed out,  to achieve useful results on a formal correspondence between 
combinatorial and classical dynamics, combinatorial vector fields are too specific and
a more general concept is needed. 
Combinatorial multivector fields introduced in \cite{mrozek:17a} and 
further generalized and studied in \cite{Lipinski:atal:2020} seem to be a remedy.
For instance, unlike combinatorial vector fields, they may be used to model
such dynamical phenomena as heteroclinic connections or chaotic invariant sets. 
Thus, they constitute a natural candidate to construct an approximation scheme 
for classical dynamics. The results of this paper formulated 
and proved for combinatorial vector fields, directly apply to combinatorial
multivector fields, because every combinatorial multivector field may be subdivided into 
a combinatorial vector field in a way which preserves Morse decomposition. Thus, 
the flow constructed for the combinatorial vector field shares the Conley-Morse graph 
with the combinatorial multivector field. An interesting observation here is that 
the there is no unique way to subdivide which is related to the known phenomenon
of non-uniqueness of connection matrices in classical dynamics.
This will be discussed in \cite{mrozek:wanner:2020}.

The present paper uses some constructions from our earlier work~\cite{kaczynski:etal:16a,batko:etal:20a}  
where we attempt to address the same question but we only manage to provide answers
in fundamentally less satisfactory terms of multivalued dynamical systems with discrete time instead of classical flows. 
Moreover, the results of the present paper may lead towards a clue how to answer the modeling question.
This is because the present results are based on a specific cellular 
tiling of the polytope associated with the combinatorial vector field which, 
on one hand, via transversality conditions on the boundary of the tiles, 
provides a  direct tool to construct isolating blocks for all involved
isolated invariant sets based purely on the combinatorial information,
but, on the other hand, suggests how to construct a combinatorial multivector field modeling
a given differential equation just from the transversality conditions. 
Research in this direction is in progress.

The remainder of this paper is organized as follows. 
Section~\ref{sec:main} provides the description of the main results of the paper
together with examples.
In Section~\ref{sec2} we collect necessary background material
on topology, on simplicial complexes and their representations,
on semiflows and the classical Conley index, as well as on
combinatorial vector fields. Section~\ref{sec3} then demonstrates
that the underlying polytope~$X$ of a given simplicial complex~$\cX$
can be subdivided in a canonical way into tiles, based on the
concept of barycentric coordinates. These tiles form the cell
decomposition used in Theorems~\ref{thm:intro1}--\ref{thm:intro3} and
are the basic building blocks of our semiflow construction.
In addition, we introduce the notion of an admissible semiflow
on~$X$ for a combinatorial vector field~$\cV$, which has to satisfy
certain transversality conditions on the tile boundaries, as well as
the condition of strong admissibility which additionally puts
restrictions on the flow behavior in arrow tiles. After that,
in Section~\ref{sec4} we recall the concepts of isolated invariant
sets and Morse decompositions for a combinatorial vector field~$\cV$,
and we show that any semiflow on~$X$ which is admissible has the same
isolated invariant sets and Conley indexes, while every strongly admissible
semiflow has the same Conley-Morse graph as the combinatorial
vector field~$\cV$, thereby establishing Theorems~\ref{thm:intro1}
and~\ref{thm:intro2}. The construction of a specific strongly
admissible semiflow is the subject of Section~\ref{sec5}, and
this finally implies Theorem~\ref{thm:intro3}. The semiflow
construction follows the design decisions made earlier in this
section.
%
%
\section{Main Results}
\label{sec:main}
We recall that a {\em combinatorial vector field\/}~$\cV$
on a {\em simplicial complex\/}~$\cX$ may be interpreted as
a certain partition of~$\cX$ into subsets of cardinality at
most two (see Sections~\ref{sec:abs-simp-comp} and~\ref{sec:cvf}
for precise definitions). Singletons in~$\cV$ are interpreted
as {\em critical cells}. Doubletons (sets of cardinality two)
are required to consist of an $n$-dimensional simplex~$\sigma\in\cX$
and one of its $(n-1)$-dimensional faces $\tau\subset\sigma$.
A doubleton $\{\tau,\sigma\}\in\cV$ is  interpreted as a
{\em vector\/} or {\em arrow\/} with tail~$\tau$ and head~$\sigma$,
denoted $\tau\to\sigma$.

As an example consider  the two-dimensional simplicial complex
\begin{displaymath}
  \cX := \left\{ A, \; B, \; C, \; D, \; E, \; F, \;
    AB, \; AD, \; BC, \; BD, \; CD, \; DE, \; DF, \;
    ABD, \; BCD  \right\} ,
\end{displaymath}
where for the sake of readability we write $A = \{ A \}$,
$AB = \{ A,B \}$, $ABD = \{ A,B,D \}$, etc.
As a sample combinatorial vector field in $\cX$ we take
\begin{displaymath}
  \cV := \left\{ \{F\}, \{BD\},\{ABD\},
                      \{A,AD\} ,\{B,AB\} ,\{BC,BCD\},
                      \{C,CD\} ,\{D,DF\} ,\{E,DE\}
                      \right\} .
\end{displaymath}
Thus, the critical cells of~$\cV$ are the vertex~$F$, the edge~$BD$ and
the triangle~$ABD$. The arrows are $A \to AD$, $B\to AB$, $BC\to BCD$,
$C\to CD$, $D\to DF$, and $E \to DE$. The simplicial complex~$\cX$ is
visualized in the left image of Figure~\ref{fig:designexample} and the
combinatorial vector field~$\cV$ is visualized in the middle of this figure.
We note that each of the fifteen simplices of~$\cX$ appears exactly once,
either in an arrow or as a critical cell.
\begin{figure}[tb]
  \centering
  \includegraphics[width=0.3\textwidth]{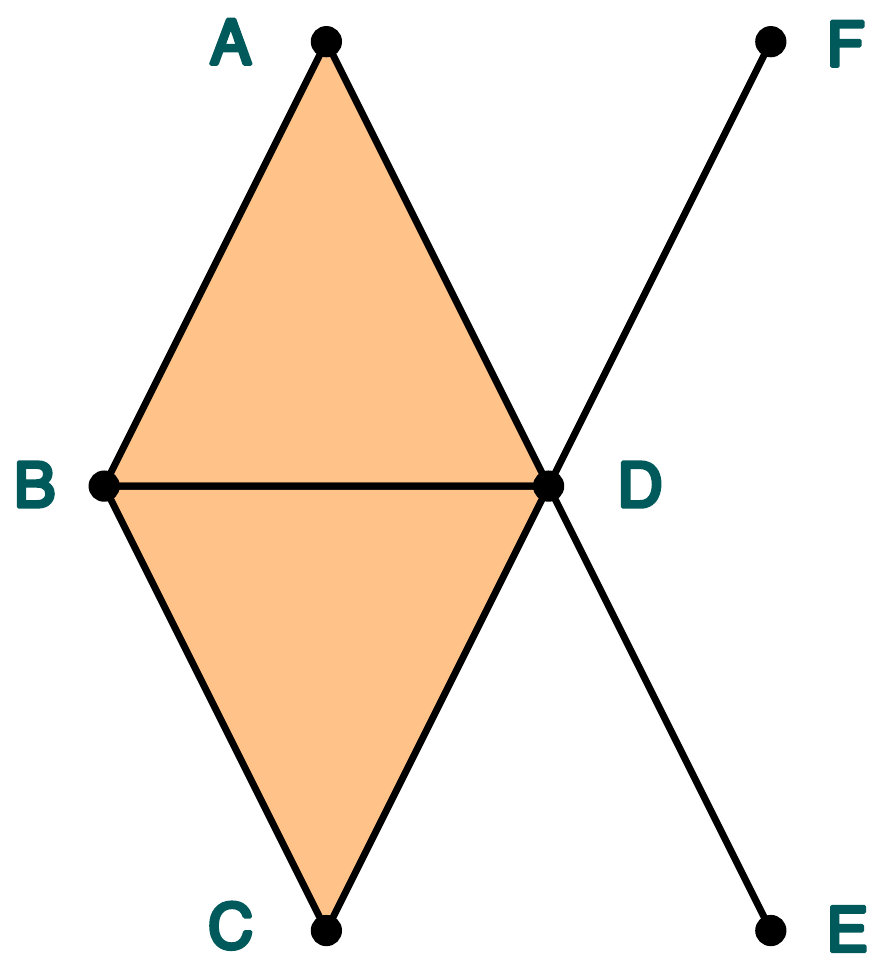}
  \hspace*{0.5cm}
  \includegraphics[width=0.3\textwidth]{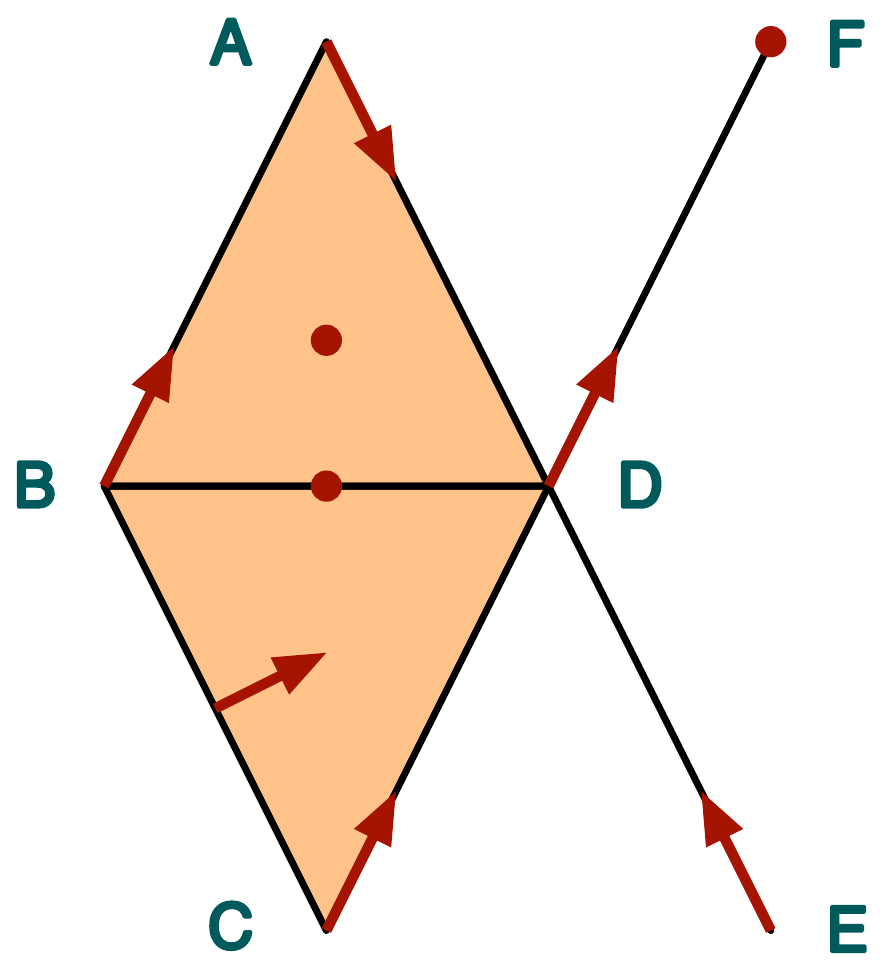}
  \hspace*{0.5cm}
  \includegraphics[width=0.3\textwidth]{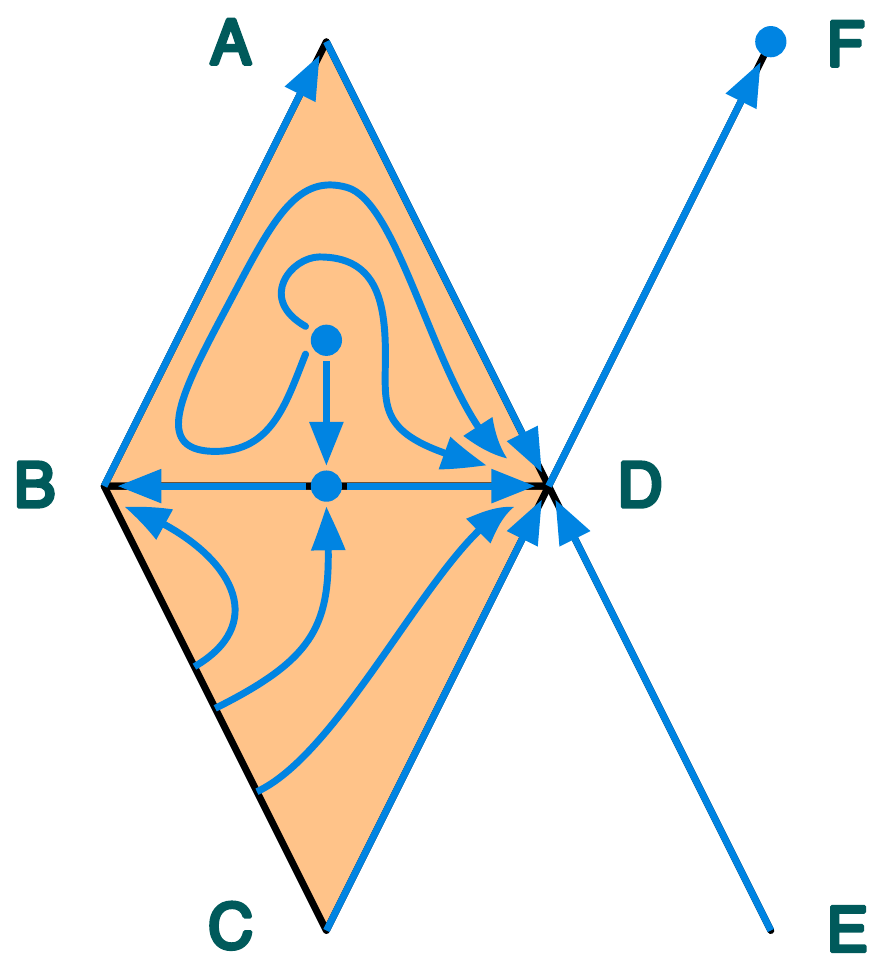}
  \caption{Sample simplicial complex~$\cX$, together with a
           combinatorial vector field~$\cV$ and the induced semiflow.
           The left panel shows a two-dimensional simplicial complex
           consisting of six vertices, seven edges, and two triangles.
           The panel in the middle depicts a combinatorial vector field
           on~$\cX$. Critical cells are indicated as red dots,
           arrows of~$\cV$ are marked in red. Finally, the right-most panel
           sketches an induced semiflow on the geometric representation~$X$ of $\cX$.}
  \label{fig:designexample}
\end{figure}

It has already been pointed out in~\cite{kaczynski:etal:16a}
that for most combinatorial vector fields~$\cV$ on a simplicial
complex~$\cX$ one can intuitively draw a continuous-time
dynamical system on the underlying polytope~$X$ of $\cX$
which mimics the behavior of~$\cV$. For our example, selected
solutions of such a dynamical system are shown in the right-most
panel of Figure~\ref{fig:designexample}. Notice that the three
critical cells of~$\cV$ give rise to three equilibrium solutions.
The Morse index of these stationary states is given by the dimension
of the underlying simplex, which is due to the intuition that
on a critical simplex the flow should move towards its boundary.
Thus, in Figure~\ref{fig:designexample} there is an unstable
equilibrium of index two, which has a local two-dimensional
unstable manifold, as well as an index one equilibrium with a
local one-dimensional unstable manifold. Between these two
stationary states, there exists a heteroclinic solution. Finally,
there is a stable equilibrium at the vertex labelled~$F$, and
almost all solutions of the system converge to this stable equilibrium
in forward time --- except for the two unstable equilibria, the
points on their heteroclinic connection, and the points on a unique
solution which starts on the edge~BC and converges to the index one
equilibrium.

While the construction of a continuous-time dynamical system
is fairly straightforward for small simplicial complexes,
higher-dimensional examples quickly become difficult. It is
therefore necessary to develop a general construction technique
which leads to an easily analyzable dynamical system. While
this is the main subject of the current paper, the example of
Figure~\ref{fig:designexample} demonstrates that there are a
number of {\em design decisions\/} that have to be made first.
\begin{itemize}
\item[(D1)] In a perfect world, we would like to be able to define
a continuous-time dynamical system on the underlying polytope~$X$
of the simplicial complex~$\cX$
through a smooth differential equation.
This should clearly be possible in neighborhoods of the
three equilibrium solutions shown in Figure~\ref{fig:designexample}.
Therefore, in general, our first design goal is to {\em define the
continuous-time dynamical system via smooth differential equations
whenever feasible\/}.
\item[(D2)] A closer look at our example shows that there
generally are solutions of the dynamical system which do not
exist for all negative time. Along the edge~$BC$ or at the
vertex~$E$, the combinatorial vector field points into the
relative interior of the simplex, and therefore the flow should
enter with positive velocity. In other words, our goal has to be
to {\em define a continuous semiflow in which all solutions exist
for all $t \ge 0$, but not necessarily in backward time\/}.
\item[(D3)] Since our  space $X$ is the underlying polytope of
a simplicial complex, the local dimension of~$X$ can change. One
such point is vertex~$D$ in Figure~\ref{fig:designexample}, and
due to the choice of~$\cV$, solutions of the sought-after semiflow
reaching this vertex from the edge~$DE$ or the two triangles
should flow through~$D$ with positive speed and enter the
edge~$DF$. This implies that our goal has to be to {\em allow
for solutions that can merge in finite forward time, and can do
this with jumps in velocity\/}.
\end{itemize}
Based on the above three design decisions, the goal of this
paper is the construction of a continuous semiflow on the
topological space~$X$ which is piece-wise smooth. While our
construction is related in spirit to Filippov systems, see
for example~\cite{dibernardo:etal:08a, filippov:88a}, we
cannot directly apply Filippov's theory in our setting. This
will be described in more detail later on.

In order to construct a semiflow with the above properties
on the underlying polytope $X$ of an arbitrary simplicial complex~$\cX$
and an arbitrary combinatorial vector field~$\cV$ on $\cX$, we
proceed in two steps. As a first step, the space~$X$ is subdivided
into cells, each of which uniquely corresponds to an arrow or a
critical cell of~$\cV$. For the example introduced in
Figure~\ref{fig:designexample}, the associated cell decomposition
is shown in the left panel of Figure~\ref{fig:designcellcomp}.
Notice that most of these cells intersect a number of different
simplices of~$\cX$. We then call a semiflow {\em admissible
for the combinatorial vector field~$\cV$\/} (see
Definition~\ref{def:admissibleflow}), if it is transverse
to the cell boundaries in a certain way. For our example, these
transversality directions are indicated in the right panel of
Figure~\ref{fig:designcellcomp}. While the detailed definitions of
the cell decomposition and the notion of admissibility will be given
later in this paper, they lead to the following first result, whose
precise form is presented in Theorem~\ref{thm:indexequiv}.
\begin{figure}[tb]
  \centering
  \includegraphics[width=0.4\textwidth]{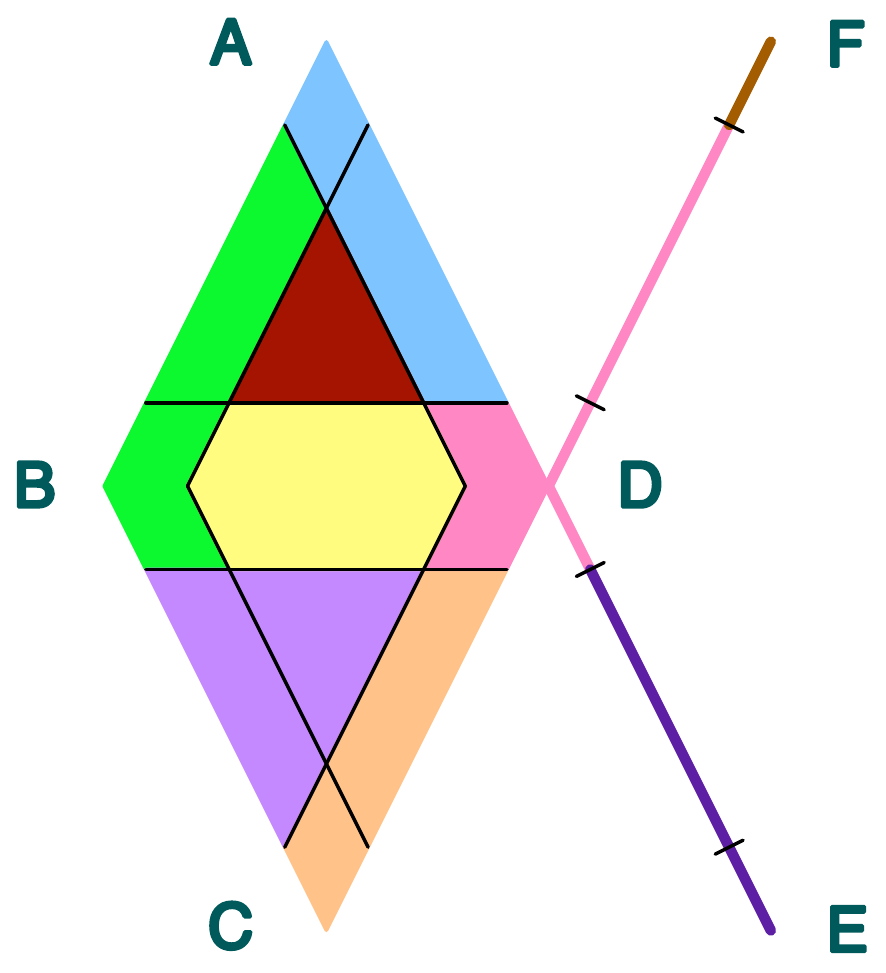}
  \hspace*{2.0cm}
  \includegraphics[width=0.4\textwidth]{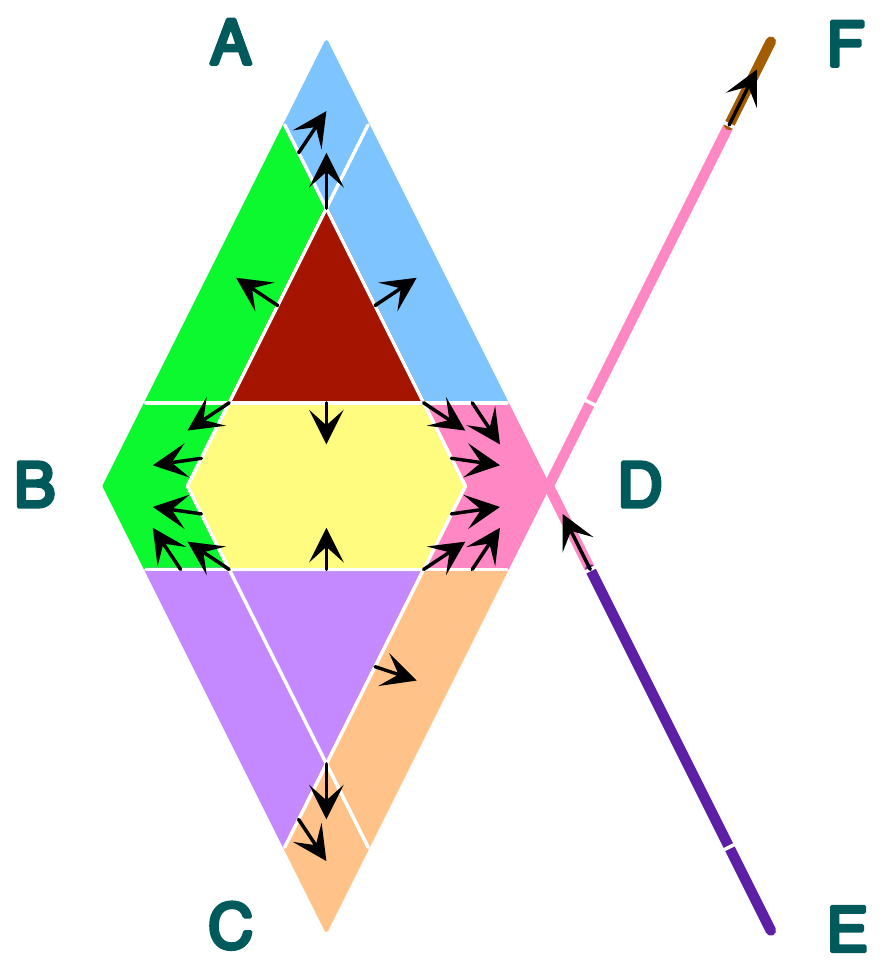}
  \caption{Cell decomposition and flow behavior across cell boundaries
           for the example of Figure~\ref{fig:designexample}. The left
           image shows the nine different cells of~$X$ associated with
           the six arrows and three critical cells of~$\cV$, highlighted
           in different colors. The panel on the right indicates the
           flow behavior across cell boundaries which is induced by~$\cV$,
           and which leads to the notion of admissible semiflow.}
  \label{fig:designcellcomp}
\end{figure}

\begin{theorem}[Admissible Semiflows Inherit Isolated Invariant
                Sets and Conley Indices]
\label{thm:intro1}
Let~$\cX$ denote a simplicial complex, and let~$X$ be
the underlying polytope of $\cX$. Furthermore, let~$\cV$
denote a combinatorial vector field on~$\cX$ in the sense of
Definition~\ref{def:combvfield}. Then there exist
a cell decomposition of~$X$ and $\cV$-induced flow directions
across the boundaries of the cell decomposition such that the
following holds. If~$\phi : \R_0^+ \times X \to X$ denotes a
continuous dynamical system whose flow is transverse to the
boundaries of the cell decomposition and moves in the prescribed
directions, then for every isolated invariant set of the
combinatorial vector field~$\cV$ there exists a corresponding
isolated invariant set for the semiflow~$\phi$ which has the
same Conley index.
\end{theorem}
While prescribing flow directions across the cell decomposition
boundaries is sufficient for establishing the equivalence of
isolated invariant sets and their Conley indices, this condition
is not enough to carry over Morse decompositions, i.e., the global
structure of the dynamics. For this we need to introduce the
notion of {\em strong admissibility\/} (see Definition~\ref{def:admissibleflow}),
which in addition limits
the semiflow behavior on cells associated with arrows of~$\cV$.
This leads to our second result, whose precise form is presented
later in Theorem~\ref{thm:morsedecompequiv}.
\begin{theorem}[Strongly Admissible Semiflows Exhibit the Same Dynamics]
\label{thm:intro2}
Let~$\cX$ denote a simplicial complex, and let~$X$ be
the underlying polytope of $\cX$. Furthermore, let~$\cV$
denote a combinatorial vector field on~$\cX$ in the sense of
Definition~\ref{def:combvfield}. Then there exists a cell
decomposition of~$X$ such that the following holds.
If~$\phi : \R_0^+ \times X \to X$ denotes a strongly admissible
continuous dynamical system (in the sense made precise in
Definition~\ref{def:admissibleflow}), then for every Morse
decomposition of the combinatorial vector field~$\cV$ there
exists a Morse decomposition for~$\phi$ which has the same
Conley-Morse graph.
\end{theorem}
While we are aware that many of the terms used in the formulation
of these two theorems have not yet been introduced, they will be in
the course of this paper. At the moment, these results only serve
to show that under certain transversality and strong admissibility
conditions which are induced by~$\cV$, and with respect to a cell
decomposition of~$X$ which is induced by~$\cX$, continuous semiflows
exhibit the same dynamics as the combinatorial vector field~$\cV$.
Verifying that one can actually construct such dynamical
systems~$\phi$ is the subject of our following final result.
Its precise form can be found in Theorem~\ref{thm:finalsemiflow}.
\begin{theorem}[Existence of Strongly Admissible Semiflows]
\label{thm:intro3}
Let~$\cX$ denote a simplicial complex, and let~$X$ be
the underlying polytope of $\cX$. Furthermore, let~$\cV$
denote a combinatorial vector field on~$\cX$ in the sense of
Definition~\ref{def:combvfield}. Then one can explicitly
construct a continuous semiflow~$\phi$ which satisfies all the
assumptions of Theorems~\ref{thm:intro1} and~\ref{thm:intro2},
and which conforms to our design decisions~(D1) through~(D3).
\end{theorem}
Combined, the above three theorems show that every combinatorial
vector field~$\cV$ on an abstract simplicial complex~$\cX$ gives
rise to a continuous semiflow on the underlying polytope~$X$
of $\cX$ which exhibits the same dynamics in the sense of
Conley theory. For the example in Figure~\ref{fig:designexample}
the semiflow constructed in the above result is shown in
Figure~\ref{fig:designfinalsemiflow}.
\begin{figure}[tb]
  \centering
  \includegraphics[width=0.6\textwidth]{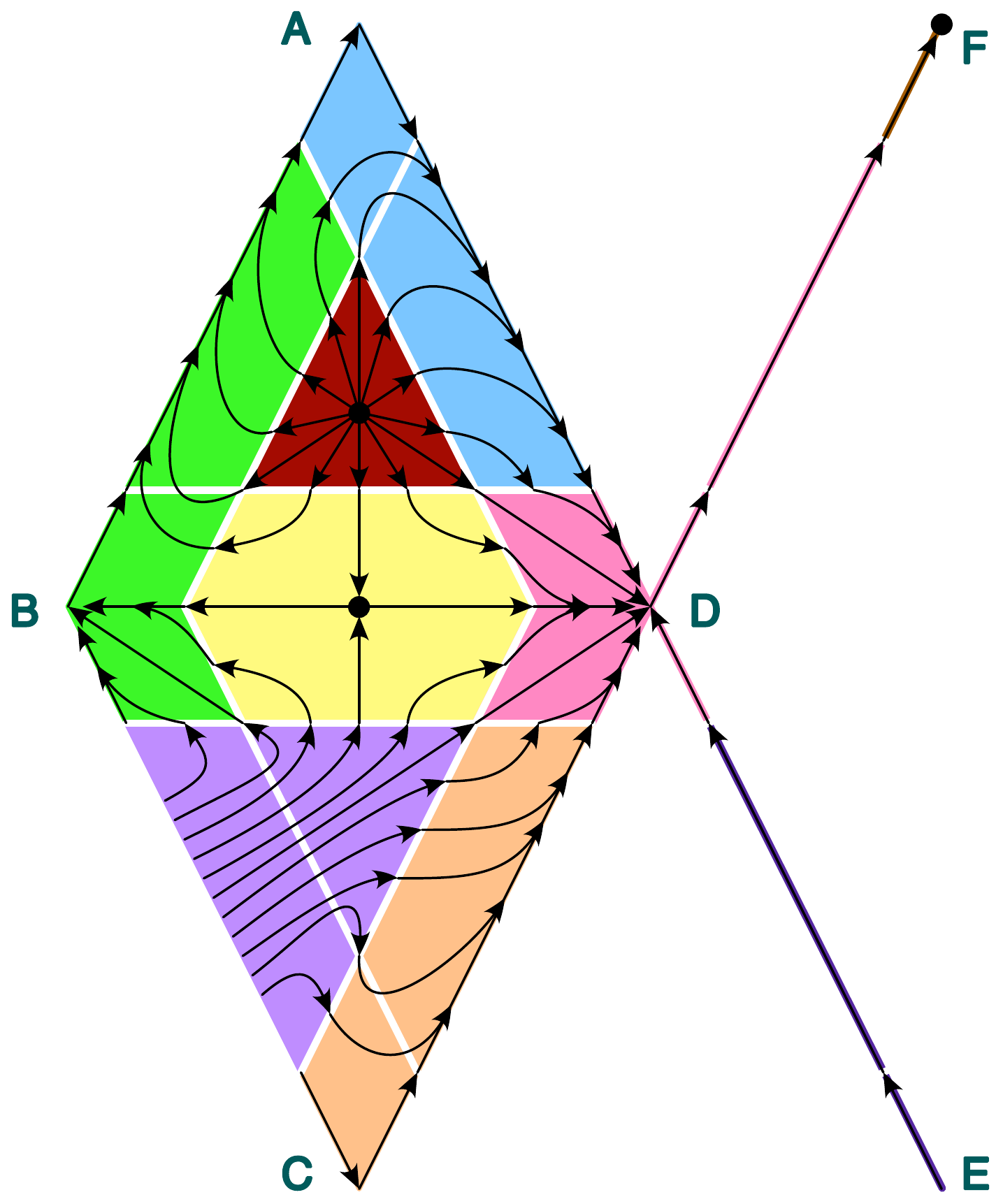}
  \caption{Sketch of a strongly admissible semiflow for the example
           introduced in Figure~\ref{fig:designexample}, as constructed
           in the proof of Theorem~\ref{thm:intro3}.}
  \label{fig:designfinalsemiflow}
\end{figure}

In the special case when the combinatorial vector field is gradient
in the sense of Forman~\cite{forman:98b} one can conclude even more. 
In this case every strongly admissible semiflow, in particular the semiflow
in Theorem~\ref{thm:intro3}, is strongly gradient-like in the sense
of Conley~\cite[Section~II.6.3]{conley:78a}. Moreover, since the combinatorial 
vector field and the strongly admissible semiflow share the same Conley-Morse
graph, the critical cells of the combinatorial vector field of dimension~$n$
are in one-to-one correspondence with the rest points of the semiflow with Morse
index~$n$. We conjecture that this correspondence extends to connection matrices
and Morse complexes, and even inside Morse sets one can see some correspondence
of recurrent, in particular periodic, behavior. This is currently under investigation.
\section{Preliminaries}
\label{sec2}
\subsection{Sets, Maps, and Topology}
\label{sec:sets-and-maps}
We denote the sets of real numbers, strictly negative real numbers,
strictly positive real numbers, non-positive real numbers,
non-negative real numbers, and integers,
respectively, by~$\R$, $\R^-$, $\R^+$, $\R_0^-$, $\R_0^+$, and~$\Z$.
Given a finite set~$X$, we write~$\card X$
for the number of elements of~$X$. We say that a set~$A$ is a {\em doubleton}
if $\card{A}=2$. By a $\Z$-interval we mean a subset~$I$ of~$\Z$ such that
$x,z\in I$ and $x\leq y\leq z$ imply $y\in I$.

We write $\gamma:X\pto Y$ for a partial map from~$X$ to~$Y$, that is, a map
defined on a subset $\dom \gamma\subset X$, called the {\em domain} of~$\gamma$,
and such that the set of values of~$\gamma$, denoted $\im \gamma$, is contained
in~$Y$. For functions $\gamma:I\to Y$ with $I = \dom \gamma\subset \Z$ we use
the sequence-type notation~$\gamma_n$ to denote the value~$\gamma(n)$.

Given a topological space~$X$ and $A\subset X$ we write~$\inte_XA$, $\cl_XA$,
and~$\bd_XA$, respectively, for the interior, the closure, and the boundary
of~$A$ in~$X$. We drop the subscript in this notation if the space $X$ is clear from the context.
By a topological pair we mean a pair $(X,A)$ of topological spaces such that $A\subset X$.
Given a topological pair~$(A,B)$ we write~$A/B$ for the quotient space
with quotient topology and we denote by~$[B]$ the point in~$A/B$ resulting
from collapsing~$B$.

\subsection{Abstract Simplicial Complexes}
\label{sec:abs-simp-comp}
The terminology, notation and conventions we use with respect to simplicial complexes
are based on \cite{hilton:wylie:60a,munkres:84a}. We summarize here the main ideas.

By an {\em abstract  simplicial complex\/}~$\cX$ we mean a finite collection of nonempty, finite
sets such that if $\sigma \in \cX$, then every nonempty subset of $\sigma$
also belongs to~$\cX$. We refer to the elements of~$\cX$ as the {\em simplices\/}
of the simplicial complex.
For every simplex~$\sigma \in \cX$ we define
its {\em dimension\/} as $\dim\sigma := \card{\sigma} - 1$.
We refer to the $0$-, $1$-, and~$2$-dimensional simplices  as
the {\em vertices\/}, {\em edges\/}, and {\em triangles\/}, respectively.
The union of the simplices in $\cX$ is called the {\em vertex set\/} of $\cX$ and denoted~$\cX_0$.
Formally speaking,  a  vertex $v\in\cX_0$ is different from a vertex~$\{v\}$ considered
as a zero-dimensional simplex. But, it will always be clear from the context what we mean by a vertex.
For a simplex~$\sigma \in \cX$, any nonempty subset~$\tau \subset \sigma$
is called a {\em face of~$\sigma$\/}, and in this case~$\sigma$
is referred to as a {\em coface of~$\tau$\/}. The face~$\tau$
is called a {\em proper face\/} if $\tau \neq \sigma$. Furthermore, a
face~$\tau$ of a simplex~$\sigma$ is called a {\em facet
of~$\sigma$\/}, if $\dim\tau = \dim\sigma - 1$. Finally, the {\em dimension\/}
of the simplicial complex is the maximum of the dimensions of all
simplices in~$\cX$, and it is denoted by~$\dim\cX$.
Two abstract simplicial complexes~$\cX$ and~$\cX'$ are called {\em isomorphic}
if there is a bijection $f:\cX_0\to\cX'_0$ such that $\sigma\in\cX$
if and only if $f(\sigma)\in\cX'$.

\subsection{Geometric Simplicial Complexes}
\label{sec:geom-simp-comp}
A {\em geometric $n$-simplex\/} $\sigma$ in $\R^N$ is the convex hull of~$n+1$ affinely
independent points $v_0, v_1, \ldots,v_n \in \R^N$, that is, points such that
the~$n$ vectors~$v_i - v_0$ for $i = 1,\ldots,n$ are linearly independent.
A {\em standard $n$-simplex} is a geometric $n$-simplex $\Delta^n\subset\R^{n+1}$
spanned by all {\em versors}  of $\R^{n+1}$ where the $i$th versor in $\R^{n+1}$
is a vector in $\R^{n+1}$ whose $i$th coordinate is one and all other coordinates are zero.
We  use the abbreviated notation~$\sigma = \langle v_0,v_1, \ldots, v_n \rangle$ to
indicate that $\sigma$ is the geometric simplex spanned by the points~$v_0, v_1,\ldots, v_n$.
The {\em vertex set\/} of $\sigma$ is the set $\setof{v_0, v_1, \ldots, v_n}$.
The elements of this set are  the {\em vertices} of $\sigma$. The number~$n$
is the {\em dimension\/} of~$\sigma$. A {\em face\/} of~$\sigma$ is a geometric
simplex whose vertices constitute a subset of $\setof{v_0,v_1, \ldots, v_n}$.
The concepts of {\em proper face\/}, {\em coface\/}, and {\em facet\/} in the
geometric setting are defined analogously to the abstract case.

We now turn our attention to the representation of points in geometric simplices.
Every point $x \in \sigma = \langle v_0, v_1, \ldots, v_n \rangle$ has a unique
representation of the form
\begin{equation} \label{def:barycentriccoord}
  x = \sum_{i=0}^n t_{v_i}(x) v_i ,
  \quad\text{where}\quad
  \sum_{i=0}^n t_{v_i}(x) = 1
  \quad\text{and}\quad
  t_{v_i}(x) \geq 0 .
\end{equation}
The number~$t_{v_i}(x)$ is called the {\em barycentric coordinate\/}
of~$x$ with respect to the vertex~$v_i$. While this definition
introduces barycentric coordinates via functions~$t_{v_i}$,
we sometimes also make use of the abbreviated notation
\begin{equation} \label{def:barycentriccoord2}
  x_{v_i} := t_{v_i}(x)
  \quad\text{for all}\quad
  x \in \sigma = \langle v_0, v_1, \ldots, v_n \rangle
  \quad\text{and}\quad
  i = 0,\ldots,n ,
\end{equation}
in which we express the barycentric coordinates as actual coordinates
of~$x$. Given a geometric $n$-simplex $\sigma = \langle v_0, v_1, \ldots, v_n \rangle$,
the associated {\em cell\/}~$\stackrel{\circ}{\sigma}$ consists of all points
in~$\sigma$ whose barycentric coordinates are all strictly positive.
Note that for every geometric simplex $\sigma$ we have
\begin{equation}
\label{eq:sigma-via-cells}
    \sigma=\bigcup\setof{\stackrel{\circ}{\tau}\;\mid \text{ $\tau$ a face of $\sigma$}}.
\end{equation}

A {\em geometric simplicial complex\/} in $\R^N$ consists of a finite collection~$\cK$
of geometric simplices in $\R^N$ such that every face of a simplex in~$\cK$ is in~$\cK$,
and the intersection of two simplices in~$\cK$ is their common face.
The {\em underlying polytope\/} or briefly {\em polytope\/} of a geometric simplicial
complex~$\cK$ is the union of all simplices in~$\cK$ considered
as a topological space with the topology inherited from $\R^N$.
The {\em vertex scheme\/} of a geometric simplicial complex~$\cK$
is an abstract simplicial complex whose abstract simplices are
the vertex sets of the geometric simplices in~$\cK$.

\subsection{Subcomplexes and Combinatorial Closures}
\label{sec:subcomplexes}
A {\em subcomplex} of a simplicial complex~$\cX$, abstract or geometric,
is a collection of simplexes in $\cX$ which itself is a simplicial complex.
The {\em combinatorial closure\/} of a subset $\cS$ of a simplicial complex $\cX$
is the set of faces of all dimensions of simplices in~$\cS$.
It is denoted by~$\Cl \cS$.
We say that a collection of simplices~$\cS$ in $\cX$ is {\em combinatorially closed\/},
if $\Cl\cS = \cS$. Hence, the combinatorial closure $\Cl\cS$
is the smallest subcomplex of~$\cX$ which contains~$\cS$,
and the set~$\cS$ is combinatorially closed if and only if it is
a subcomplex of~$\cX$.
For a simplex~$\sigma \in \cX$
we denote the set of all proper faces of~$\sigma$ by~$\Bd\sigma$ and
call it the {\em combinatorial boundary\/} of the simplex.
For any vertex~$v \in \cX_0$,
the {\em star of~$v$\/} consists of all simplices of~$\cX$ which
contain~$v$ as a vertex.

A reader familiar with finite topological spaces will immediately notice
that the phrases ``combinatorial closure'' and ``combinatorially closed''
may be replaced respectively by ``closure'' and ``closed'' with respect
to the Alexandrov topology \cite{alexandrov:37a}
induced on a finite simplicial complex~$\cX$, abstract or geometric,
considered as a poset of simplexes ordered by inclusion \cite[Definition~1.4.10]{barmak:11a}.
Similarly, the star of a vertex is the smallest open set in the Alexandrov topology containing the vertex.
However, we notice that the combinatorial boundary of a simplex need not be its topological boundary
in the Alexandrov topology.

\subsection{Geometric Realizations}
\label{sec:geom-real}
A geometric simplicial complex~$\cK$ is a {\em geometric realization\/}
of an abstract simplicial complex~$\cX$ if the vertex scheme of~$\cK$
is isomorphic to~$\cX$.
One fundamental property of abstract simplicial complexes is the following
theorem, see for example~\cite[Proposition~1.9.3 and Theorem~1.9.5]{hilton:wylie:60a},
\cite[Proposition~1.5.4]{matousek:03a}, or~\cite[Theorem~3.1]{munkres:84a}.
\begin{theorem} \label{thm:simplicial-homology}
Every abstract simplicial complex admits a geometric realization
and any two of its geometric realizations are piecewise linear isomorphic.
In particular, their underlying polytopes are homeomorphic.
\qed
\end{theorem}
Theorem~ \ref{thm:simplicial-homology} states that in topological terms
the geometric realization of an abstract simplicial complex is unique.
Two geometric realizations of the same abstract simplicial complex may differ
geometrically but the underlying polytopes are homeomorphic.
In particular, by the {\em underlying polytope} of an abstract simplicial complex $\cX$
we mean the underlying polytope of any geometric realization of $\cX$.
Note that the underlying polytope of an abstract simplicial complex is unique up to a homeomorphism.

Among the many geometric realizations of an abstract simplicial complex $\cX$
there is one, typically used in the proofs of existence, constructed as follows.
We set $d:=\card\cX_0$, identify the vertices $v_1,v_2,\ldots v_d$ of $\cX_0$ in a fixed order
with versors  $e_1,e_2,\ldots e_d$ of $\R^d$ and take as the geometric realization
of $\cX$ the subcomplex of the standard $(d-1)$-simplex in $\R^d$ consisting of faces
$\langle e_{i_0},e_{i_1}, \ldots, e_{i_n} \rangle$ such that $\{ v_{i_0},v_{i_1}, \ldots, v_{i_n}\}$
is a simplex in $\cX$. We refer to this geometric realization of $\cX$ as the {\em standard} geometric realization.
For convenience, in this paper we typically work with the standard geometric realization.
We emphasize, however, that the results of the paper are purely topological and apply to every geometric realization.

Depending on the context, a simplex~$\sigma$ of~$\cX$
is interpreted as either an  abstract simplex $\setof{v_0,\ldots,v_n}$,
or the corresponding geometric simplex $\langle v_0,\ldots,v_n \rangle\subset \R^d$.
For example, if we write $\sigma\subset \cX_0$, we always mean an abstract
simplex. Whenever for an arbitrary point $x \in \R^d$ we write $x \in \sigma$,
then~$\sigma$ is interpreted as the geometric simplex. If~$v$
is a vertex, writing $v \in \sigma$ makes equal sense in both cases.

Given the underlying polytope $X$ of $\cX$ the barycentric coordinates
introduced earlier can be extended to a well-defined continuous
functions $t_v : X \to [0,1]$ which assigns to each
point~$x \in X$ its barycentric coordinate with respect
to the vertex~$v$, whenever~$x$ belongs to a simplex in the
star of~$v$, and zero otherwise.

With every subset $\cS \subset \cX$
we associate a subset $|\cS|$ of the underlying polytope $X$
given by
\begin{displaymath}
  |\cS| \; := \;
  \bigcup_{\sigma \in \cS} \stackrel{\circ}{\sigma} \;\; \subset \;\;
  X \; = \; |\cX| .
\end{displaymath}
Note that in the case when~$\cS$ is combinatorially closed,
that is, $\cS$ is a subcomplex of $\cX$, we get from \eqref{eq:sigma-via-cells} that
\begin{equation}
\label{eq:geom-rep-subcomplex}
|\cS|=  \bigcup_{\sigma \in \cS}\sigma.
\end{equation}
In particular, if $\cS$ is a subcomplex of $\cX$ then $|\cS|$ is the underlying polytope of $\cS$.

As an example consider~$\cS= \{\{ v \}\} \subset \cX$ consisting of a singleton of a vertex in~$\cX_0$.
Then~$|\cS|= \{ v \}$
is the associated singleton in the underlying polytope $X$.
But, if we have~$\cS = \{\{ v,w \}\} \subset \cX$, then~$|\cS|$ is
the line segment between the points~$v$ and~$w$, but not including
either of the endpoints. Moreover, one can easily see that the
following holds.
\begin{lemma} \label{lem:comb:closed}
Let~$\cX$ denote a simplicial complex, and let~$X = |\cX|$ be the underlying polytope.
Then a subset~$\cS \subset \cX$ is
combinatorially closed if and only if the set~$|\cS|
\subset |\cX|$ is closed.
\qed
\end{lemma}
This simple lemma will be useful for constructing isolating blocks
for admissible semiflows, as it allows us to translate properties
for isolated invariant sets in the combinatorial setting directly
to the classical dynamical systems framework.
\subsection{Representable Sets and Convex Partitions}
\label{sec:cw-complexes}
The construction we present in this paper relies on the class of triangulable topological spaces
to guarantee some properties of homology modules as  explained in Section~\ref{sec:homology}.
We recall that a  topological space is {\em triangulable\/} if it is homeomorphic to the polytope
of a geometric simplicial complex, and a topological pair~$(X,A)$ is {\em triangulable} if~$X$
is homeomorphic to the polytope of a geometric simplicial complex~$\cK$ and~$A$ is homeomorphic
to the polytope of a subcomplex of~$\cK$. To ensure that certain sets and pairs in our
construction are triangulable we need some definitions and results
presented in this section.

Let $C\subset\R^N$ be a convex set.  We recall that the {\em dimension} of~$C$
is the dimension of the affine hull of~$C$ (see \cite[Section~1]{rockafellar:03a})
and~$C$ is {\em relatively open} (see \cite[Section~6]{rockafellar:03a})
if it is open in its affine hull.
A simple argument based on \cite[Theorem~6.1]{rockafellar:03a}
proves the following proposition.
\begin{proposition}
\label{prop:rel-open-convex}
Assume $C\subset\R^N$ is a relatively open, convex set of dimension $n$.
Then the topological pair $(\cl C, \cl C\setminus C)$
is homeomorphic to the pair $(B^n,S^{n-1})$ where $B^n$ denotes the closed unit ball in $\R^n$
and $S^{n-1}$ denotes the boundary of $B^n$.
\qed
\end{proposition}
By a {\em partition\/}~$\cC$ of a set~$X$ we mean a finite family of pairwise disjoint,
non-empty subsets of~$X$ such that $\bigcup\cC=X$.
Given a partition~$\cC$ of~$X$ we say that a subset $A\subset X$ is {\em $\cC$-representable\/}
if $A=\bigcup\cC'$ for a $\cC'\subset\cC$.
We say that a partition~$\cC$ of a compact set~$X$ is {\em topologically closed\/} if~$\cl C$
is $\cC$-representable for every $C\in\cC$.
The following proposition is straightforward.
\begin{proposition}
\label{prop:representable-sets}
Let~$\cC$ be a partition of a set~$X$. Then the union, intersection, and set
difference of two $\cC$-representable sets is $\cC$-representable. Moreover,
if~$X$ is compact and~$\cC$ is topologically closed, then the closure, interior,
and boundary of a  $\cC$-representable set is $\cC$-representable.
\qed
\end{proposition}
By a {\em convex partition\/} of a compact set $X\subset\R^N$ we mean a
topologically closed partition~$\cC$ of~$X$ such that every $C\in\cC$ is a
relatively open, convex set. As an easy consequence of
Proposition~\ref{prop:rel-open-convex} we obtain the following theorem
whose triangulability part follows from \cite[Theorem~2.1,
Theorem~1.7]{lundell:weingram:69a}.
\begin{theorem} \label{thm:convex-partition}
Assume $X\subset\R^N$ admits a convex partition~$\cC$.
Then~$X$ is a regular CW complex whose closed cells are the closures of the elements of~$\cC$.
Moreover, every closed, $\cC$-representable subset of~$X$ is also a regular CW complex and,
in consequence, it is triangulable. The same applies to pairs $(A,B)$ of closed $\cC$-representable subsets of~$X$.
\qed
\end{theorem}
\subsection{The Family of $\cD^d_\epsilon$-Representable Sets}
\label{sec:d-eps-rep}
We will now introduce a class of representable sets needed in this paper.
We first need an auxiliary proposition whose elementary proof is left to the reader.
\begin{proposition}
\label{prop:convex-partitions}
Convex partitions have the following properties:
\begin{itemize}
   \item[(i)] If $\cC$ and $\cC'$ are convex partitions respectively of  compact sets $X\subset\R^N$  and $X'\subset\R^{N'}$
   then
\[
   \cC\btimes\cC':=\setof{C\times C'\mid C\in\cC,\,C'\in\cC'}
\]
   is a convex partition of $X\times X'\subset\R^{N+N'}$.
   \item[(ii)] If $\cC$ is a convex partition of a compact set $X\subset\R^N$  and $A\subset\R^N$
   is a convex set such that $X\cap A$ is compact, then
\[
   \cC\bcap A:=\setof{C\cap A\mid C\in\cC,\, C\cap A\neq\emptyset}
\]
   is a convex partition of $X\cap A$.
   \item[(iii)] If $\cC$ is a convex partition of a compact set $X\subset\R^N$
   and $Y\subset X$ is a $\cC$-representable, closed set, then
\[
   \cC_Y:=\setof{C\in\cC\mid C\subset Y}
\]
   is a convex partition of $Y$.
   \qed
\end{itemize}
\end{proposition}
Given an $\epsilon\in(0,1)$ one easily verifies that
\[
   \cC^1_\epsilon:=\{\{0\},(0,\epsilon),\{\epsilon\},(\epsilon,1),\{1\}\}
\]
is a convex partition of $[0,1]\subset\R$. Using Proposition~\ref{prop:convex-partitions}(i)
and proceeding recursively, we construct a convex partition $\cC^d_\epsilon$ of $[0,1]^d$
from a convex  partition $\cC^{d-1}_\epsilon$ of $[0,1]^{d-1}$
by setting $\cC^d_\epsilon:=\cC^{d-1}_\epsilon\btimes\cC^1_\epsilon$.
Thus, by Proposition~\ref{prop:convex-partitions}(ii)
we also have a convex partition $\cD^d_\epsilon:=\cC^d_\epsilon\bcap\Delta^d$
of the standard $d$-simplex $\Delta^d\subset\R^{d+1}$.

We have the following corollary of Proposition~\ref{prop:representable-sets}
and Theorem~ \ref{thm:convex-partition}.
\begin{corollary}
\label{cor:eps-d-representable-sets}
The family of $\cD^d_\epsilon$-representable sets has the following properties.
\begin{itemize}
  \item[(i)] The family is closed under the set-theoretic operations of union,
  intersection, and difference,
  as well as under the topological operations of closure, interior, and boundary.
  \item[(ii)] Every pair of sets in this family is triangulable.
  \qed
\end{itemize}
\end{corollary}
\subsection{Homology}
\label{sec:homology}
Given an abstract simplicial complex $\cX$ and its subcomplex $\cA$
we denote the simplicial homology of the simplicial pair $(\cX,\cA)$
by ~$H_*(\cA,\cB)$ (see \cite[Section~5]{munkres:84a}).
For a compact metric topological space $X$ and its closed subsets
$B\subset A\subset X$, unless explicitly specified otherwise,
by~$H_*(A,B)$ we mean the singular homology of the pair $(A,B)$
(see~\cite[Section~29]{munkres:84a}).
We reduce this notation to~$H_*(A,b)$ if $B=\{b\}$ is a singleton.
We note that for triangulable pairs singular homology
is isomorphic to Steenrod homology~\cite{massey:78a, steenrod:40a, steenrod:41a},
because they both satisfy the Eilenberg-Steenrod axioms
(see~\cite[Theorem 10.1.c]{eilenberg:steenrod:52a}).
Since Steenrod homology satisfies the so-called strong
excision property~\cite[Axiom 8, Theorem 5]{milnor:95a},
we get the following theorem for singular homology of triangulable pairs.
\begin{theorem}[Strong Excision for Triangulable Pairs]
\label{thm:relative-homeo}
If $f:(X,A)\to (Y,B)$ is a relative homeomorphism of compact, metric, triangulable pairs,
that is, a continuous map $f:X\to Y$  which carries~$A$ into~$B$ and $X\setminus A$
homeomorphically onto $Y\setminus B$, then $f_*:H(X,A)\to H(Y,B)$ is an isomorphism.
\qed
\end{theorem}
In particular, we have $H_*(A,B) \cong H_*(A/B,[B])$, a property useful in the
definition of the Conley index.

We also recall the following theorem which is a straightforward consequence
of the Vietoris-Begle Theorem~\cite{begle:50a, begle:55a} for Steenrod
homology~\cite{anh:84a,volovikov:anh:84a}(see also~\cite{dydak:86a, ferry:18a}),
the Five Lemma, the exactness, and homotopy axioms for homology
and the isomorphism between Steenrod and singular homology for
triangulable pairs.
\begin{theorem}[Relative Vietoris-Begle Theorem for Singular Homology]
\label{thm:vietoris-begle}
Assume that the map $f:(X,A)\to (Y,B)$ is a continuous surjection of
compact, metric, triangulable pairs, satisfying $f^{-1}(B)=A$ and such
that $f^{-1}(y)$ is contractible for every $y\in Y$. Then the induced
map $f_*:H(X,A)\to H(Y,B)$ is an isomorphism.
\qed
\end{theorem}
To ensure triangulability, in this paper we apply singular homology only
to pairs of closed, $\cD^d_\epsilon$-representable sets. Since by
Theorem~\ref{thm:convex-partition} such pairs are triangulable,
the singular and Steenrod homology of such pairs, as well as the
simplicial homology of any triangulation of such a pair, are isomorphic.
This lets us easily switch between singular, Steenrod, and simplicial
homology according to the context and our needs.

We note that a pair of abstract simplicial complexes $(\cX,\cA)$
may also be considered as a pair of finite topological spaces
with Alexandrov topology \cite{alexandrov:37a} induced by the face poset.
Therefore, singular homology of the pair $(\cX,\cA)$ is well-defined.
Also here there is no ambiguity, because, by McCord's Theorem \cite{mccord:66},
the singular homology of the pair $(\cX,\cA)$
is isomorphic to its simplicial homology.
\subsection{Semiflows}
\label{sec:sflows}
Our main tool for establishing the connections between combinatorial
and classical dynamics is based on Conley theory.
Hence, we now recall some basic facts of  this theory,
see for example~\cite{rybakowski:87a} for more details.

Consider a semiflow on a compact
metric space~$X$, i.e., a continuous map $\phi : \R_0^+\times X \to X$
satisfying $\phi(0,x)=x$ for all $x\in X$ and $\phi(s,\phi(t,x))=\phi(s+t,x)$
for all $x\in X$ and $s,t\in\R_0^+$. Let $I\subset\R$ be an interval.
We say that a map $\gamma:I\to X$ is a {\em solution} of $\phi$ if
for any $t\in I$ and any $s\in\R_0^+$ such that $s+t\in I$ we have
$\phi(s,\gamma(t))=\gamma(s+t)$. A solution~$\gamma$ is a {\em full
solution\/} if~$I=\R$ is satisfied.
The {\em $\alpha$- and $\omega$-limit sets} of a full solution $\gamma$ are given respectively by
\begin{displaymath}
  \alpha(\gamma) := \bigcap_{\tau \le 0} \cl\gamma((-\infty,\tau])
  \quad\mbox{ and }\quad
  \omega(\gamma) := \bigcap_{\tau \ge 0} \cl\gamma([\tau,\infty)).
\end{displaymath}

We say that a solution $\gamma:I\to X$
is a solution {\em through $x\in X$} if $0\in I$ and $\gamma(0)=x$.
Given an arbitrary subset~$N$ of~$X$ we define
\begin{eqnarray*}
  \Inv^+(N,\phi) & := & \left\{ x \in N \; \mid \;
    \phi(\R_0^+, x) \subset N \right\} , \\[1ex]
  \Inv^-(N,\phi) & := & \left\{ x \in N \; \mid \;
    \mbox{there exists a solution $\gamma : \R_0^- \to N$ of $\phi$
    through~$x$ in~$N$} \right\} , \\[1ex]
  \Inv(N,\phi) & := & \left\{ x \in N \; \mid \;
    \mbox{there exists a solution $\gamma : \R \to N$ of $\phi$
    through~$x$ in~$N$} \right\} .
\end{eqnarray*}
If the considered semiflow~$\phi$ is clear from context, we simplify this
notation to~$\Inv^- N$, $\Inv^+ N$, and~$\Inv N$, respectively.
One can easily see that $\Inv N = \Inv^+N \cap \Inv^-N$.
We say that a subset~$S \subset X$ is {\em invariant\/},
if we have $\Inv S = S$.

Of fundamental importance for Conley theory is the notion
of isolation. A closed invariant set~$S$ is called an
{\em isolated invariant set\/} if there exists a closed
neighborhood~$N$ of~$S$ such that
\begin{displaymath}
  \Inv N \; = \;
  S \; \subset \;
  \inte N .
\end{displaymath}
In this case, the set~$N$ is called an
{\em isolating neighborhood\/}. Isolating neighborhoods
play an important role in Conley theory, since they allow
one to make assertions about~$S$ by studying the dynamics
of~$\phi$ close to the boundary of~$N$. One approach is
centered around specific isolating neighborhoods called
isolating blocks. To define this notion, consider a closed
subset~$B \subset X$ and let~$x \in \bd B$ be an
arbitrary boundary point. Then~$x$ is called a
\begin{itemize}
\item {\em strict egress point\/}, if for every solution
$\gamma : [\delta_1,\delta_2] \to X$ through~$x$ with
$\delta_1 \le 0 < \delta_2$ there exists a neighborhood~$J$
of~$0$ in~$[\delta_1,\delta_2]$ with $\gamma(t) \not\in B$
for all $t \in J \cap \R^+$, as well as $\gamma(t) \in
\inte B$ for all $t \in J \cap \R^-$,
\item {\em strict ingress point\/}, if for every solution
$\gamma : [\delta_1,\delta_2] \to X$ through~$x$ with
$\delta_1 \le 0 < \delta_2$ there is a neighborhood~$J$
of~$0$ in~$[\delta_1,\delta_2]$ with $\gamma(t) \in
\inte B$ for all $t \in J \cap \R^+$ and
$\gamma(t) \not\in B$ for all $t \in J \cap \R^-$,
\item {\em bounce-off point\/}, if for every solution
$\gamma : [\delta_1,\delta_2] \to X$ through~$x$ with
$\delta_1 \le 0 < \delta_2$ there exists a neighborhood~$J$
of~$0$ in~$[\delta_1,\delta_2]$ with $\gamma(t) \not\in B$
for all $t \in J \setminus \{ 0 \}$,
\end{itemize}
where again~$\R^{\pm}$ denotes the set of all strictly
positive/negative real numbers. The set of all strict
egress, strict ingress, and bounce-off points of~$B$ are
denoted by~$B^e$, $B^i$, and~$B^b$, respectively.
We define the {\em exit set} of $B$ by
\begin{equation} \label{eq:exit-set}
   B^- := B^e \cup B^b.
\end{equation}
Then the closed set~$B \subset X$ is called an {\em isolating
block\/} if we have
\begin{equation} \label{def:classicalisoblock1}
  \bd B = B^e \cup B^i \cup B^b
\end{equation}
and
\begin{equation} \label{def:classicalisoblock2}
  \mbox{ the exit set } B^- \;\mbox{ is closed in~$X$} .
\end{equation}
One can readily see that every compact isolating
block~$B$ is an isolating neighborhood for the invariant
set~$S = \Inv B$, since no full solution in~$B$ can touch
the boundary~$\bd B$ due to the lack of internal
tangencies of solutions at the boundary of~$B$.

Knowledge of an isolating block often suffices to make
statements about the isolated invariant set~$S$, even
if~$S$ is unknown. For this, Conley \cite{conley:78a} defined the {\em homotopy
Conley index of~$S$\/} as the homotopy type of the pointed space
\begin{displaymath}
  h(S) \; := \; \left[B / B^-, \; [B^-] \right] .
\end{displaymath}
The Conley index is well-defined, because, given~$S$, an isolating
block such that $S=\Inv B$ always exists and the homotopy
type is independent of the choice of~$B$ up to homotopy equivalence.
It only depends on the underlying isolated invariant set~$S$.
In this paper we exclusively use the derived {\em homological Conley
index of~$S$\/} defined as the Steenrod homology of the pointed space
\begin{displaymath}
  CH_*(S) \; := \; H_*\left( B / B^-, \; [B^-] \right),
\end{displaymath}
where the Steenrod homology on the right-hand side may be replaced with singular homology
when the pair $(B,B^-)$ is triangulable.
Since we work with compact metric spaces,
the strong excision property of Steenrod homology enables us to rewrite
the Conley index as the relative homology
\begin{displaymath}
  CH_*(S) \; = \; H_*(B,B^-) .
\end{displaymath}
Note that for the computation of the Conley index only one isolating block
is necessary. The homological Conley index is a graded Abelian group, i.e.,
the notation~$CH_*(S)$ denotes a sequence~$( CH_k(S) )_{k=0}^{\infty}$ of Abelian
groups~$CH_k(S)$. The celebrated Wa\.zewski principle can then be stated
as follows: If at least one of the homology groups~$CH_k(S)$ is non-trivial,
then we necessarily have $S \neq \emptyset$.

While the Conley index allows one to study specific isolated invariant sets,
we are frequently interested in a finer decomposition of a given isolated
invariant set into smaller ones. For this, we need to introduce the notion
of Morse decomposition for semiflows.
Let~$S \subset X$ denote an isolated invariant set for the semiflow $\varphi$.
A {\em Morse decomposition\/} of~$S$ (see~\cite{mccord:mischaikow:92a})
is a collection $M = \{ M_p \mid p \in \P \}$
of mutually disjoint isolated invariant sets~$M_p \subset S$
and a strict partial order~$>$ on~$\P$ which satisfy the following property.
For every $x \in S$, we either have $x \in M_p$ for some~$p \in \P$, or there
exist $p,q \in \P$ with $p > q$ and a full solution~$\gamma$ through~$x$
such that the $\alpha$- and $\omega$-limit sets satisfy
\begin{displaymath}
  \alpha(\gamma)  \subset M_p
  \quad\mbox{ and }\quad
  \omega(\gamma)  \subset M_q .
\end{displaymath}
\subsection{Combinatorial Vector Fields}
\label{sec:cvf}
We first recall the original definition of combinatorial vector field by
Forman~\cite{forman:98a}.
\begin{definition}[Combinatorial Vector Field as a Map] \label{def:combvfield-forman}
A combinatorial vector field on a simplicial complex~$\cX$ is a map $\cV : \cX \to \cX \cup\{0\}$
such that
\begin{itemize}
\item[(i)]   if $\cV(\sigma)\neq 0$, then $\sigma$ is a facet of $\cV(\sigma)$,
\item[(ii)]  if $\tau\in\im \cV\setminus\{0\}$, then $\cV(\tau)=0$,
\item[(iii)] for  $\tau\in\cX$ the cardinality of  $\cV^{-1}(\tau)$ is at most one.
\end{itemize}
\end{definition}
In this paper we use the following, equivalent definition.
\begin{definition}[Combinatorial Vector Field as a Partition] \label{def:combvfield}
A {\em combinatorial vector field\/}~$\cV$ on a simplicial complex~$\cX$ is
a partition of~$\cX$ into singletons and doubletons such that each doubleton
consists of a simplex and one of its facets.
\end{definition}
The equivalence of the two definitions is established by the following proposition,
which is straightforward to verify.
\begin{proposition} \label{prop:combvfield}
If~$\cV$ is a combinatorial vector field in the sense of
Definition~\ref{def:combvfield-forman}, then
\begin{displaymath}
  \setof{\{\sigma,\cV(\sigma)\} \, \mid \, \cV(\sigma)\neq 0} \;\cup\;
  \setof{\{\sigma\}\,:\,\cV(\sigma)= 0,\;\cV^{-1}(\sigma) = \emptyset}
\end{displaymath}
is a combinatorial vector field in the sense of Definition~\ref{def:combvfield}.
If~$\cV$ is a combinatorial vector field in the sense of Definition~\ref{def:combvfield},
then
\begin{displaymath}
  \tau\mapsto
  \begin{cases}
    \sigma & \text{ if $\{\tau,\sigma\}\in\cV$ and $\tau$ is a facet of $\sigma$,}\\
    0 & \text{ otherwise,}
  \end{cases}
\end{displaymath}
is a combinatorial vector field in the sense of Definition~\ref{def:combvfield-forman}.
Moreover, the two constructions are mutually inverse.
\qed
\end{proposition}
Assume now that~$\cV$ is a fixed combinatorial vector field on~$\cX$.
We say that a simplex~$\sigma\in\cX$ is a {\em critical cell\/} if $\{\sigma\}\in\cV$.
A doubleton $\{\tau,\sigma\}\in \cV$ is an {\em arrow\/} of $\cV$. The
facet relation in an arrow of~$\cV$ lets us write an arrow in the form
$\tau\to\sigma$ meaning that $\{\tau,\sigma\}\in \cV$ is an arrow and~$\tau$
is a facet of~$\sigma$. If $\tau\to\sigma$ is an arrow of~$\cV$, then we
say that~$\tau$ is the {\em tail\/} of~$\sigma$ and~$\sigma$ is the {\em head\/}
of~$\tau$. We say that $\tau\in\cX$ is a {\em tail\/} if it is the tail of some
$\sigma\in\cX$. We say that $\sigma\in\cX$ is a {\em head\/} if it is the head
of some $\tau\in\cX$. We denote the set of critical cells, tails, and heads
of~$\cV$ by~$\Crit \cV$, $\Tail \cV$, and~$\Head\cV$, respectively. Note that
\begin{displaymath}
  \cX=\Crit\cV\cup\Tail\cV\cup\Head\cV ,
\end{displaymath}
and all of these three sets are mutually disjoint. Finally, for a simplex
$\sigma\in\cX$ we write
\begin{eqnarray}
\label{def:flowtiles1}
     \sigma^-&:=&\begin{cases}
             \tau & \text{ if $\tau$ is a facet of $\sigma$ and $\{\tau,\sigma\}\in\cV$,}\\
             \sigma        &  \text{ otherwise,}
            \end{cases}\\
   \sigma^+&:=&\begin{cases}
             \tau & \text{ if $\sigma$ is a facet of $\tau$ and $\{\tau,\sigma\}\in\cV$,}\\
             \sigma        &  \text{ otherwise.}
            \end{cases}
\end{eqnarray}
Thus, $\sigma$ is a critical cell if $\sigma^-=\sigma^+$, a tail if
$\sigma=\sigma^-\neq\sigma^+$, and a head if $\sigma=\sigma^+\neq\sigma^-$.
We extend this notation to an $\omega\in\cV$ by setting $\omega^-:=\omega^+:=\sigma$
if $\omega$ is a critical cell $\{\sigma\}$ and $\omega^-:=\tau$, $\omega^+:=\sigma$
if $\omega$ is an arrow $\tau\to\sigma$.
\section{Cell Decompositions and Admissible Semiflows}
\label{sec3}
In this section we lay the groundwork for the semiflow extension problem
which was outlined in the introduction. We begin with recalling the cell
decomposition which forms the foundation of our approach and which was
introduced and used in~\cite{batko:etal:20a, kaczynski:etal:16a}. In
addition, we define the notions of admissible and strongly admissible
semiflow.

In this and the following sections
we assume that~$\cX$ is a fixed simplicial complex
and $\cV$ is a fixed combinatorial vector field  on $\cX$.
Furthermore, we suppose that $X := |\cX|$ is the underlying polytope
of the standard geometric realization of $\cX$
and $\epsilon$ is a fixed  constant satisfying
\begin{equation} \label{eq:epsilon}
 0 < \epsilon < \frac{1}{1 + \dim\cX}.
\end{equation}
\subsection{A First Cell Decomposition of the Underlying Polytope}
\label{sec22}
The goal of this paper is the construction of a continuous-time
semiflow which mimics the behavior of the underlying combinatorial
vector field. For example, in the situation shown in
Figure~\ref{fig:designexample} we would like the critical triangle to
correspond to an unstable equilibrium of index two. As we intend to
use Conley theory to formalize the connection between the two
frameworks, it will be necessary to work with {\em isolated invariant
sets\/}, i.e., with invariant sets which in some sense can be separated
from the surrounding dynamics via neighborhoods. Another glance at the
rightmost image in Figure~\ref{fig:designexample} shows that this can
easily be done for the index two equilibrium at the center of the
triangle~$ABD$. However, the index one equilibrium on the edge~$BD$
is another matter. While in the picture one can clearly isolate this
stationary state via a small neighborhood, this neighborhood necessarily
has to cover parts of the adjacent two-dimensional simplices~$ABD$ and~$BCD$.
In other words, relying purely on the decomposition of the polytope~$X$
given by the simplices in~$\cX$ will not be enough to design an easily
implementable construction of isolating neighborhoods in the general
case.

This situation is similar to the one encountered in our previous
papers~\cite{batko:etal:20a, kaczynski:etal:16a}, and it was resolved
by the introduction of a new cell decomposition of the polytope~$X$.
While this decomposition is inherently connected to the simplices
in~$\cX$, it leads to cells which have nontrivial intersection with
all the cofaces of a given simplex. Since we will use the same cell
decomposition as the foundation for our semiflow completion problem,
we recall a few definitions and results from the above-cited papers.
\begin{figure}[tb]
  \centering
  \includegraphics[width=0.55\textwidth]{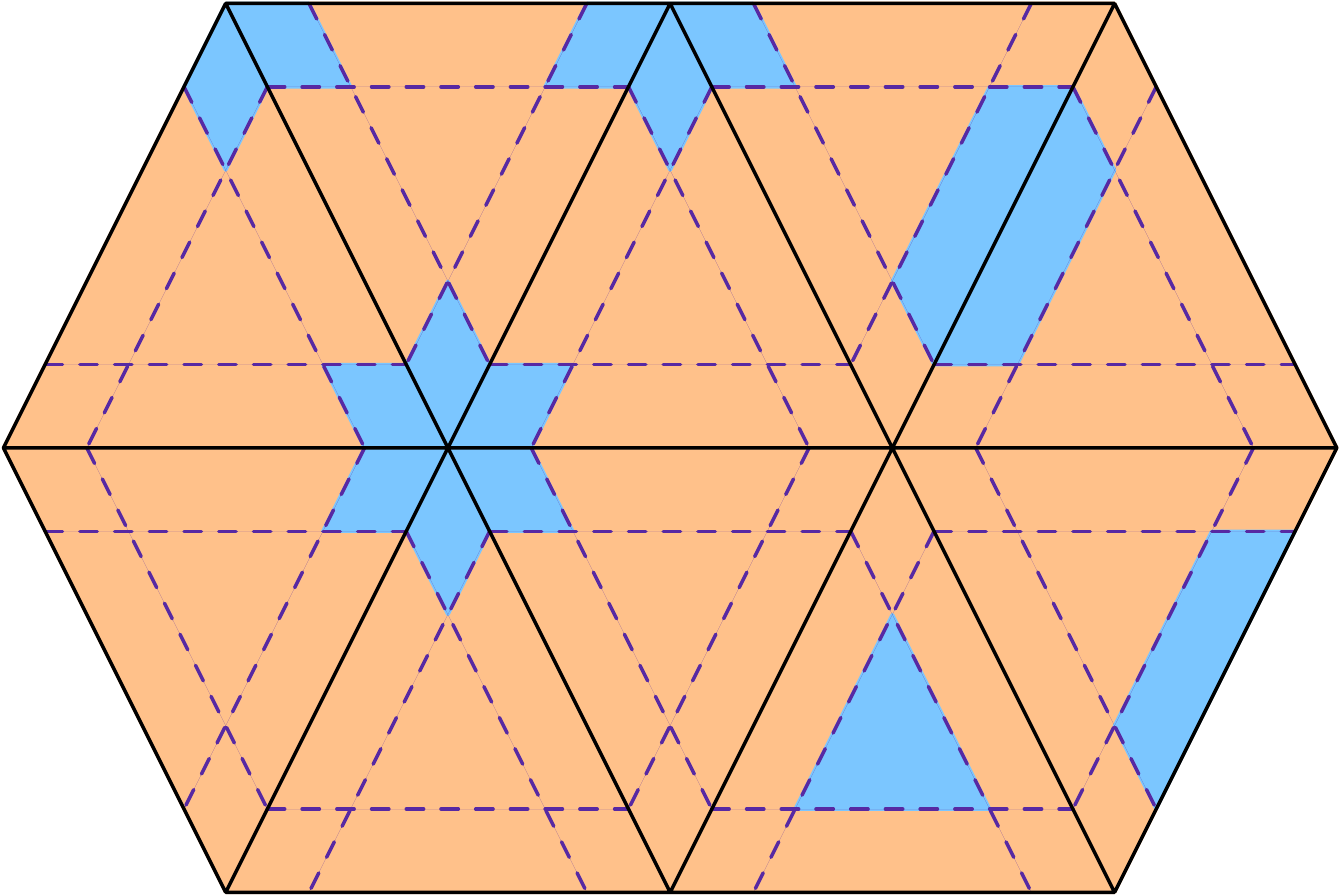}
  \hspace*{1.0cm}
  \includegraphics[width=0.34\textwidth]{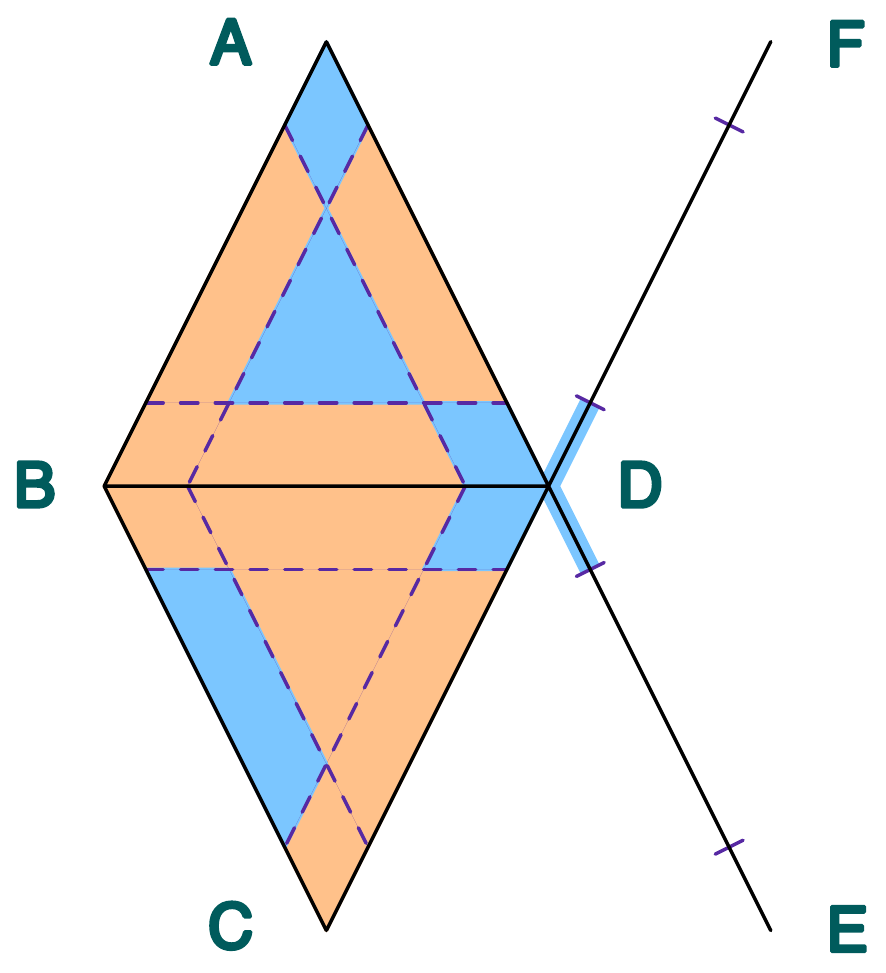}
  \caption{Sample $\epsilon$-cell decomposition boundaries for a
           simplicial complex~$\cX$, as introduced in
           Definition~\ref{def:epsiloncell}. The left panel shows
           a complex~$\cX$ which consists of ten vertices, nineteen
           edges, and ten triangles, and whose polytope~$X$ is
           homeomorphic to a closed disk. The dashed purple lines indicate
           the boundaries of $\epsilon$-cells, and the image shows
           six specific $\epsilon$-cells as blue polygons. The
           three cells in the left half of the diagram correspond
           to vertices in~$\cX$, while the two right-most cells are
           for edges. Finally, the blue triangle on the lower right
           is associated with a $2$-simplex. All of these cells are
           open subsets of~$X = |\cX|$. Similarly, the right panel
           shows the decomposition into $\epsilon$-cells for the
           complex~$\cX$ of Figure~\ref{fig:designexample}. Note in
           particular the $\epsilon$-cell corresponding to the
           vertex~$D$, which reaches into all cofaces, regardless
           of their dimension.}
  \label{fig:celldecomp}
\end{figure}
\begin{definition}[$\epsilon$-Cell Associated with a Simplex]
\label{def:epsiloncell}
For every simplex~$\sigma \in \cX$
its associated {\em $\epsilon$-cell\/} is defined as the set
\begin{displaymath}
  \cse{\sigma} \; := \;
  \left\{ x \in X \, \mid \,
    t_v(x) > \epsilon \;\mbox{ for all }\; v \in \sigma
    \quad\mbox{and}\quad
    t_v(x) < \epsilon \;\mbox{ for all }\; v \notin \sigma
    \right\} \; \subset \; X .
\end{displaymath}
\end{definition}
This definition is illustrated in Figure~\ref{fig:celldecomp}.
The left panel shows a simplicial complex~$\cX$ which consists
of ten vertices, nineteen edges, and ten triangles, and whose
polytope~$X$ is homeomorphic to a closed disk. The dashed purple lines
indicate the boundaries of $\epsilon$-cells, and they consist of
points $x \in X$ which have at least one barycentric coordinate
equal to~$\epsilon$. Six specific $\epsilon$-cells are shown as blue
polygons. The three cells in the left half of the figure correspond
to vertices in~$\cX$, while the two right-most cells are for edges.
Finally, the blue triangle on the lower right is associated with a
$2$-simplex. The right panel of the figure depicts sample
$\epsilon$-cells for the complex of Figure~\ref{fig:designexample}.

All but one of the $\epsilon$-cells in Figure~\ref{fig:designexample} are
homeomorphic to an open Euclidean ball. The cell showing that
in general $\epsilon$-cells need not be homeomorphic to open Euclidean balls
is the blue cell containing the vertex~$D$ in the right panel.
But, $\epsilon$-cells are finite unions of relatively open, convex sets
belonging to the family~$\cD^d_\epsilon$ defined in Section~\ref{sec:d-eps-rep}.
In other words, we have the following proposition.
\begin{proposition}[Representability of $\epsilon$-Cells]
\label{prop:cse-representable}
Every $\epsilon$-cell is $\cD^d_\epsilon$-representable.
\end{proposition}
\begin{proof}
Consider $\sigma,\tau\in\cX$ and an $x\in X$.
Then $x\in\cse{\sigma}\cap\stackrel{\circ}{\tau}$ if and only if
for every $v\in\cX_0$ we have $t_v(x)\in I^{\epsilon,\tau,\sigma}_v$ where
\[
    I^{\epsilon,\tau,\sigma}_v:=\begin{cases}
           \{0\} & \text{ for $v\not\in\tau,\;$ }\\
          (0,\epsilon) & \text{ for $v\in\tau\setminus\sigma,$ }\\
           (\epsilon,1]  & \text{ for $v\in\tau\cap\sigma$.}
        \end{cases}
\]
Hence,  $\cse{\sigma}\cap\stackrel{\circ}{\tau}\neq\emptyset$
implies $\cse{\sigma}\cap\stackrel{\circ}{\tau}\in \cD^d_\epsilon$.
Since we have
\[
   \cse{\sigma}=\bigcup_{\tau\in\cX} \cse{\sigma}\cap\stackrel{\circ}{\tau}
\]
for every $\sigma\in\cX$, the conclusion follows.
\end{proof}
\medskip

The following result states a number of elementary
properties of $\epsilon$-cells which were established
in~\cite[Lemma~4.5]{kaczynski:etal:16a}. In particular, it contains
an explicit characterization of the closures of
$\epsilon$-cells which will be crucial later on.
\begin{lemma}[Properties of $\epsilon$-Cells]
\label{lem:propepscells}
The $\epsilon$-cells introduced in Definition~\ref{def:epsiloncell}
for different simplices in~$\cX$ are disjoint. Moreover, for every
simplex $\sigma \in \cX$ the $\epsilon$-cell~$\cse{\sigma}$ is a
nonempty open subset of the topological space~$X$, and its topological
closure can be characterized as
\begin{displaymath}
  \cl\cse{\sigma} \; = \;
  \left\{ x \in X \, \mid \,
    t_v(x) \ge \epsilon \;\mbox{ for all }\; v \in \sigma
    \quad\mbox{and}\quad
    t_v(x) \le \epsilon \;\mbox{ for all }\; v \notin \sigma
    \right\} .
\end{displaymath}
\qed
\end{lemma}
It is clear from this result, see also Figure~\ref{fig:celldecomp},
that the $\epsilon$-cells provide a decomposition of a certain subset
of~$X$, but not of the whole polytope. However, by considering the
closures of $\epsilon$-cells one can easily show that
\begin{equation} \label{eq:celldecomp}
  X \; = \; |\cX| \; = \;
  \bigcup_{\sigma \in \cX} \cl\cse{\sigma} ,
\end{equation}
which is a cell decomposition of the polytope~$X$ into closed
cells which intersect at most on their boundaries. This cell
decomposition forms the backbone for our semiflow construction,
and it requires us to have a comprehensive understanding and
characterization of how the closures of $\epsilon$-cells
intersect, and which underlying simplices~$\sigma$ lead to
intersections. We therefore recall both the following definition
and the simple result from~\cite[Lemma~4.3]{kaczynski:etal:16a}.
\begin{definition}[$\epsilon$-Characteristic Simplices]
\label{def:charactcells}
Let~$x \in X$ be an arbitrary point in the polytope~$X$. Then the
{\em minimal and maximal $\epsilon$-characteristic simplices of~$x$\/}
are defined by
\begin{eqnarray}
  \sigma^\epsilon_{\min}(x) & := &
    \left\{ v \in \cX_0 \, \mid \, t_v(x) > \epsilon \right\}
    \quad\mbox{ and }
    \label{def:charactcells1} \\[1ex]
  \sigma^\epsilon_{\max}(x) & := &
    \left\{ v \in \cX_0 \, \mid \, t_v(x) \ge \epsilon \right\} ,
    \label{def:charactcells2}
\end{eqnarray}
respectively, and the {\em set of $\epsilon$-characteristic
simplices\/} is defined as
\begin{equation} \label{def:charactcells3}
  \cX^\epsilon(x) \; := \;
  \left\{ \sigma \in \cX \, \mid \,
    t_v(x) \ge \epsilon \;\mbox{ for all }\; v \in \sigma
    \quad\mbox{ and }\quad
    t_v(x) \le \epsilon \;\mbox{ for all }\; v \not\in \sigma
    \right\} .
\end{equation}
We also set
\begin{equation}
\label{eq:sigma-0}
   \sigma^0(x):=\setof{v\in\cX_0\,\mid\, t_v(x)>0}.
\end{equation}
\end{definition}
We note that $\sigma^0(x)$ is the smallest simplex $\sigma\in X$ which
satisfies $x\in\sigma$, and it is also the unique simplex $\sigma\in\cX$
such that $x\in\;\stackrel{\circ}{\sigma}$. The following result follows
from~\cite[Lemma~4.5]{kaczynski:etal:16a}.
\begin{lemma}[Upper Semi-Continuity of the Set of
              $\epsilon$-Characteristic Simplices ]
\label{lem:xepssemicontinuity}
Consider the set of $\epsilon$-characteristic simplices introduced
in Definition~\ref{def:charactcells}. Then for all $x \in X$
we have~$\cX^\epsilon(x) \neq \emptyset$. Moreover, there exists a
neighborhood~$U$ of the point~$x$ such that the inclusion
$\cX^\epsilon(y) \subset \cX^\epsilon(x)$ is satisfied for all
$y \in U$. This may be rephrased by saying that
the mapping $x \mapsto \cX^\epsilon(x)$ is strongly upper
semi-continuous, see~\cite[Definition~3.3]{barmak:etal:20a}.
\qed
\end{lemma}
This result allows us to close the circle and reveal the
connection between $\epsilon$-characteristic simplices and
the notion of $\epsilon$-cells introduced in
Definition~\ref{def:epsiloncell}. To see this, notice that
the strong upper semicontinuity of~$\cX^\epsilon(x)$ with
respect to~$x$ implies that if at a given point~$x \in X$
the set~$\cX^\epsilon(x)$ consists of exactly one simplex~$\sigma$,
then we have to have~$\cX^\epsilon(y) = \{ \sigma \}$ for all
points~$y$ in an open neighborhood of~$x$. In other words, the
subset of~$X$ which consists of points with exactly one
$\epsilon$-characteristic simplex is open. In fact, the
following result based on~\cite{kaczynski:etal:16a} shows that
its connected components are precisely the
$\epsilon$-cells~$\cse{\sigma}$.
\begin{lemma}[Alternative Characterization of $\epsilon$-Cells
              and their Closure]
\label{lem:altcharepscell}
For every simplex~$\sigma \in \cX$
the $\epsilon$-cell from Definition~\ref{def:epsiloncell} can
be characterized as
\begin{equation} \label{lem:altcharepscell1}
  \cse{\sigma} \; = \;
  \left\{ x \in X \, : \;
    \cX^\epsilon(x) = \{ \sigma \} \right\} .
\end{equation}
In addition, the following three statements are pairwise equivalent:
\begin{itemize}
\item[(a)] The simplex~$\sigma$ belongs to~$\cX^\epsilon(x)$,
\item[(b)] the inclusions $\sigma^\epsilon_{\min}(x) \subset \sigma
\subset \sigma^\epsilon_{\max}(x)$ hold,
\item[(c)] the point~$x$ belongs to~$\cl\cse{\sigma}$.
\end{itemize}
In other words, one can characterize the closure of an $\epsilon$-cell
via $\epsilon$-characteristic simplices.
\qed
\end{lemma}
While the first part of the lemma follows easily from the
definitions, the proof of the second part can be found
in~\cite[Corollary~4.6]{kaczynski:etal:16a}.
\subsection{Flow Tiles and Strongly Admissible Semiflows}
\label{sec23}
The decomposition of the polytope~$X$ into the closures
of $\epsilon$-cells as in~(\ref{eq:celldecomp}) is only a
first step in the derivation of a cell decomposition which
can be used in our setting. So far, this decomposition
depends only on the underlying simplicial complex~$\cX$. Yet,
with respect to the fixed combinatorial vector field~$\cV$
on~$\cX$, we will use a slightly coarser decomposition,
which is defined as follows.
\begin{definition}[Flow Tiling for a Combinatorial Vector Field]
\label{def:flowtiles}
The {\em flow tiling associated with~$\cV$\/} is defined as the collection~$\cC$
of compact subsets of~$X$ called {\em flow tiles\/}, which in turn are given by
\begin{equation} \label{def:flowtiles2}
  C_\omega := \cl\cse{\omega^-} \cup \cl\cse{\omega^+}
  \qquad\mbox{ for all }\qquad
  \omega \in \cV ,
\end{equation}
where the $\epsilon$-cells  $\cse{\omega^-}$, $\cse{\omega^+}$
are defined in Definition~\ref{def:epsiloncell}.
\end{definition}
Since either $\omega^-=\omega^+$ or $\omega^-\to\omega^+$ is an arrow, there are two types of flow tiles:
\begin{itemize}
\item For every critical cell~$\sigma \in \cX$ of~$\cV$ the associated flow tile is
$C_{\{\sigma\}}=\cl\cse{\sigma}$.
\item For every arrow~$\tau\to\sigma$ of ~$\cV$
the associated flow tile is ~$C_{\{\tau,\sigma\}}=\cl\cse{\tau} \cup \cl\cse{\sigma}$.
\end{itemize}
We distinguish between these
two types by calling them {\em critical flow tiles\/} and
{\em arrow flow tiles\/}, respectively.
Thus, the flow tiling is obtained from the cell decomposition
in~(\ref{eq:celldecomp}) by simply combining the closures
of $\epsilon$-cells of arrows. This is illustrated in the
left image of Figure~\ref{fig:designcellcomp},
where each flow tile is marked with a different color.
Compare also with the right panel of Figure~\ref{fig:celldecomp}.

We are finally in a position to complete the first
step outlined in the introduction. In the next definition,
we introduce the concept of an admissible semiflow on the
polytope~$X$ for a combinatorial vector field~$\cV$.
\begin{definition}[Admissible and Strongly Admissible Semiflows]
\label{def:admissibleflow}
Consider the
flow tiling~$\cC$ associated with~$\cV$ from
Definition~\ref{def:flowtiles}.
A continuous semiflow $\phi : \R_0^+ \times X \to X$
on the polytope~$X$ is called an {\em admissible semiflow
for~$\cV$\/}, if for every~$x \in X$ which is
contained in at least two flow tiles from~$\cC$, and for
every solution $\gamma : [t_-,t_+] \to X$ of the
semiflow~$\phi$ through~$x$ with $t_- \le 0 < t_+$ there
exists an open neighborhood~$U$\/of~$t = 0$ in~$[t_-,t_+]$
(in the subspace topology) such that
\begin{equation} \label{def:admissibleflow1}
  \begin{array}{ccl}
    \DS \gamma(t) \in \cse{\sigma^\epsilon_{\max}(x)} &
      \mbox{ for all } & t \in U \cap \R^- ,
      \mbox{ and } \\[1.5ex]
    \DS \gamma(t) \in \cse{\sigma^\epsilon_{\min}(x)} &
      \mbox{ for all } & t \in U \cap \R^+ ,
  \end{array}
\end{equation}
where~$\R^{\pm}$ denotes the set of all strictly
positive/negative real numbers. The semiflow is called
a {\em strongly admissible semiflow for~$\cV$\/} if in
addition every forward solution which originates in an
arrow flow tile exits the tile in finite forward time, and
every solution through a point in an arrow flow tile which
exists for all negative times exits the flow tile in
finite backward time.
\end{definition}
While the above definition of admissibility might be
strange at first sight, we can easily illuminate it using
Lemma~\ref{lem:altcharepscell}. Since a point~$x$ lies on
the boundary of at least two flow tiles only if it lies
on the boundaries of at least two $\epsilon$-cells,
this lemma shows that there are at least two simplices
in the set~$\cX^\epsilon(x)$. Moreover, the simplices in this
set are in one-to-one correspondence with the $\epsilon$-cells
$\cse{\sigma}$ which have~$x$ on their boundary. Thus, the
condition in~(\ref{def:admissibleflow1}) requires that the solution
through~$x$ has to come from the $\epsilon$-cell associated with
the largest simplex~$\sigma_1$ in~$\cX^\epsilon(x)$ and has
to move into the $\epsilon$-cell for the smallest
simplex~$\sigma_2$ in~$\cX^\epsilon(x)$. According to
Lemma~\ref{lem:altcharepscell}{\em (b)\/} the
simplex~$\sigma_2$ is a face of~$\sigma_1$, and this means
that an admissible semiflow, when crossing the boundary between tiles,
always flows towards the boundary
of a simplex. Notice also that in general solutions of the
semiflow~$\phi$ are allowed to merge in finite time, and
therefore the condition~(\ref{def:admissibleflow1}) has to
be satisfied for every solution~$\gamma$ which passes through
the point~$x$. In this sense, solutions through flow tile
boundaries exhibit well-defined exit and entrance behavior.

For the combinatorial vector field in
Figure~\ref{fig:designexample} the corresponding
flow tiles and flow directions along the boundaries
between flow tiles are shown in Figure~\ref{fig:designcellcomp}.
Notice, in particular, that the notion of admissibility from
Definition~\ref{def:admissibleflow} does not prescribe any
flow directions in the interior of flow tiles --- not even
on the boundary between the two $\epsilon$-cells which comprise
an arrow flow tile. Only under the assumption of strong
admissibility do we impose restrictions on the semiflow in the
interior of arrow flow tiles.

The main results of the next section show that any admissible
semiflow in the sense of Definition~\ref{def:admissibleflow} exhibits
the same isolated invariant sets as the combinatorial vector
field~$\cV$, while every strongly admissible semiflow exhibits
the same global dynamics in terms of Morse decompositions and
Conley-Morse graphs.

If we take another look at the example from Figure~\ref{fig:designexample}
and the associated flow tiling in Figure~\ref{fig:designcellcomp},
then one can easily see that every critical flow tile~$C \in \cC$
is an isolating block, and the associated Conley index is the
one for an equilibrium whose Morse index is the dimension~$n$ of the
underlying critical simplex. Thus, the Wa\.zewski principle
implies that {\em any admissible semiflow\/}~$\phi$ has a
nontrivial isolated invariant set in~$C$ which on the level
of the Conley index acts like an index~$n$ equilibrium. Note
that in order to verify the properties of an isolating block,
we can easily use the admissibility conditions from
Definition~\ref{def:admissibleflow}. In the next section
we show that the flow tiles associated with~$\cV$ can be used
in a straightforward way to construct isolating blocks for
more complicated isolated invariant sets.
\section{Isolated Invariant Sets and Conley-Morse Graphs}
\label{sec4}
In this section we show that for every combinatorial vector
field~$\cV$ on a simplicial complex~$\cX$ and any associated
strongly admissible semiflow~$\phi$ on~$X = |\cX|$, their dynamics
is equivalent in the sense of Conley theory. For this, we first
recall the combinatorial notions of isolated invariant sets and
Morse decompositions in Section~\ref{sec41}, based on our
results in~\cite{batko:etal:20a, kaczynski:etal:16a}. This
is followed in Section~\ref{sec42} by the explicit construction
of isolating blocks for admissible semiflows~$\phi$ on~$X$ from
the combinatorial information, as well as as the verification in
Section~\ref{sec43} that the associated homological Conley indices
are isomorphic. Finally, in Section~\ref{sec44} we
demonstrate that under the assumption of strong admissibility
every Morse decomposition of~$\cV$ gives rise to a Morse
decomposition for~$\phi$ with isomorphic Conley-Morse graphs.
\subsection{Morse Decompositions for Combinatorial Vector Fields}
\label{sec41}
We begin by recalling the qualitative dynamical theory for
combinatorial vector fields which has been developed
in~\cite{batko:etal:20a, kaczynski:etal:16a}. In its original
form, a combinatorial vector field~$\cV$ on a simplicial
complex~$\cX$ does not create a dynamical system. However,
based on the intuition that we laid out in the previous
sections, one can easily associate with~$\cV$ a suitable
multivalued discrete-time dynamical system on~$\cX$, which
exhibits the following behavior.
\begin{itemize}
\item Critical cells allow for both fixed points
and for flow towards the combinatorial boundary of the
simplex.
\item Arrow tails, i.e., simplices $\sigma\in\Tail\cV$
lead to flow towards the simplex~$\sigma^+ = \cV(\sigma)$.
\item Arrow heads, i.e., simplices $\sigma\in\Head\cV$
always lead to flow towards the boundary of~$\sigma$,
but not towards the face~$\sigma^- = \cV^{-1}(\sigma)$.
\end{itemize}
This behavior can be formalized with the introduction of a
multivalued map $\Pi_\cV : \cX \multimap \cX$ defined by
\begin{equation} \label{def:multimap:pi}
  \Pi_\cV(\sigma) :=
  \begin{cases}
    \Cl\sigma &
      \mbox{ if $\;\;\sigma \in \Crit\cV$} , \\
    \{\cV(\sigma)\} &
      \mbox{ if $\;\;\sigma \in \Tail\cV$} , \\
    \Bd\sigma \setminus\{ \cV^{-1}(\sigma) \} &
      \mbox{ if $\;\;\sigma \in \Head\cV$} .
    \end{cases}%
\end{equation}
Iteration of the multivalued map~$\Pi_\cV$ defines a discrete-time
dynamical system on the simplicial complex~$\cX$ in the usual
way. More precisely, a {\em solution~$\rho$\/}
of the combinatorial vector field~$\cV$ is a partial map~$\rho : \Z \pto \cX$,
where~$\dom\rho$ is a $\Z$-interval, such that
\begin{displaymath}
  \rho_{k+1} \in \Pi_\cV\left( \rho_k \right)
  \qquad\mbox{for all}\qquad
  k, k+1 \in \dom\rho .
\end{displaymath}
As in the classical case, a {\em solution through $\sigma \in \cX$\/} is a solution
such that $\rho_0 = \sigma$ and a {\em full solution} is a solution satisfying $\dom\rho=\Z$.

In the qualitative theory of dynamical systems, solutions themselves
are not the primary target. Rather one concentrates on specific
collections of solutions, which comprise invariant sets.
Borrowing directly from the classical setting, we call a set
$\cS \subset \cX$ an {\em invariant set\/} for the associated
multivalued flow map~$\Pi_\cV$, if for each simplex $\sigma \in \cS$
there exists a full solution~$\rho : \Z \to \cX$ through~$\sigma$
which lies completely in the set~$\cS$. We would like to point out
that in general, there are many solutions of~$\Pi_\cV$ which pass
through a given simplex~$\sigma$, and while some of them might be
full solutions, not all of them have to be. For the notion of
invariance, however, all that matters is the existence of (at
least) one full solution through~$\sigma$ which stays in~$\cS$.
For examples of invariant sets, we refer the reader to the
discussion in~\cite{kaczynski:etal:16a}.

One of the crucial insights of Conley~\cite{conley:78a} is the
observation that general invariant sets are difficult to study.
While in the classical dynamical systems case this is due to their
sensitivity to perturbations, it was pointed out in~\cite{kaczynski:etal:16a}
that even in the combinatorial setting invariance alone is too weak
a concept. This leads to the following definition.
\begin{definition}[Isolated Invariant Set]
\label{def:isolatedinvset}
Let~$\cS \subset \cX$ denote an invariant set
for the multivalued map~$\Pi_\cV$ defined in~(\ref{def:multimap:pi}).
Define the {\em exit set\/} or
{\em mouth\/} of~$\cS$ by
\begin{displaymath}
  \Exit\cS := \Cl\cS \setminus \cS .
\end{displaymath}
Then the invariant set~$\cS$ is called an {\em isolated invariant
set\/}, if the following two conditions are satisfied:
\begin{itemize}
\item[(a)] The mouth of~$\cS$ is combinatorially closed in
the simplicial complex~$\cX$, i.e., for every simplex $\sigma
\in \Exit\cS$ the mouth contains all faces of~$\sigma$.
\item[(b)] There exists no solution $\rho : [-1,1] \cap \Z
\to \cX$ of~$\Pi_\cV$ such that $\rho_{-1},\,\rho_{1} \in \cS$ and $\rho_{0} \in \Exit\cS$.
\end{itemize}
If~$\cS$ is an isolated invariant set, then the combinatorial
closure~$\Cl\cS$ is called an {\em isolating block\/} for~$\cS$.
\end{definition}
The above definition is inspired by the classical notion
of isolating block as introduced in~\cite{conley:78a,
rybakowski:87a}, see also our discussion in
Section~\ref{sec:sflows} concerning isolating blocks.
Considering $\Cl\cS$ as the combinatorial counterpart of the isolating
block for $\cS$, we see that condition~{\it (a)\/} directly corresponds
to condition~\eqref{def:classicalisoblock2} and condition~{\it (b)\/}
is a combinatorial version of the exclusion of internal flow
tangencies implied by condition~\eqref{def:classicalisoblock1}
of the definition of an isolating block.
One can easily see that there are combinatorial
vector fields with invariant sets which are not isolated.
Such examples can be found in~\cite{kaczynski:etal:16a},
and they demonstrate that the two conditions in
Definition~\ref{def:isolatedinvset} are in fact independent.
\begin{figure}[tb]
  \centering
  \includegraphics[width=0.48\textwidth]{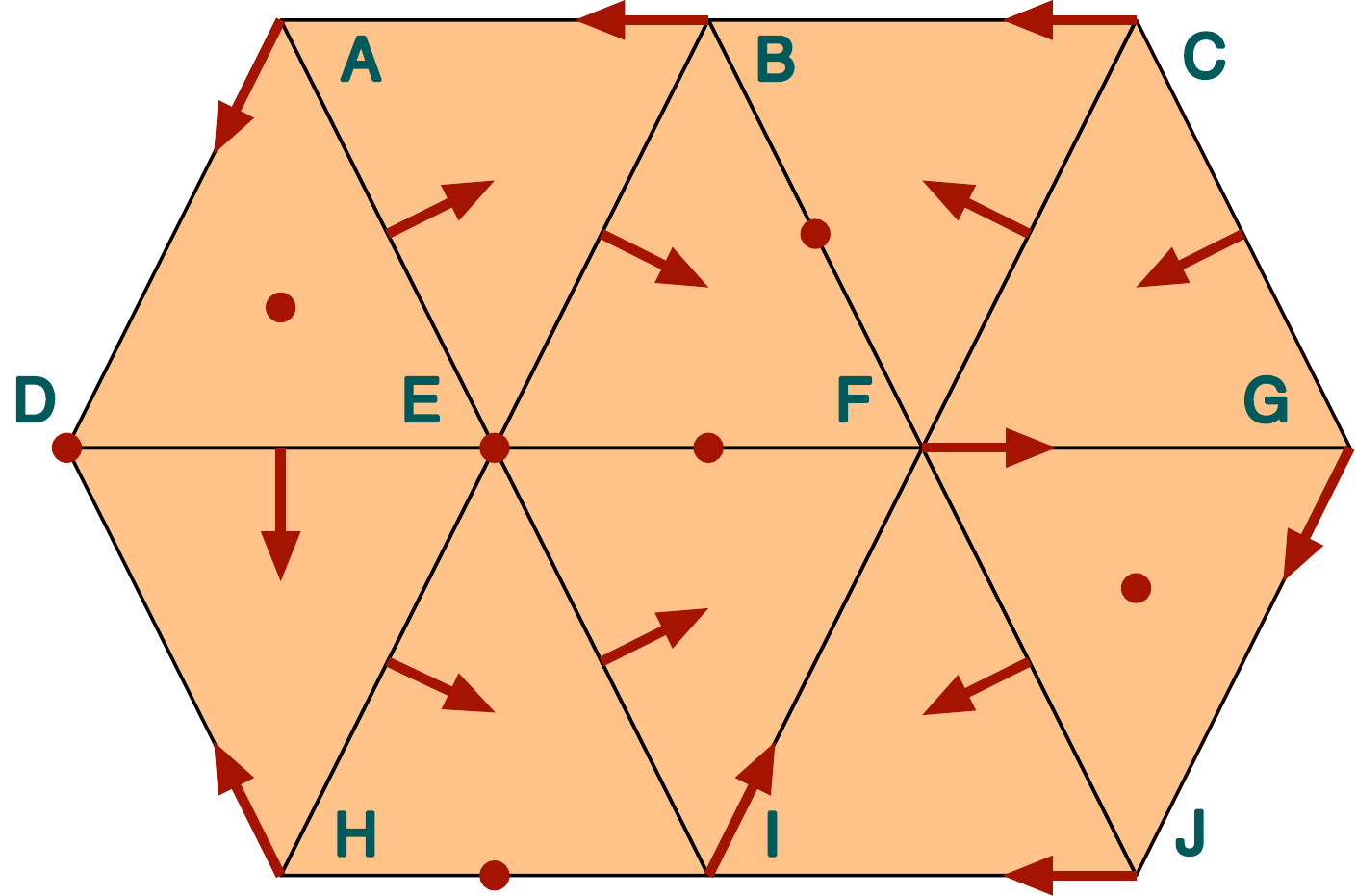}
  \hspace*{0.3cm}
  \includegraphics[width=0.48\textwidth]{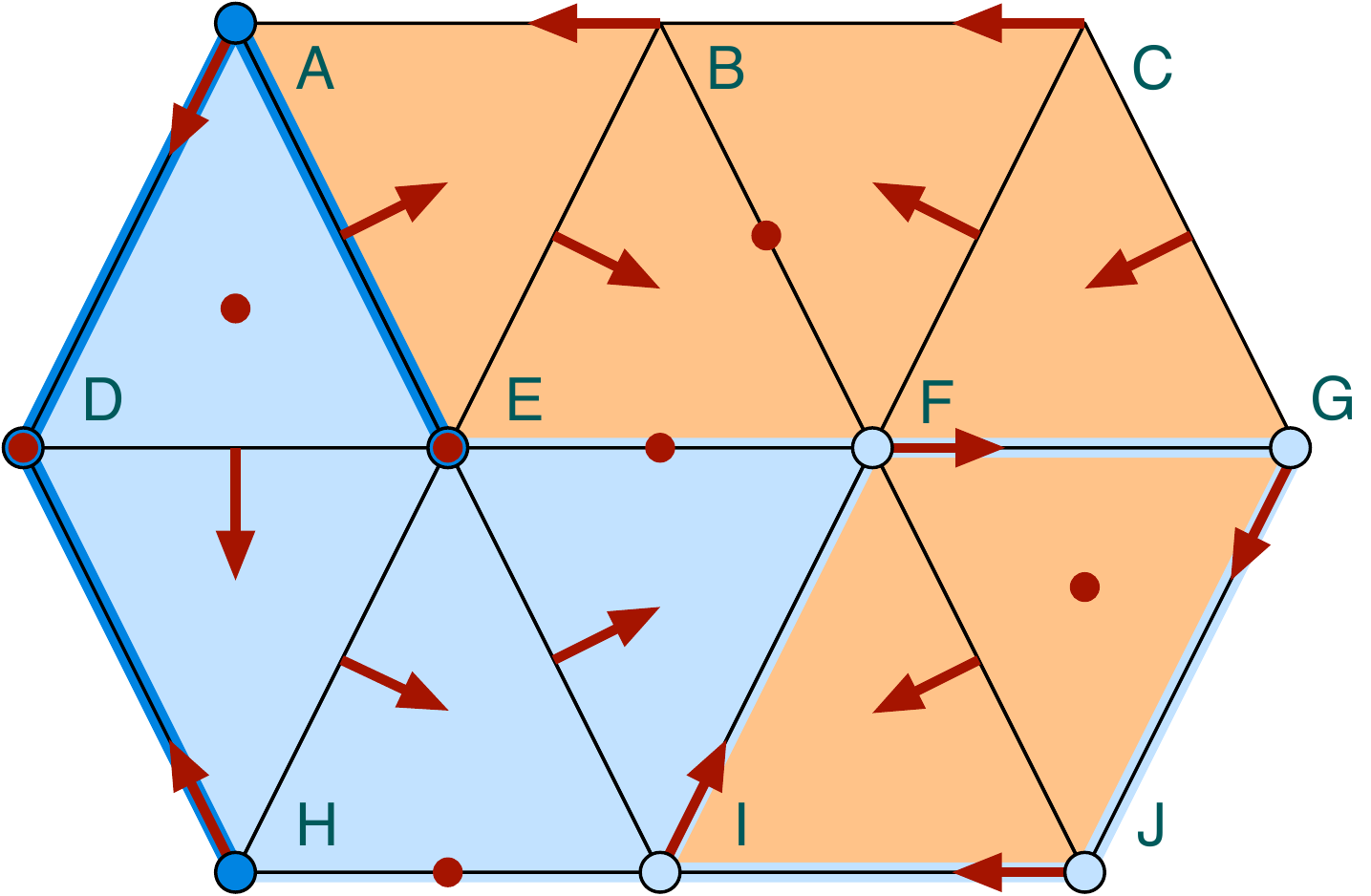}
  \caption{Sample combinatorial vector field with an isolated invariant set.
           The left figure shows a simplicial complex~$\cX$ which
           triangulates a hexagon, together with a combinatorial vector field.
           Critical cells are indicated by red dots, vectors of the vector
           field are shown as red arrows. The right image depicts a sample
           isolated invariant set for this combinatorial vector field. The
           simplices which belong to the isolated invariant set~$\cS$
           are indicated in light blue, and are given by four vertices,
           nine edges, and four triangles. Its mouth~$\Exit\cS$ is
           shown in dark blue, and it consists of four vertices and three
           edges.}
  \label{fig:mainexample}
\end{figure}

From a practical perspective, the above definition of isolated
invariant set is not optimal. While the combinatorial closedness
of the mouth of~$\cS$ can easily be verified in the simplicial
complex~$\cX$, the verification of~{\it (b)\/} necessitates the
use of the multivalued map~$\Pi_\cV$. However, it was shown
in~\cite{kaczynski:etal:16a} that this condition can be reformulated
using the given combinatorial vector field~$\cV$, and this leads
to the following result (see \cite[Proposition~3.7]{kaczynski:etal:16a}).
\begin{lemma}[Characterization of Isolated Invariant Sets]
\label{lem:isolatedinvset}
Let~$\cS \subset \cX$ denote an invariant set
for the multivalued map~$\Pi_\cV$ defined in~(\ref{def:multimap:pi}).
Then~$\cS$ is an isolated invariant set if and only if the
mouth~$\Exit\cS$ is combinatorially closed, and every arrow
of~$\cV$ either lies completely in~$\cS$ or completely outside
of~$\cS$.
\qed
\end{lemma}
The lemma is illustrated in Figure~\ref{fig:mainexample}.
While the left image shows a simplicial complex~$\cX$ which
triangulates a hexagon, together with a combinatorial vector
field~$\cV$, the right panel depicts a sample isolated
invariant set for~$\cV$ in light blue. One can verify
that its mouth is given by the simplices shown in dark
blue, and that the assumptions of Lemma~\ref{lem:isolatedinvset}
are satisfied.

We would like to point out that in contrast to the classical
case, an isolating block in the combinatorial setting does
not determine the associated isolated invariant set. To see
this, take another look at Figure~\ref{fig:mainexample}. Both
sets
\begin{displaymath}
  \cS_1 = \{ EF \}
  \qquad\mbox{ and }\qquad
  \cS_2 = \{ EF, E \}
\end{displaymath}
are isolated invariant sets for~$\cV$ according to
Lemma~\ref{lem:isolatedinvset}, and in both cases we obtain
the same isolating block $\Cl\cS_1 = \Cl\cS_2 = \{ EF, E, F \}$.
Nevertheless, it is still possible to distinguish between the
two isolated invariant sets, and this leads to the notion of
Conley index.
\begin{definition}[Conley Index and Poincar\'e Polynomial]
Let~$\cS \subset \cX$ denote an isolated invariant set
for the multivalued map~$\Pi_\cV$ defined in~(\ref{def:multimap:pi}).
Then the {\em Conley index of~$\cS$\/} is defined as the relative
homology
\begin{displaymath}
  CH_*(\cS) := H_*\left( \Cl\cS, \Exit\cS \right) .
\end{displaymath}
Moreover, the associated {\em Poincar\'e polynomial of~$\cS$\/}
is given by
\begin{displaymath}
  p_{\cS}(t) := \sum_{k=0}^\infty \beta_k(\cS) t^k ,
  \quad\mbox{where}\quad
  \beta_k(\cS) = \mathrm{rank}\, CH_k(\cS) .
\end{displaymath}
\end{definition}
Notice that in the above definition, both sets~$\Cl\cS$
and~$\Exit\cS$ are combinatorially closed, which in the latter
case is due to the fact that~$\cS$ is an isolated invariant set.
This implies that both sets are simplicial subcomplexes of~$\cX$,
and the relative homology in the definition is therefore just
standard simplicial homology.

Returning to our above example, it is now possible to distinguish
between the isolated invariant sets~$\cS_1$ and~$\cS_2$. One can
easily see that on the one hand we have $p_{\cS_1}(t) = t$, and on
the other hand one obtains $p_{\cS_2}(t) = 0$. While the first
Poincar\'e polynomial corresponds in the classical theory to an
equilibrium of index one, the second one corresponds to the index
of an empty set. This is in accordance with our intuition, since
the second case mimics the case of an attractor-repeller pair with
connecting solution, which can disappear through a saddle-node
bifurcation upon perturbation. In the combinatorial setting, the 
index one equilibrium is the open edge~$EF$, and the stable one is
the vertex~$E$; since flow towards the boundary is implicitly assumed
in Forman's setting, there is an automatic connecting flow from~$EF$
to~$E$.

To close this subsection, we now recall how the global dynamics of
a combinatorial vector field can be decomposed. For this, we first
need to introduce some notation. Consider a solution $\rho:
[a,\infty) \cap \Z \to \cX$, which is defined on a $\Z$-interval
which is unbounded to the right. We define the {\em $\omega$-limit
set of~$\rho$\/} as the set
\begin{displaymath}
  \omega(\rho) := \bigcap_{n \ge a} \left\{ \rho_k \; \mid \;
    k \ge n \right\} .
\end{displaymath}
Similarly, for a solution $\rho: (-\infty,a] \cap \Z \to \cX$
of~$\Pi_\cV$ defined on a $\Z$-interval which is unbounded
to the left, we define the {\em $\alpha$-limit set of~$\rho$\/}
via
\begin{displaymath}
  \alpha(\rho) := \bigcap_{n \le a} \left\{ \rho_k \; \mid \;
    k \le n \right\} .
\end{displaymath}
Since the underlying simplicial complex is assumed to be finite,
one can easily see that both the $\alpha$- and the $\omega$-limit
set of a solution are nonempty, whenever they are defined.
\begin{figure}[tb]
  \centering
  \includegraphics[width=0.99\textwidth]{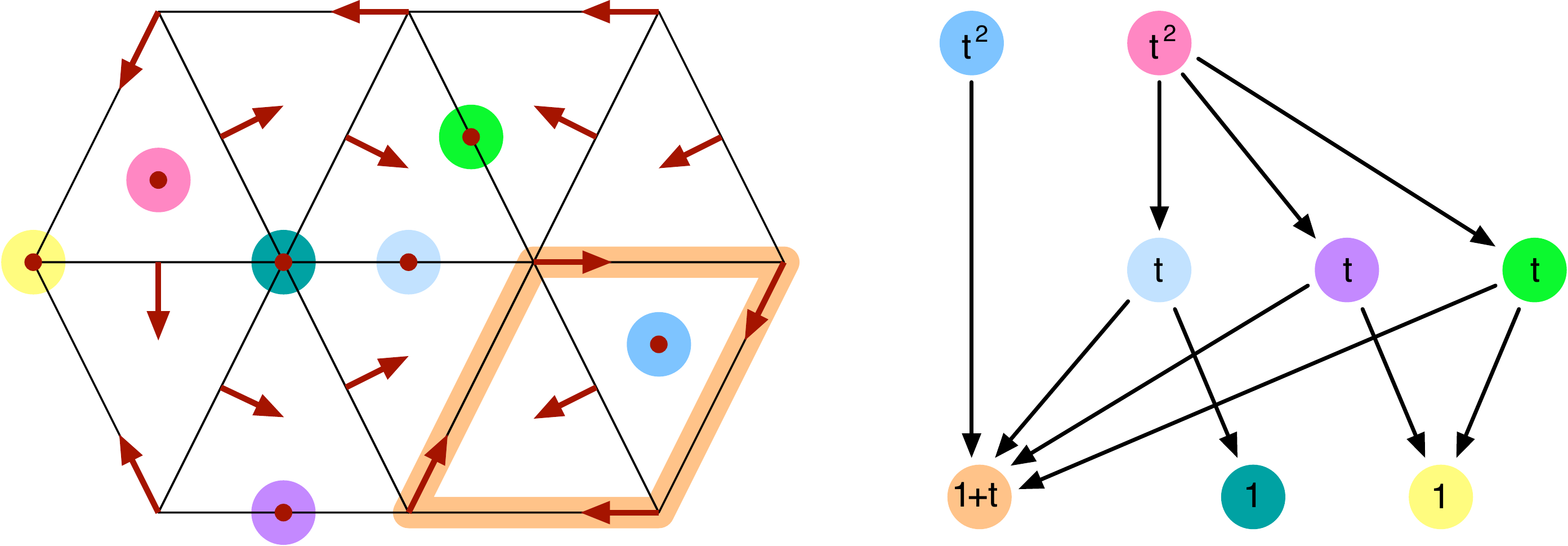}
  \caption{Morse decomposition for the example shown in the left
           panel of Figure~\ref{fig:mainexample}. For this example,
           one can find eight minimal Morse sets, which are
           indicated in the left image in different colors.
           The right image shows the associated Morse graph.}
  \label{fig:morsedecomp}
\end{figure}

After these preparations, we can finally present the notions
of Morse decomposition and Conley-Morse graph which were
introduced in~\cite{batko:etal:20a}.
\begin{definition}[Morse Decomposition and Conley-Morse Graph]
\label{def:morsedecomp}
Let~$\mbbP$ be a poset. The family $\cM = \{\cM_p \, | \, p \in \mbbP\}$
of disjoint, non-empty isolated invariant subsets of~$\cX$ is called a {\em Morse
decomposition of~$\cX$\/ with respect to a combinatorial vector field $\cV$},
if the following three statements hold:
\begin{itemize}
\item[(a)] If~$\rho$ is a solution such that its $\alpha$-limit
set is defined, then $\alpha(\rho) \subset \cM_p$ for some
$p \in \mbbP$. Analogously, this statement also has to hold
for $\omega$-limit sets.
\item[(b)] For every full solution~$\rho$ of the map~$\Pi_\cV$
we have both $\alpha(\rho) \subset \cM_{p}$ and $\omega(\rho)
\subset \cM_{q}$ for some indices $p \ge q$, that is $\rho$
is a {\em connection} from $\cM_{p}$ to $\cM_{q}$.
\item[(c)] If in~(b) we have $p = q$, then the given
full solution~$\rho$ has to satisfy $\im\rho \subset \cM_p$.
\end{itemize}
Finally, the associated {\em Conley-Morse graph\/} is the partial
order induced on~$\cM$ by the existence of connections, and it is
represented as a directed graph labelled with the Conley indices
of the isolated invariant sets in~$\cM$ in terms of their
Poincar\'e polynomials.
\end{definition}
Given a combinatorial vector field~$\cV$ on a simplicial
complex~$\cX$, the strongly connected components of the
multivalued flow map $\Pi_\cV : \cX \multimap \cX$ considered
as a digraph form the unique finest Morse decomposition of~$\cV$
(see~\cite[Theorem 4.1]{DJKKLM2017}). As mentioned earlier,
the Poincar\'e polynomials of the Morse sets can be determined
via simplicial homology. For the example shown in the left panel
of Figure~\ref{fig:mainexample} this procedure leads to the
Morse decomposition depicted in Figure~\ref{fig:morsedecomp}.
\subsection{Construction of Isolating Blocks for Admissible Semiflows}
\label{sec42}
After the preparations of the last section, we can now easily
construct isolating blocks for an admissible semiflow~$\phi$
based on the combinatorial information encoded by~$\cV$.
Let~$\cS \subset \cX$ be an isolated invariant set
for the multivalued map~$\Pi_\cV$ defined in~(\ref{def:multimap:pi}).
Consider the set
\begin{equation} \label{thm:isolatingblocks1}
  B := \bigcup_{\sigma \in \cS} \cl\cse{\sigma}.
\end{equation}
The following lemma is a special case of \cite[Lemma 5.5]{kaczynski:etal:16a}.
\begin{lemma}
If $x \in \bd B \subset X$ then
\begin{equation} \label{thm:isolatingblocks2}
  \cX^\epsilon(x) \cap \cS \neq \emptyset
  \qquad\mbox{ and }\qquad
  \cX^\epsilon(x) \cap \left( \cX \setminus \cS \right)
    \neq \emptyset .
\end{equation}
\qed
\end{lemma}
We then have the following central result.
\begin{proposition}[Isolating Block Construction]
\label{prop:isolatingblocks}
Let~$\cS \subset \cX$ be an isolated invariant set for the
multivalued map~$\Pi_\cV$ defined in~(\ref{def:multimap:pi})
and let~$\phi : \R_0^+ \times X \to X$ denote an arbitrary
admissible semiflow for~$\cV$ in the sense of
Definition~\ref{def:admissibleflow}. Then the set~$B$ given
by~\eqref{thm:isolatingblocks1} is  an isolating block for~$\phi$.
\end{proposition}
\begin{proof}
Since the set~$B$ is a finite union of compact sets it is
clearly compact.
Furthermore, according to Lemma~\ref{lem:isolatedinvset}
and Definition~\ref{def:flowtiles} the set~$B$ is in fact a union
of flow tiles, since arrows of~$\cV$ either lie completely inside
or completely outside of~$\cS$.
Suppose now that $x \in \bd B$ lies on the boundary of
at least two different flow tiles, which immediately implies that~$x$
has to be in the closure of at least two different $\epsilon$-cells.
In combination with Lemma~\ref{lem:altcharepscell} this in turn yields
$\sigma^\epsilon_{\min}(x) \neq \sigma^\epsilon_{\max}(x)$. From the same
lemma, it follows that the set~$\cX^\epsilon(x)$ consists of all
simplices~$\sigma \in \cX$ which satisfy the inclusions
$\sigma^\epsilon_{\min}(x) \subset \sigma \subset \sigma^\epsilon_{\max}(x)$,
and $\sigma \in \cX^\epsilon(x)$ if and only if $x \in \cl\cse{\sigma}$.

Finally, let~$\gamma : [t_-,t_+] \to X$ denote an arbitrary solution
of the semiflow~$\phi$ through the point~$x$ with~$t_- \le 0 < t_+$.
Then, according to Definition~\ref{def:admissibleflow}, there exists
an open neighborhood~$U$\/of~$0$ in~$[t_-,t_+]$ such that we
have the inclusions
\begin{equation} \label{thm:isolatingblocks3}
  \gamma(t) \in \cse{\sigma^\epsilon_{\max}(x)}
    \;\;\mbox{ for }\;\;
    t \in U \cap \R^-
  \quad\;\mbox{ and }\quad\;
  \gamma(t) \in \cse{\sigma^\epsilon_{\min}(x)}
    \;\;\mbox{ for }\;\;
    t \in U \cap \R^+ ,
\end{equation}
i.e., the solution flows from the $\epsilon$-cell
associated with~$\sigma^\epsilon_{\max}(x)$ to the one
for~$\sigma^\epsilon_{\min}(x)$.

Since each of the two characteristic simplices has to be an
element of~$\cS$ or not, we now distinguish four cases.
\begin{itemize}
\item[{\bf (i)}]   $\sigma^\epsilon_{\max}(x) \in \cS$ and
             $\sigma^\epsilon_{\min}(x) \in \cS$:
Due to~(\ref{thm:isolatingblocks2}), there has to be a simplex
$\tau \in \cX^\epsilon(x)$ which satisfies $\tau \notin \cS$.
Together with Lemma~\ref{lem:altcharepscell} this implies that
the simplex~$\tau$ is a face of~$\sigma^\epsilon_{\max}(x)$,
and therefore we have $\tau \in \Exit\cS$. Since
$\sigma^\epsilon_{\min}(x) \subset \tau$ and the mouth
is combinatorially closed, this in turn furnishes
$\sigma^\epsilon_{\min}(x) \in \Exit\cS$, which is a
contradiction. Thus, this case is impossible.
\item[{\bf (ii)}]  $\sigma^\epsilon_{\max}(x) \notin \cS$ and
             $\sigma^\epsilon_{\min}(x) \in \cS$:
The definition of the set~$B$ and~(\ref{thm:isolatingblocks3})
show that in this case, the point~$x$ has to be a strict
ingress point, see Section~\ref{sec23}. Moreover, an argument
similar to the one in the previous case implies that
$\cX^\epsilon(x) \cap \Exit\cS = \emptyset$.
\item[{\bf (iii)}] $\sigma^\epsilon_{\max}(x) \in \cS$ and
             $\sigma^\epsilon_{\min}(x) \notin \cS$:
One can easily see that these assumptions, together with
\eqref{thm:isolatingblocks3}  imply that~$x$ is a strict
egress point. In addition, since~$\sigma^\epsilon_{\min}(x)$
is a face of~$\sigma^\epsilon_{\max}(x)$, it has to be in
the mouth of~$\cS$.
\item[{\bf (iv)}]  $\sigma^\epsilon_{\max}(x) \notin \cS$ and
             $\sigma^\epsilon_{\min}(x) \notin \cS$:
In this final case, the inclusions in~(\ref{thm:isolatingblocks3})
imply that~$x$ is a bounce-off point. Furthermore, the first
inequality in~(\ref{thm:isolatingblocks2}) and
Lemma~\ref{lem:altcharepscell} show that~$\sigma^\epsilon_{\min}(x)$
has to be in the mouth of~$\cS$.
\end{itemize}
The conclusions from these four cases are collected in
Table~\ref{tab:isoblock}.
\begin{table}[tb]
  \centering
  \begin{tabular}{|c|c|c||c|c|}
    \hline
    \multicolumn{3}{ |c||}{Case} & \multicolumn{2}{ |c|}{Implied}\\
   \cline{1-3}
       & $\sigma^\epsilon_{\max}(x)$ &      $\sigma^\epsilon_{\min}(x)$ &
         \multicolumn{2}{ |c|}{Properties} \\
     \hline\hline
    {\bf (i)} & $\in\cS$ & $\in\cS$ &   \multicolumn{2}{ |c|}{Case not possible} \\
    \hline
    {\bf (ii)}  & $\notin\cS$ & $\in\cS$ & $x \in B^i$ &
      $\cX^\epsilon(x) \cap \Exit\cS = \emptyset$ \\
    \hline
    {\bf (iii)} & $\in\cS$ & $\notin\cS$ & $x\in B^e$ &
      $\sigma^\epsilon_{\min}(x)\in \Exit\cS$  \\
     \hline
    {\bf (iv)}  & $\notin\cS$ & $\notin\cS$ & $x\in B^b$ &
      $\sigma^\epsilon_{\min}(x)\in \Exit\cS$  \\
     \hline
  \end{tabular}
  \vspace*{0.5cm}
  \caption{Classification of boundary points of the set~$B$ defined
           in~(\ref{thm:isolatingblocks1}) in terms of the location
           of the simplices~$\sigma^\epsilon_{\max}(x)$
           and~$\sigma^\epsilon_{\min}(x)$ with respect to~$\cS$.
           For a point~$x \in \bd B$ which lies on two different flow
           tiles, the collection~$\cX^\epsilon(x) \subset \cX$ defined
           in~(\ref{def:charactcells3}) has to contain at least two
           simplices, and it is given by all simplices~$\sigma$ which
           satisfy $\sigma^\epsilon_{\min}(x) \subset \sigma \subset
           \sigma^\epsilon_{\max}(x)$.}
  \label{tab:isoblock}
\end{table}

After these preparations the proof of the proposition can readily
be completed. We have to prove that $B$ satisfies properties
\eqref{def:classicalisoblock1} and \eqref{def:classicalisoblock2}.
The four cases above show that the semiflow~$\phi$ does not form
any internal tangencies with~$\bd B$. Therefore, by the admissibility
of~$\phi$ we see that \eqref{def:classicalisoblock1} holds. Moreover,
directly from the definition of strict egress, strict ingress and
bounce-off points we get $B^i \cap (B^e \cup B^b) = \emptyset$.
It follows that $x\in B^i$ is possible only in case~${\bf (ii)}$.

Now let~$x \in \bd B \setminus B^-$ be arbitrary. According
to~$B^- = B^e \cup B^b$ and the last comment, this implies~$x \in B^i$.
By Lemma~\ref{lem:xepssemicontinuity} there exists a neighborhood~$U$
of~$x$ in~$X$ such that the inclusion $\cX^\epsilon(y) \subset
\cX^\epsilon(x)$ holds for all $y \in U$. Moreover, since $x\in B^i$
is possible only in case~{\bf (ii)}, we obtain $\cX^\epsilon(x) \cap
\Exit\cS = \emptyset$. Therefore, we have
\begin{equation}
\label{eq:cX-epsilon-y}
  \cX^\epsilon(y) \cap \Exit\cS = \emptyset
  \quad\mbox{ for all }\quad
  y \in U \cap \bd B .
\end{equation}
Another glance at Table~\ref{tab:isoblock} then shows that
every point~$y \in U \cap \bd B$ has to be in~$B^i$, since
otherwise~$\sigma^\epsilon_{\min}(y) \in \Exit\cS$, which
contradicts~\eqref{eq:cX-epsilon-y}. This finally implies
that~$\bd B \setminus B^- = B^i$ is open in~$\bd B$,
i.e., the set~$B^-$ is closed. Thus, also \eqref{def:classicalisoblock2}
is satisfied. Therefore, $B$ is indeed an isolating block.
\end{proof}
\begin{figure}[tb]
  \centering
  \includegraphics[height=5.0cm]{isolatedinvset.pdf}
  \hspace*{0.2cm}
  \includegraphics[height=5.0cm]{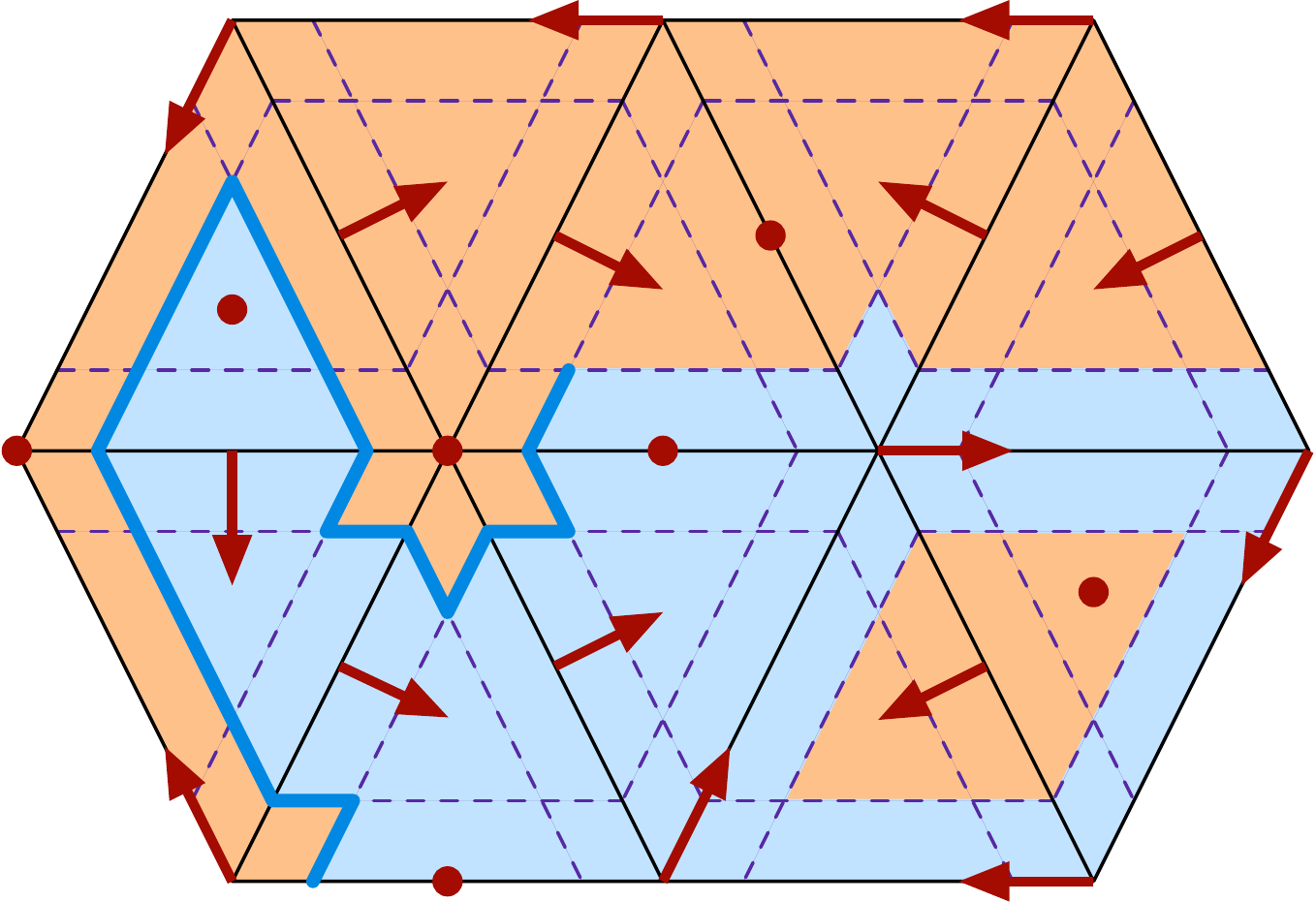}
  \caption{Sample construction of an isolating block. The left panel shows
           the isolated invariant set~$\cS$ from Figure~\ref{fig:mainexample}
           in light blue, while its mouth~$\Exit\cS$ is indicated in
           dark blue. The panel on the right depicts the associated
           isolating block~$B$ constructed in Proposition~\ref{prop:isolatingblocks}
           in light blue, with its mouth~$B^-$ shown in medium
           dark blue.}
  \label{fig:isoblock}
\end{figure}
\medskip

The proposition is illustrated in Figure~\ref{fig:isoblock}.
The left panel reproduces the isolated invariant set~$\cS$ from
Figure~\ref{fig:mainexample} in light blue, with its mouth~$\Exit\cS$
indicated in dark blue.
The panel on the right
depicts in light blue the associated isolating block~$B$ constructed in
Proposition~\ref{prop:isolatingblocks}. The
corresponding mouth~$B^-$ is shown in medium dark blue.
\subsection{Admissibility and the Equivalence of Conley Indices}
\label{sec43}
The results of the last section, particularly
Proposition~\ref{prop:isolatingblocks}, show that for every isolated
invariant set~$\cS$ in the combinatorial setting we can construct
an isolating block~$B$ which isolates an isolated invariant set
$S_\phi=\Inv(B,\phi)$ for every admissible semiflow $\phi$.
Therefore, it is natural to wonder whether we have
\begin{equation}
\label{eq:classical-vs-combinatorial-index}
  H_*(\Cl\cS, \Exit\cS) \cong H_*(B, B^-) ,
\end{equation}
i.e., whether the Conley index of the combinatorial isolated
invariant set~$\cS$ computed via simplicial homology is isomorphic
to the singular homology of the index pair~$(B,B^-)$, which in
turn is the Conley index of the isolated invariant set $S_\phi$.
In the remainder of this section, we will show
that \eqref{eq:classical-vs-combinatorial-index} holds.
\begin{figure}[tb]
  \centering
  \includegraphics[height=5.0cm]{isolatingblock.pdf}
  \hspace*{0.5cm}
  \includegraphics[height=5.0cm]{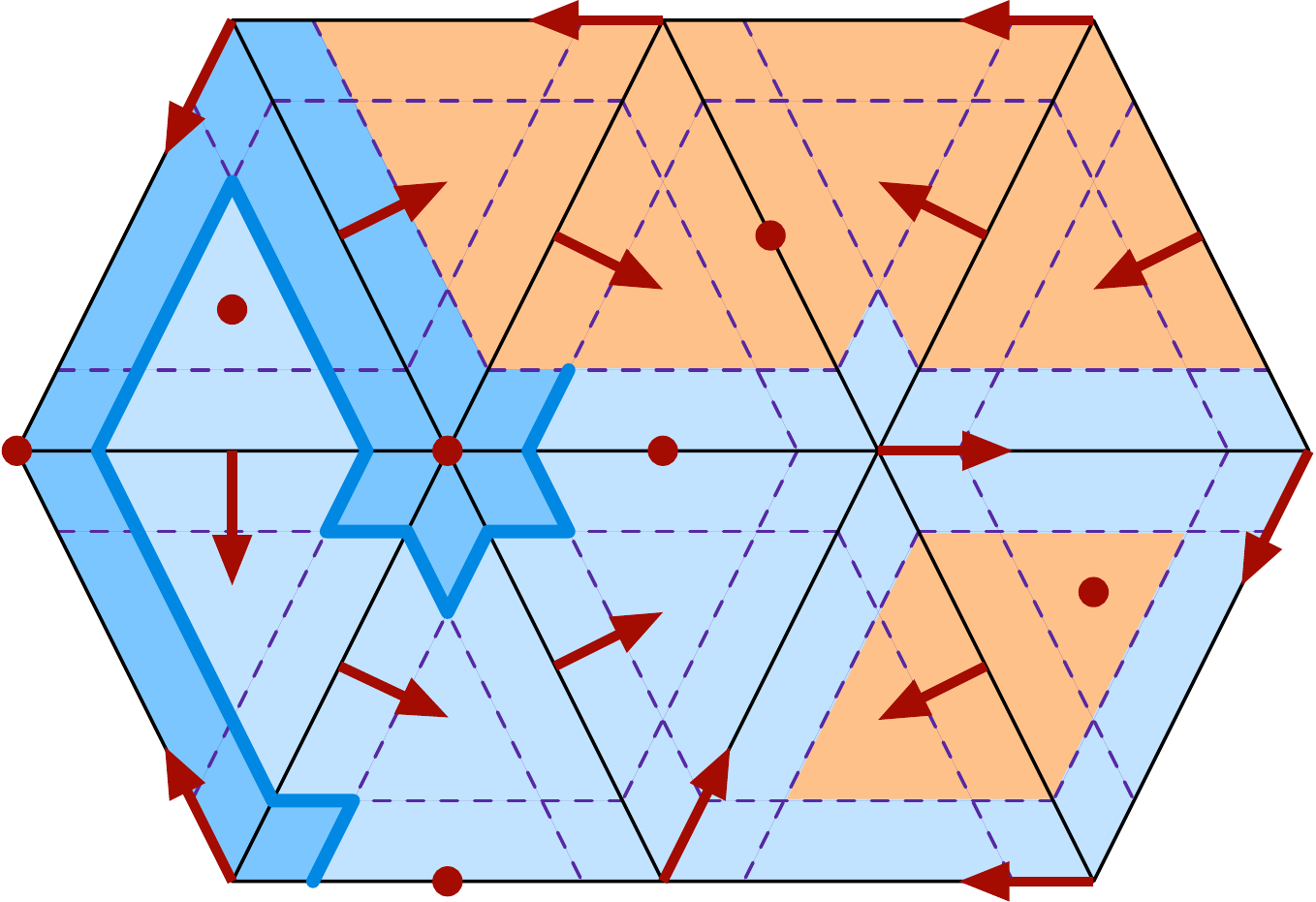}
  \caption{Auxiliary pairs for the Conley index equivalence proof.
           The left panel shows the pair~$(P_1,P_2)$ defined
           in~(\ref{conleyindexequiv2}) for the isolated invariant
           set~$\cS$ from Figure~\ref{fig:mainexample}. While~$P_2$
           is indicated in dark blue, the set~$P_1$ comprises all
           points colored in any shade of blue. In the right panel
           we illustrate the auxiliary pair~$(Q_1,Q_2)$ defined
           in~(\ref{conleyindexequiv3}). The set~$Q_2$ consists
           of points in medium and dark blue, while~$Q_1$ contains
           all points colored in any shade of blue. Notice that we
           clearly have~$(P_1,P_2) \subset (Q_1,Q_2)$ as pairs.}
  \label{fig:equivproof}
\end{figure}

In order to verify that the Conley indices in the combinatorial and
the classical settings are isomorphic, we adapt arguments from our
recent work~\cite{batko:etal:20a}. For any subset~$\cA \subset \cX$
of simplices and an~$\epsilon > 0$ we define the compact set
\begin{equation} \label{conleyindexequiv1}
  N_\epsilon(\cA) :=
  \bigcup_{\sigma \in \cA} \cl\cse{\sigma}
  \subset X .
\end{equation}
Now let~$\cS \subset \cX$ denote an isolated invariant set for the
combinatorial vector field~$\cV$. We use the topological pairs
in~$X$ given by
\begin{equation} \label{conleyindexequiv2}
  P_1 := N_\epsilon(\cS)
  \quad\mbox{ and }\quad
  P_2 := N_\epsilon(\Exit\cS) \cap \bd N_\epsilon(\cS) ,
\end{equation}
as well as
\begin{equation} \label{conleyindexequiv3}
  Q_1 := N_\epsilon(\Cl\cS)
  \quad\mbox{ and }\quad
  Q_2 := N_\epsilon(\Exit\cS) .
\end{equation}
They are illustrated in Figure~\ref{fig:equivproof}.
Using these two auxiliary pairs the Conley index equivalence will be
established in several steps. We begin with a simple lemma.
\begin{lemma} \label{lem:indexequiv1}
Let~$\cS \subset \cX$ denote an isolated invariant set for the multivalued
map~$\Pi_\cV$ defined in~(\ref{def:multimap:pi}), and consider the topological pair
$(B,B^-)$ given by~ \eqref{thm:isolatingblocks1} and~\eqref{eq:exit-set}
and the two topological pairs~$(P_1,P_2)$ and~$(Q_1,Q_2)$ as defined
in~(\ref{conleyindexequiv2}) and~(\ref{conleyindexequiv3}),
respectively. Then $(B,B^-)$, $(P_1, P_2)$, $(Q_1, Q_2)$ are pairs of
$\cD^d_\epsilon$-representable sets and we have
\begin{equation} \label{lem:indexequiv1a}
  H_*(B, B^-) = H_*(P_1, P_2) \cong H_*(Q_1, Q_2).
\end{equation}
\end{lemma}
\begin{proof}
By Proposition \ref{prop:cse-representable}, the sets~$\cse{\sigma}$ are
$\cD^d_\epsilon$-representable. By Corollary~\ref{cor:eps-d-representable-sets}(i),
the closures~$\cl\cse{\sigma}$ are $\cD^d_\epsilon$-representable. Therefore,
again by Corollary~\ref{cor:eps-d-representable-sets}(i), the sets~$P_1$, $Q_1$,
and~$Q_2$ are $\cD^d_\epsilon$-representable as unions of $\cD^d_\epsilon$-representable
sets, and~$P_2$ is $\cD^d_\epsilon$-representable as the intersection of a
$\cD^d_\epsilon$-representable set with the boundary of another
$\cD^d_\epsilon$-representable set. Moreover, we clearly have $P_1 = B$,
and Proposition~\ref{prop:isolatingblocks}, in view of Table~\ref{tab:isoblock},
shows that also $B^- = B^e \cup B^b=P_2$. Hence, both~$B$ and~$B^-$ are
$\cD^d_\epsilon$-representable as well, and the first equality
in~\eqref{lem:indexequiv1a} holds trivially.

Finally, the definitions of~$(P_1, P_2)$ and~$(Q_1, Q_2)$
in~(\ref{conleyindexequiv2}) and~(\ref{conleyindexequiv3}), respectively,
immediately give $P_1 \subset Q_1$ and $P_2 \subset Q_2$, as well
as both $P_1 \setminus P_2 = Q_1 \setminus Q_2$. The result now follows from
an application of Theorem~\ref{thm:relative-homeo}, i.e., the strong excision
property of singular homology for representable pairs.
\end{proof}
\medskip

The following second step shows that the enlarged pair~$(Q_1,Q_2)$
has the same homology as the combinatorial isolated invariant
set~$\cS$. It makes use of the Vietoris-Begle theorem.
\begin{lemma} \label{lem:indexequiv2}
Let~$\cS \subset \cX$ denote an isolated invariant set for the multivalued
map~$\Pi_\cV$ defined in~(\ref{def:multimap:pi}), and consider the
topological pair~$(Q_1,Q_2)$ defined in~(\ref{conleyindexequiv3}).
Then we have
\begin{displaymath}
  H_*(Q_1, Q_2) \cong H_*(\Cl\cS, \Exit\cS).
\end{displaymath}
\end{lemma}
\begin{proof}
In order to prove the lemma it suffices to verify the assumptions of
Theorem~\ref{thm:vietoris-begle}. For this, we will construct a
map~$\psi_\epsilon : (Q_1,Q_2) \to (|\Cl\cS|, |\Exit\cS|)$ which is a
continuous surjection with~$\psi_\epsilon^{-1}(|\Exit\cS|) = Q_2$, and
which has contractible fibers. We note that, as we verified in
Lemma~\ref{lem:indexequiv1}, the sets~$Q_1$ and~$Q_2$ are
$\cD^d_\epsilon$-representable. Also the sets~$|\Cl\cS|$ and~$|\Exit\cS|$
are $\cD^d_\epsilon$-representable as unions of $\epsilon$-cells which
are representable by Proposition~\ref{prop:cse-representable}.
Therefore, by Corollary~\ref{cor:eps-d-representable-sets}(ii)
both pairs $(Q_1, Q_2)$ and~$(\Cl\cS, \Exit\cS)$ are triangulable.

Thus, we only need to construct the map $\psi_\epsilon$ and verify its properties.
We recall that  $\epsilon>0$ is a fixed constant satisfying~(\ref{eq:epsilon}).
The construction of the map~$\psi_\epsilon$ closely follows
a similar consideration in~\cite{batko:etal:20a}, and therefore
we only present the essential steps. To begin with, define a
function~$\phi_\epsilon : [0,1] \to [0,1]$ via
\begin{displaymath}
  \phi_\epsilon(t) \; = \;
  \left\{ \begin{array}{ccc}
    0 & \mbox{ for } & 0 \le t \le \epsilon , \\[1ex]
    \DS \frac{t - \epsilon}{1 - \epsilon} & \mbox{ for } &
      \epsilon \le t \le 1 .
  \end{array} \right.
\end{displaymath}
This continuous function maps the interval~$[0,\epsilon]$
to zero and the interval~$[\epsilon,1]$ homeomorphically
onto~$[0,1]$. In addition, consider the map
$\psi_\epsilon : X \to X$ defined as
\begin{displaymath}
  \psi_\epsilon(x) \; = \;
  \sum_{u \in \cX_0} \frac{\phi_\epsilon(t_u(x))}
    {\sum\limits_{v \in \cX_0} \phi_\epsilon(t_v(x))} \, u ,
\end{displaymath}
where~$\cX_0$ denotes the set of all vertices of the simplicial
complex~$\cX$. Notice that due to the assumed constraint~(\ref{eq:epsilon})
on~$\epsilon$ the sum in the denominator of the above formula is strictly
positive for all $x \in X$, and this readily implies that~$\psi_\epsilon$
is well-defined and continuous. In addition, one can easily
establish the following properties of~$\psi_\epsilon$:
\begin{itemize}
\item[(i)] For every simplex $\sigma \in \cX$ and any
$x \in \cl\cse{\sigma}$ we have $\psi_\epsilon(x) \in \sigma$. \\[1ex]
This statement follows from the fact that according to
Lemma~\ref{lem:propepscells} one has $t_u(x) \le \epsilon$
for all vertices $u \not\in \sigma$, and therefore
$\phi_\epsilon(t_u(x)) = 0$.
\item[(ii)] For every simplex $\sigma \in \cX$ and any $y \in\;
\stackrel{\circ}{\sigma}$ we have $\psi_\epsilon^{-1}(y) \subset
\cl\cse{\sigma}$.  \\[1ex]
To see this, note that if~$\psi_\epsilon(x) = y$, then the definition
of~$\psi_\epsilon$ implies $\phi_\epsilon(t_u(x)) = 0$
for all vertices $u \not\in \sigma$, and therefore $t_u(x) \le \epsilon$
for all $u \not\in \sigma$. On the other hand, for all vertices
$u \in \sigma$ we have  $\phi_\epsilon(t_u(x)) > 0$, i.e.,
one obtains $t_u(x) > \epsilon$. The statement now follows from
Lemma~\ref{lem:propepscells}.
\item[(iii)] For every simplex $\sigma \in \cX$, any
$x \in \cl\cse{\sigma}$, and every vertex $u \in \sigma$ we have
\begin{displaymath}
  t_u(\psi_\epsilon(x)) =
  \frac{t_u(x) - \epsilon}{1 - \epsilon (1+\dim\sigma) -
    \sum\limits_{w \not\in \sigma} t_w(x)} ,
\end{displaymath}
while $t_u(\psi_\epsilon(x)) = 0$ for all $u \not\in \sigma$. \\[1ex]
For this identity, one uses the fact that $t_v(x) \ge \epsilon$
for all vertices $v \in \sigma$, and therefore the definitions
of~$\phi_\epsilon$ and~$\psi_\epsilon$ imply
\begin{equation} \label{lem:indexequiv2a}
  t_v(\psi_\epsilon(x)) =
  \frac{\phi_\epsilon(t_v(x))}{\sum\limits_{w \in \cX_0}
    \phi_\epsilon(t_w(x))} =
  \frac{t_v(x) - \epsilon}{(1-\epsilon) \sum\limits_{w \in \cX_0}
    \phi_\epsilon(t_w(x))}
  \quad\mbox{for all}\quad
  v \in \sigma,
\end{equation}
which then immediately leads to
\begin{equation}
\label{eq:sum-tvx-eq-1}
t_v(x) = \epsilon + t_v(\psi_\epsilon(x))(1-\epsilon)\sum_{w \in \cX_0}
\phi_\epsilon(t_w(x)).
\end{equation}
Since by (i) we have $\sum_{v \in \sigma} t_v(\psi_\epsilon(x))=1$,
summing both sides of \eqref{eq:sum-tvx-eq-1} for all $v\in\sigma$ we obtain
\[
   \sum_{v \in \sigma}t_v(x)=\epsilon (1+\dim\sigma)+(1-\epsilon)
      \sum_{w \in \cX_0} \phi_\epsilon(t_w(x)).
\]
Substituting $\sum_{v \in \cX_0} t_v(x) = 1 - \sum_{w \not\in \sigma} t_w(x)$
and rearranging we obtain
\[
   (1-\epsilon)\sum_{w \in \cX_0} \phi_\epsilon(t_w(x))=
      1-\epsilon (1+\dim\sigma)-\sum_{w\not\in\sigma}t_w(x).
\]
Combined with~(\ref{lem:indexequiv2a}) this implies the claimed formula.
\item[(iv)] For all $\sigma \in \cX$ we have $\sigma =
\psi_\epsilon(\sigma \cap \cl\cse{\sigma})$. \\[1ex]
In view of~(i) one only has to show that the left-hand side is
contained in the image on the right-hand side.
Since the right-hand side is clearly compact
and $\sigma=\cl\stackrel{\circ}{\sigma}$, it suffices to prove
that $\stackrel{\circ}{\sigma}$ is contained in the right-hand side.
Let $y \in \;\stackrel{\circ}{\sigma}$
and define $x := \sum_{v \in \sigma} (\epsilon + t_v(y) (1 -
\epsilon(1+\dim\sigma))) v \in \sigma$. Then~(iii) implies
that~$\psi_\epsilon(x) = y$, and the statement follows
with~(ii).
\item[(v)] For any simplex $\sigma \in \cX$ and arbitrary
$y \in\; \stackrel{\circ}{\sigma}$ we have $\psi_\epsilon(x) = y$
if and only if
\begin{displaymath}
  t_v(x) =
  \epsilon + t_v(y) \left(1 - \epsilon (1+\dim\sigma) -
    \sum_{w \not\in \sigma} t_w(x) \right)
  \quad\mbox{ for all }\quad
  v \in \sigma ,
\end{displaymath}
as well as $t_v(x) \le \epsilon$ for all $v \not\in \sigma$. \\[1ex]
This statement follows immediately from (ii) and the explicit formula
in~(iii).
  \item[(vi)] $\psi_\epsilon(Q_1) = |\Cl\cS|$ and $\psi_\epsilon(Q_2) = |\Exit\cS|$. \\[1ex]
Using (i) we get $\psi_\epsilon(Q_1) \subset  |\Cl\cS|$ and $\psi_\epsilon(Q_2) \subset |\Exit\cS|$.
The opposite inclusions follow from (iv)
  \item[(vii)] $\psi_\epsilon^{-1}(|\Exit\cS|) = Q_2$. \\[1ex]
  From (vi) we we get $Q_2\subset \psi_\epsilon^{-1}(\psi_\epsilon(Q_2))=\psi_\epsilon^{-1}(|\Exit\cS|)$.
  To see the opposite inclusion take an $x\in\psi_\epsilon^{-1}(|\Exit\cS|)$.
  Let $y:=\psi_\epsilon(x)$ and let $\sigma\in\Exit\cS$ be such that $y\in\; \stackrel{\circ}{\sigma}$.
  Then, by (ii), $x\in\psi_\epsilon^{-1}(y)\subset\cl\cse{\sigma}\subset Q_2$.
\end{itemize}
Property (vi) implies that the restriction of $\psi_\epsilon:X\to X$ to~$Q_1$ is
a well-defined, continuous map of pairs $\left. \psi_{\epsilon} \right|_{Q_1} :
(Q_1,Q_2) \to (|\Cl\cS|,|\Exit\cS|)$. Thus, in view of (vii), in order to apply
the Vietoris-Begle theorem to prove the lemma, we only have to show
that~$\psi_\epsilon$ has contractible fibers.

Thus, let $y\in X$ and let $x\in\psi_\epsilon^{-1}(y)$.
Set $\sigma:=\sigma^0(y)$, where $\sigma^0$ is given by~\eqref{eq:sigma-0}.
Then $y \in \stackrel{\circ}{\sigma}$.
For $\theta\in [0,1]$ and $v\in\cX_0$ set
\begin{equation}
\label{eq:s_x-theta-v}
   s_{x,\theta,v}:=
   \begin{cases}
      \epsilon + t_v(y)\left(1 - \epsilon (1+\dim\sigma) - \sum_{w \not\in \sigma} \theta t_w(x)\right)  & \text{ for $v\in\sigma,$}\\
      \theta t_v(x)                                                                                      & \text{ for $v\not\in\sigma.$}
   \end{cases}
\end{equation}
It follows from \eqref{eq:epsilon} that $s_{x,\theta,v}\geq 0$.
Moreover, one easily verifies that $\sum_{v\in\cX_0}s_{x,\theta,v}=1$.
Hence, $s_{x,\theta,v}\in [0,1]$ and we have a well-defined point $z_{x,\theta}:=\sum_{v\in\cX_0}s_{x,\theta,v}v\in X$,
because clearly $t_v(z_{x,\theta})>0$ implies $t_v(x)>0$, that is, $\sigma^0(z_{x,\theta})\subset\sigma^0(x)\subset X$.
Moreover, we easily get from (v) that $z_{x,\theta}\in \psi_\epsilon^{-1}(y)$. Thus, we have a well-defined homotopy
\begin{displaymath}
  h_y : [0,1] \times \psi_\epsilon^{-1}(y) \to \psi_\epsilon^{-1}(y).
\end{displaymath}
Again using~(v) one can easily see that~$h_y(1,x) = x$
for all $x \in \psi_\epsilon^{-1}(y)$. Finally, for every~$x$
in the fiber~$\psi_\epsilon^{-1}(y)$ formula \eqref{eq:s_x-theta-v} shows that
the point~$h_y(0,x)$ is independent of $x$, i.e., the map~$h_y(0,\cdot)$
is constant. This proves the contractibility
of~$\psi_\epsilon^{-1}(y)$ and the result follows.
\end{proof}
\medskip

After these preparations, we collect the results of this and the
last section in the following theorem, which is valid under the
weak notion of admissible flow.
\begin{theorem}[Isolated Invariant Sets and Conley Index Equivalence]
\label{thm:indexequiv}
Let~$\cS \subset \cX$ be an isolated invariant set for the multivalued
map~$\Pi_\cV$ defined in~(\ref{def:multimap:pi}) and
let~$\phi : \R_0^+ \times X \to X$ denote an arbitrary admissible
semiflow for~$\cV$ in the sense of Definition~\ref{def:admissibleflow}.
Then the set
\begin{displaymath}
  B := N_\epsilon(\cS) = \bigcup_{\sigma \in \cS} \cl\cse{\sigma}
\end{displaymath}
is an isolating block for~$\phi$ and we have
\begin{displaymath}
  H_*(\Cl\cS, \Exit\cS) \cong H_*(B, B^-) .
\end{displaymath}
In other words, every combinatorial isolated invariant set gives
rise to an isolated invariant set for~$\phi$ with the same Conley
index in the classical setting.
\end{theorem}
\begin{proof}
The result follows immediately from Proposition~\ref{prop:isolatingblocks}
and Lemmas~\ref{lem:indexequiv1} and~\ref{lem:indexequiv2}.
\end{proof}
\subsection{Equivalence of Morse Decompositions and Conley-Morse Graphs}
\label{sec44}
In this final part of Section~\ref{sec4} we demonstrate that under the
assumption of strong admissibility, any semiflow~$\phi$ on~$X$ exhibits
the same global dynamics as the underlying combinatorial vector field.
More precisely, we have the following result.
\begin{theorem}[Morse Decomposition Equivalence]
\label{thm:morsedecompequiv}
Suppose we are given a Morse decomposition $\cM = \{\cM_p \, | \,
p \in \mbbP\}$ in the sense of Definition~\ref{def:morsedecomp}.
Let~$\phi : \R_0^+ \times X \to X$ denote an arbitrary strongly admissible
semiflow for~$\cV$ in the sense of Definition~\ref{def:admissibleflow}
and let
\begin{displaymath}
  M_p :=
    \Inv\left( \bigcup_{\sigma \in \cM_p} \cl\cse{\sigma} \right)
  \qquad\mbox{ for all }\qquad
  p \in \mbbP .
\end{displaymath}
Then the collection $M = \{ M_p \, | \, p \in \mbbP\}$ is a Morse
decomposition for the semiflow~$\phi$.  Moreover, its Conley-Morse graph is
isomorphic to the Conley-Morse graph of~$\cM$.
\end{theorem}
\begin{proof}
In view of~(\ref{conleyindexequiv1}) we have $M_p =
\Inv(N_\epsilon(\cM_p))$ for all $p \in \mbbP$. According to
Theorem~\ref{thm:indexequiv} the sets~$N_\epsilon(\cM_p)$ are
isolating blocks for~$\phi$.
Since the sets $\cM_p$ are pairwise disjoint, the sets~$N_\epsilon(\cM_p)$ can only intersect along
their boundaries. This immediately implies that the sets~$M_p$
are disjoint isolated invariant sets for~$\phi$. Moreover,
Theorem~\ref{thm:indexequiv} also shows that the Poincar\'e
polynomials of~$M_p$ and~$\cM_p$ agree.

It remains to show that the $\alpha$- and $\omega$-limit sets
of any solution for~$\phi$ (as long as they are defined) are contained
in Morse sets with the correct order relationship, and that if both
of these limit sets are contained in the same Morse set, then the
whole solution is contained in the Morse set. In the following,
we only consider the case of full solutions for~$\phi$, as the case
of forward solutions can be treated completely analogously.

Assume therefore that $\gamma : \R \to X$ denotes an arbitrary
full solution of~$\phi$. Due to the compactness of~$X$, both
its $\alpha$- and its $\omega$-limit sets exist, are nonempty,
and invariant. Let~$\cC$ denote the set of flow tiles
associated with the combinatorial vector field~$\cV$, as defined
in Definition~\ref{def:flowtiles}. Then we can define a multivalued
map $\eta : \R \multimap \cC$ via
\begin{displaymath}
  \eta(t) := \left\{ C \in \cC \; \mid \;
    \gamma(t) \in C \right\}
  \qquad\mbox{ for all }\qquad
  t \in \R .
\end{displaymath}
Then the following hold:
\begin{itemize}
\item Due to our definition of admissibility, if~$\eta(t_0)$
contains more than one flow tile, then there exists a
$\delta > 0$ such that~$\gamma(t)$ is single-valued for
all $t \in (t_0 - \delta, t_0 + \delta) \setminus \{ t_0 \}$.
\item According to the continuity of~$\phi$ and the
characterization of the boundaries of $\epsilon$-cells given
in Lemma~\ref{lem:propepscells}, if~$\eta(t_0)$ contains
exactly one flow tile, then there exists a $\delta > 0$
such that~$\gamma(t)$ is single-valued for all $t \in
(t_0 - \delta, t_0 + \delta)$.
\item Combined, these two facts show that the times at
which~$\eta(t)$ contains more than one flow tile do not
have any accumulation points. Thus, there exists a
$\Z$-interval~$I \subset \Z$ and a strictly increasing
sequence~$(t_k)_{k \in I} \subset \R$ without accumulation
points such that~$\eta$ is multivalued on
$R = \{ t_k \; \mid\; k \in I \}$ and single-valued on $\R \setminus R$.
\end{itemize}
Based on the above observations we can now choose a
$\Z$-interval~$J \subset \Z$ whose cardinality equals the number
of connected components of~$\R \setminus R$, which is bounded from
below if~$\R^- \setminus R$ has a component of infinite
length, and bounded from above if~$\R^+ \setminus R$ has a
component of infinite length. Furthermore, for every $k \in J$
we can choose precisely one point~$t_k^*$ in each of the connected
components of~$\R \setminus R$ such that $t_k^* < t_j^*$ whenever
$k < j$. In other words, as~$k$ increases through~$J$ the points~$t_k^*$
hit every connected component of~$\R \setminus R$ precisely once, with
respect to the standard linear increasing order. Then~$\eta(t_k^*)$
contains exactly one flow tile associated with a simplex $\sigma_k\in\cX$.
Without loss of generality we may assume that~$\sigma_k=\sigma_k^-$.
Now define a function~$\Gamma : J \to \cX$ via
$
  \Gamma(k) := \sigma_k=\sigma_k^-.
$
Thus, as~$k$ ranges through~$J$ in increasing order, the
sequence~$\Gamma(k)$ visits the lowest-dimensional simplices
associated with the flow tiles which are traversed by~$\gamma$,
in the correct order, and with every arrow flow tile being represented
by the arrow tail simplex. In particular, we have $\Gamma(k) \neq
\Gamma(k+1)$ for all arguments $k, k+1 \in J$.

If $k,k+1\in J$, then the solution $\gamma$ exits $\cl\cse{\sigma_k}$
at a time $\bar{t}\in(t_k^*,t_{k+1}^*)$.
Then $\gamma(\bar{t})\in\cl\cse{\sigma_k}$ and from Lemma~\ref{lem:altcharepscell}
we get $\sigma^\epsilon_{\min}(\gamma(\bar{t}))\subset\sigma_k$.
By \eqref{def:admissibleflow1}   we know that $\gamma$ enters
$\cse{\sigma^\epsilon_{\min}(\gamma(\bar{t}))}$
when crossing $\bar{t}$. Therefore, $\sigma_{k+1}=\sigma^\epsilon_{\min}(\gamma(\bar{t}))^-$
and
$$
\Gamma(k+1)=\sigma_{k+1}=\sigma^\epsilon_{\min}(\gamma(\bar{t}))^-\subset \sigma^\epsilon_{\min}(\gamma(\bar{t}))\subset
\sigma_k\subset\sigma_k^+=\Gamma(k)^+.
$$
Thus, we proved that for all $k, k+1 \in J$ the simplex~$\Gamma(k+1)$ is a face
of~$\Gamma(k)^+$.
Moreover, if~$\eta(t_k^*)$ contains an arrow flow
tile, then one necessarily has $\Gamma(k+1) \neq \Gamma(k)^-$.
This immediately implies that the arrowhead extension of~$\Gamma$
as defined in~\cite[Definition~5.2]{kaczynski:etal:16a} is a
solution of the combinatorial vector field~$\cV$. In addition, the
following two implications hold:
\begin{eqnarray*}
  k^+ := \sup J <\infty & \Rightarrow & \mbox{$\Gamma(k^+)$
    is a critical simplex} , \\
  k^- := \inf J >-\infty & \Rightarrow & \mbox{$\Gamma(k^-)$
    is a critical simplex} .
\end{eqnarray*}
We only verify the first implication, as the second one can be
established analogously. If the implication were false, then
one would have $\gamma(t) \in C$ for all $t \ge t_{k^+}^*$
for some arrow flow tile~$C \in \cC$. This, however, contradicts
our assumption of strong admissibility of~$\phi$ (see
Definition~\ref{def:admissibleflow}).

Based on this discussion, the arrowhead extension of the
mapping~$\Gamma$ can be extended to a full solution
$\Gamma_{\mathrm{full}} : \Z \to \cX$ of~$\cV$. In the cases
$k^+ <\infty$ or $k^- > - \infty$ one only has to pad the respective
end with infinite repetitions of the critical simplex~$\Gamma(k^+)$
or~$\Gamma(k^-)$, respectively. Moreover, the following holds.
\begin{itemize}
\item Due to the definition of Morse decomposition for~$\cV$
there exist Morse sets~$\cM_p$ and~$\cM_q$ with $p \ge q$ such
that $\alpha(\Gamma_{\mathrm{full}}) \subset \cM_p$ and
$\omega(\Gamma_{\mathrm{full}}) \subset \cM_q$. In addition,
if $p = q$ then one has $\Gamma_{\mathrm{full}}(\Z) \subset \cM_p$.
\item According to our construction, the classical solution~$\gamma$
traverses the flow tiles associated with~$\Gamma_{\mathrm{full}}(\Z)$
as~$t$ ranges through~$\R$. This implies $\alpha(\gamma) \subset
N_\epsilon(\cM_p)$ and $\omega(\gamma) \subset N_\epsilon(\cM_q)$.
Since both limit sets are invariant sets, we finally obtain
$\alpha(\gamma) \subset M_p$ and $\omega(\gamma) \subset M_q$.
Moreover, if~$p = q$ then $\gamma(\R) \subset
N_\epsilon(\Gamma_{\mathrm{full}}(\Z)) \subset N_\epsilon(\cM_p)$.
As a full solution, we therefore have $\gamma(\R) \subset
\Inv(N_\epsilon(\cM_p)) = M_p$.
\end{itemize}
This completes the proof of the theorem.
\end{proof}
\medskip

The above theorem is the precise version of Theorem~\ref{thm:intro2}
from the introduction. Notice that strong admissibility is essential,
since without it we would not have been able to conclude that in the
case~$k^+ <\infty$ (respectively $k^- > - \infty$) the simplex~$\Gamma(k^+)$ (respectively~$\Gamma(k^-)$) is critical. This
in turn was necessary for the construction of the full combinatorial
solution~$\Gamma_{\mathrm{full}}$.
\section{Existence of Strongly Admissible Semiflows}
\label{sec5}
In this section we show that for any combinatorial vector
field~$\cV$ on a simplicial complex~$\cX$ one can explicitly construct
a strongly admissible semiflow~$\phi$ on the polytope $X = |\cX|$. This
construction is based on the three design principles outlined in the
introduction and is divided into four parts. In Section~\ref{sec51}
we present the precise definition of the vector field which generates
the semiflow. This definition relies on a family of vector fields~$f^\omega$
which are indexed by the simplices $\omega \in \cX$, each of which describes
the semiflow in its associated flow tile. As this vector field definition is
somewhat involved, we spend the remainder of this section describing the main
features of the semiflow and its geometry on different flow tiles. The
next three sections are devoted to showing that the fields~$f^\omega$ do
indeed generate a continuous, strongly admissible semiflow on~$X$. In Section~\ref{sec52}
we establish that each vector field~$f^\omega$ generates a continuous
semiflow, and discuss some of its elementary properties. This is followed
in Section~\ref{sec53} with a detailed study of the induced dynamics on
a flow tile, in particular the behavior near the boundary of the flow
tile. In this section we also derive crucial results towards strong
admissibility. Finally, Section~\ref{sec54} combines the previous
results to generate a continuous semiflow on~$X$ via gluing the semiflows
generated by the vector fields~$f^\omega$.
\subsection{Vector Field Definition and Basic Geometry}
\label{sec51}
In view of the three design principles which are described in the
introduction, our goal is the construction of a semiflow on
the underlying polytope~$X$ of an abstract simplicial complex~$\cX$.
On the one hand, we would like this semiflow to be generated by a differential
equation, while on the other hand the resulting semiflow generally
will have to have velocity jumps. In addition, our definition has
to be flexible enough to accommodate different tilings based on
different combinatorial vector fields~$\cV$ on~$\cX$.

With these considerations in mind, we have adopted the following
framework for the construction of a strongly admissible semiflow.
First of all, as in the preceding sections, it is convenient to assume that
the geometric realization of the abstract simplicial complex~$\cX$
is the standard geometric realization (see Section~\ref{sec:geom-real}).
Recall that the standard geometric realization of $\cX$ is a subcomplex
of the standard $d$-simplex in $\R^d$ where $d := \card{\cX_0}$
denotes the number of vertices in~$\cX$. It is based on an identification
of vertices in $\cX_0$ with versors of $\R^d$ which lets us write
the coordinates of a vector $x\in\R^d$ in the form $x_v$ where $v\in\cX_0$.
In the standard geometric realization of $\cX$ a simplex $\sigma \in \cX$
is represented as a geometric simplex  consisting
of all points~$x \in \R^d$ with $\sum_{u \in \sigma} x_u = 1$, as
well as $x_u \ge 0$ for all $u \in \sigma$ and~$x_u = 0$ for all
$u \not\in \sigma$. Notice that the notation for the components
of~$x$ is an extension of the barycentric coordinate notation
introduced earlier.

The standard geometric realization has the advantage that we can
extend the range of every barycentric coordinate associated with
the simplicial complex to take values in the reals, and this in
turn will allow us to define a vector field on all of~$\R^d$ in
such a way that modifications of standard results give the
existence of the semiflow, and that this semiflow leaves
the underlying polytope of the simplicial complex invariant.
We would like to point out, however, that once we have constructed
a strongly admissible semiflow on the underlying polytope~$X$
of the standard geometric realization,
one can easily map it onto underlying polytopes of
other geometric realizations of the
given abstract simplicial complex.

Having settled on the geometric realization, we now turn
our attention to the underlying principles for the definition of
the {\em vector field\/}:
\begin{itemize}
\item Our vector field will be defined piece-wise. In fact, for
every simplex $\omega \in \cX$ there will be a bounded and
measurable vector field~$f^\omega : \R^d \to \R^d$ which induces
a semiflow through the ordinary differential equation
$\dot{x} = f^\omega(x)$, interpreted in the Carath\'eodory sense.
\item The vector fields~$f^\omega$ only depend on the underlying
flow tiles associated with a combinatorial vector field~$\cV$
on~$\cX$, i.e., we have $f^\omega = f^{\omega^+} = f^{\omega^-}$
for all $\omega \in \cX$, where again we use the notation introduced
in~(\ref{def:flowtiles1}). In fact, we are only interested in the
behavior of~$f^\omega$ on the flow tile $C_\omega = \cl\cse{\omega^-}
\cup \cl\cse{\omega^+}$ defined in~(\ref{def:flowtiles2}).
\item The final strongly admissible semiflow is defined in the
spirit of a Filippov system on~$X$ through the vector field
family~$\{ f^\omega \}_{\omega \in \cX}$ over the flow tiles
in~$\cC$, and its solutions cross the boundaries of the flow tiles
transversally. Unfortunately, however, standard results on Filippov
systems do not apply in our situation, and we have to construct the
semiflow differently.
\end{itemize}
We recall that as in the earlier sections, $\epsilon>0$ denotes a small
positive parameter satisfying~\eqref{eq:epsilon}. So far it was used in
the construction of the $\epsilon$-cells in Section~\ref{sec22} and flow
tiles in Section~\ref{sec23}. It will also be used in this section to
construct some auxiliary functions needed to present the precise
definition of the vector fields~$f^\omega$. Starting with
Proposition~\ref{prop:vectorfieldbound}, apart from assumption~\eqref{eq:epsilon}
we will also require that~$\epsilon$ satisfies
\begin{equation} \label{prop:vectorfieldbound1}
  0 < \epsilon < \frac{1}{6d}
\end{equation}
where $d=\card\cX_0$.
We begin by defining  two auxiliary scalar functions~$g, h : \R \to \R$ via
\begin{equation} \label{def:vectorfield1}
  g(s) := \sqrt[3]{\epsilon^2 s}
  \qquad\mbox{ and }\qquad
  h(s) := \left\{
  \begin{array}{ccc}
    \DS 1 & \mbox{ if } &
      \DS \left| s - \epsilon \right| \ge \frac{\epsilon}{2}
      , \\[2ex]
    \DS \left( 1 + \frac{\epsilon}{2} \right)
      \frac{2}{\epsilon} \left| s - \epsilon \right|
      - \frac{\epsilon}{2} & \mbox{ if } &
      \DS \left| s - \epsilon \right| \le \frac{\epsilon}{2}
      .
  \end{array}
  \right.
\end{equation}
These two functions are shown in blue in the left and right panels
of Figure~\ref{fig:ghfunction}, respectively. Notice that the function~$g$
is continuous, but not differentiable at~$s = 0$, while the function~$h$
is Lipschitz continuous on~$\R$. In fact, the function~$h$ is of constant
value~$1$ everywhere except in a small neighborhood around~$s = \epsilon$,
where it drops to the slightly negative value~$-\epsilon / 2$.
\begin{figure}[tb]
  \centering
  \includegraphics[height=0.20\textheight]{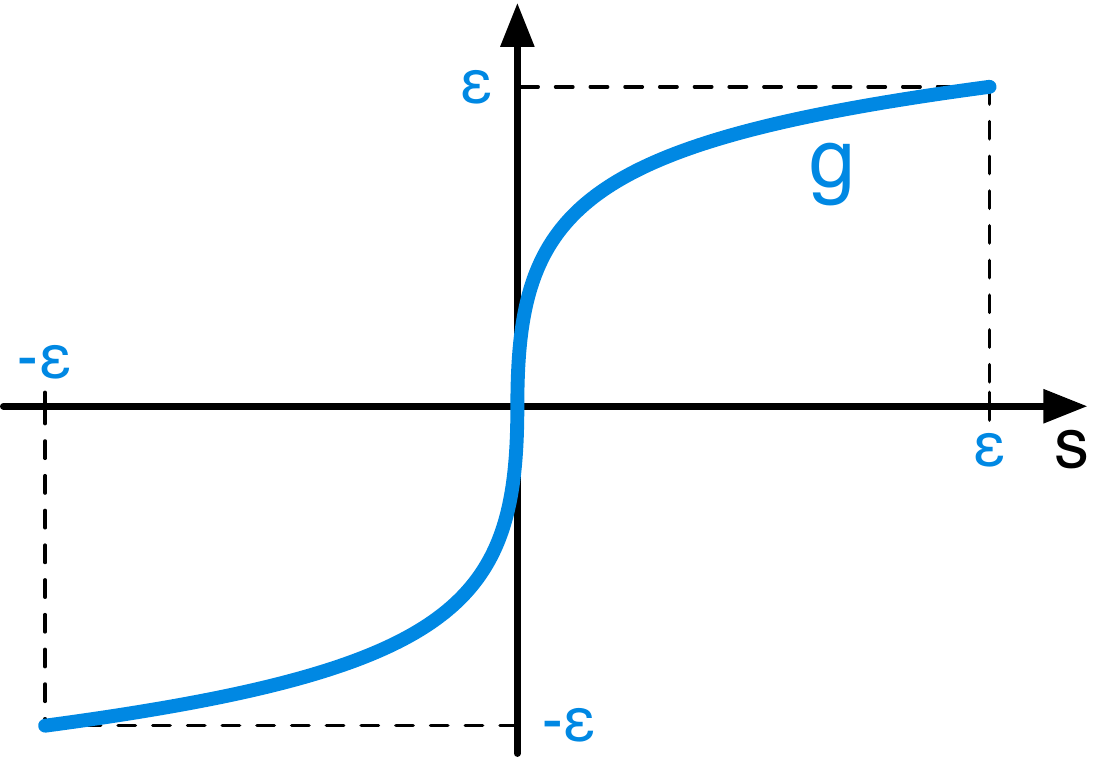}
  \hspace*{0.75cm}
  \includegraphics[height=0.20\textheight]{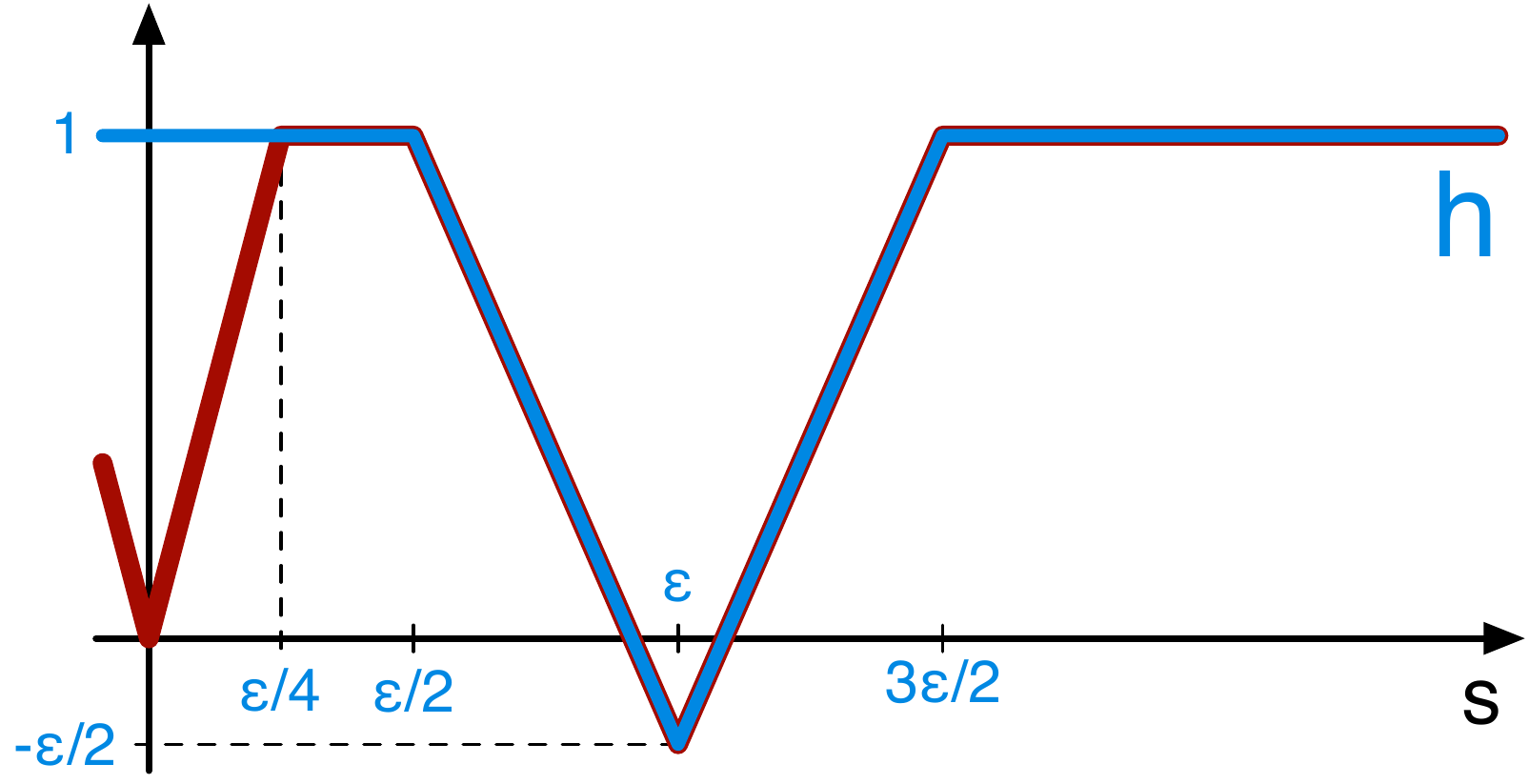}
  \caption{Two auxiliary functions for the definition of the vector
           fields~$f^\omega$. The left and right panels depict in
           blue the functions~$g$ and~$h$, respectively, which are
           defined in~(\ref{def:vectorfield1}). The red function in
           the right panel shows the product~$\eta^\omega(s,x) h(s)$
           if one has $\sum_{u \not\in \omega^+} x_u^2 > 0$,
           which occurs in the second part of the definition
           of~$f^\omega$ in~(\ref{def:vectorfield4}).}
  \label{fig:ghfunction}
\end{figure}

In contrast to~$g$ and~$h$, which only depend on the small
parameter~$\epsilon$, the next two auxiliary functions also
depend on the underlying simplex~$\omega \in \cX$. The
function~$\theta^\omega : \R^d \to \R$ is given by the
definition
\begin{equation} \label{def:vectorfield2}
  \theta^\omega(x) :=
  \min \left( \{ \epsilon \} \cup \left\{ x_u - \epsilon \; \mid \;
    u \in \omega^- \right\} \right) .
\end{equation}
Notice that for any point~$x$ in the underlying polytope $X$
of~$\cX$, we have $\theta^\omega(x) \in [-\epsilon, \epsilon]$,
and the identity $\theta^\omega(x) = \epsilon$ is satisfied if and
only if~$x_u \ge 2\epsilon$ for all $u \in \omega^-$. The final
auxiliary function $\eta^\omega : \R \times \R^d \to \R$ is
defined as
\begin{equation} \label{def:vectorfield3}
  \eta^\omega(s,x) := \left\{
  \begin{array}{ccc}
    \DS 1 & \mbox{ if } & \DS \sum_{u \not\in \omega^+} x_u^2 = 0
      \;\;\mbox{ or }\;\; |s| > \frac{\epsilon}{4} , \\[4ex]
    \DS \frac{4|s|}{\epsilon} & \mbox{ if } & \DS \sum_{u \not\in \omega^+}
      x_u^2 > 0 \;\;\mbox{ and }\;\; |s| \le \frac{\epsilon}{4} .
  \end{array}
  \right.
\end{equation}
For the case $\sum_{u \not\in \omega^+} x_u^2 > 0$ the
product~$\eta^\omega(s,x) h(s)$ is shown in red
in the right panel of Figure~\ref{fig:ghfunction}.

After these preparations, for every $\omega \in \cX$  the vector
field $f^\omega : \R^d \to \R^d$ is defined componentwise in
the form
\begin{equation} \label{def:vectorfield4}
  f^\omega_v(x) := \left\{
  \begin{array}{rcl}
    \DS -g(x_v)
      & \mbox{ for } & v \not\in \omega^+ , \\[2ex]
    \DS \eta^\omega(x_v,x) \left( h(x_v) + \theta^\omega(x)
      - \sum_{u \not\in \omega^+} x_u \right)
      & \mbox{ for } & v \in \omega^+ \setminus \omega^- , \\[5ex]
    \DS x_v - \frac{1}{\card{\omega^-}} \left( \sum_{u \in \omega^-} x_u
      + \sum_{u \not\in \omega^-} f^\omega_u(x) \right)
      & \mbox{ for } & v \in \omega^- .
  \end{array}
  \right.
\end{equation}
At first glance, the vector field definition given in
formulas~(\ref{def:vectorfield1}) through~(\ref{def:vectorfield4})
is clearly overwhelming. Thus, before we establish that the so-defined
vector fields do indeed generate a strongly admissible semiflow~$\phi$
on the underlying polytope~$X \subset \R^d$ of $\cX$, we pause for a brief
description of the main features of the induced semiflow on the two
types of flow tiles.

\smallskip\noindent
{\bf (I) The induced semiflow on critical flow tiles.}
To begin with, suppose that $\omega \in \cX$ is a critical cell for the
combinatorial vector field~$\cV$, and let $C_\omega = \cl\cse{\omega}
\subset X$ denote the associated flow tile. Then the induced
semiflow~$\phi$ on~$C_\omega$ is given by the solution of the
ordinary differential equation
\begin{displaymath}
  \begin{array}{rclcl}
    \DS \dot{x}_v & = & \DS x_v - \frac{1}{1+\dim\omega}
      \left( \sum_{u \in \omega} x_u - \sum_{u \not\in \omega}
      g(x_u) \right) &
      \mbox{ for } & \DS v \in \omega , \\[5ex]
    \DS \dot{x}_v & = & \DS -g(x_v) &
      \mbox{ for } & \DS v \not\in \omega .
  \end{array}
\end{displaymath}
First consider the intersection~$C_\omega \cap \omega$. Then
the semiflow further reduces to the linear differential equation
\begin{displaymath}
  \dot{x}_v = x_v - \frac{1}{1+\dim\omega}
  \qquad\mbox{ for all }\qquad
  v \in \omega
  \quad\mbox{ and }\quad
  x \in C_\omega \cap \omega .
\end{displaymath}
This differential equation has a unique equilibrium at the barycenter
with barycentric coordinates $x_v := 1 / (1+\dim\omega)$ for all
$v \in \omega$. This equilibrium is
unstable with index~$\dim\omega$, and the induced flow is towards the
boundary~$\bd C_\omega \cap \omega$. See also the illustration in
Figure~\ref{fig:semiflowc}.
\begin{figure}[tb]
  \centering
  \raisebox{3.0cm}{\parbox{6.3cm}{
    \includegraphics[width=0.40\textwidth]{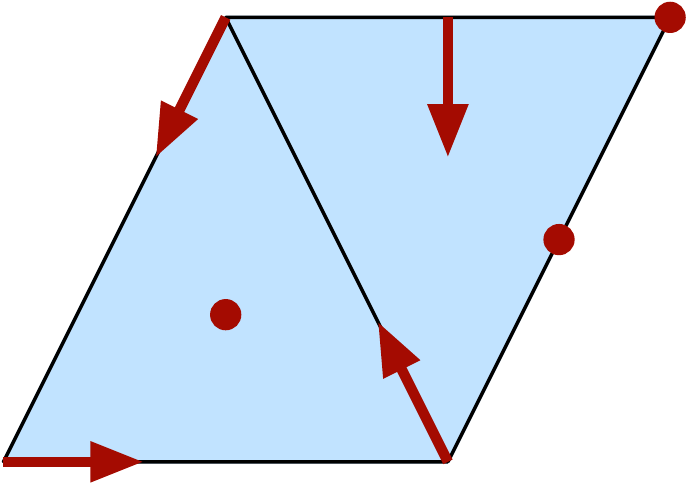}}}
  \includegraphics[width=0.55\textwidth]{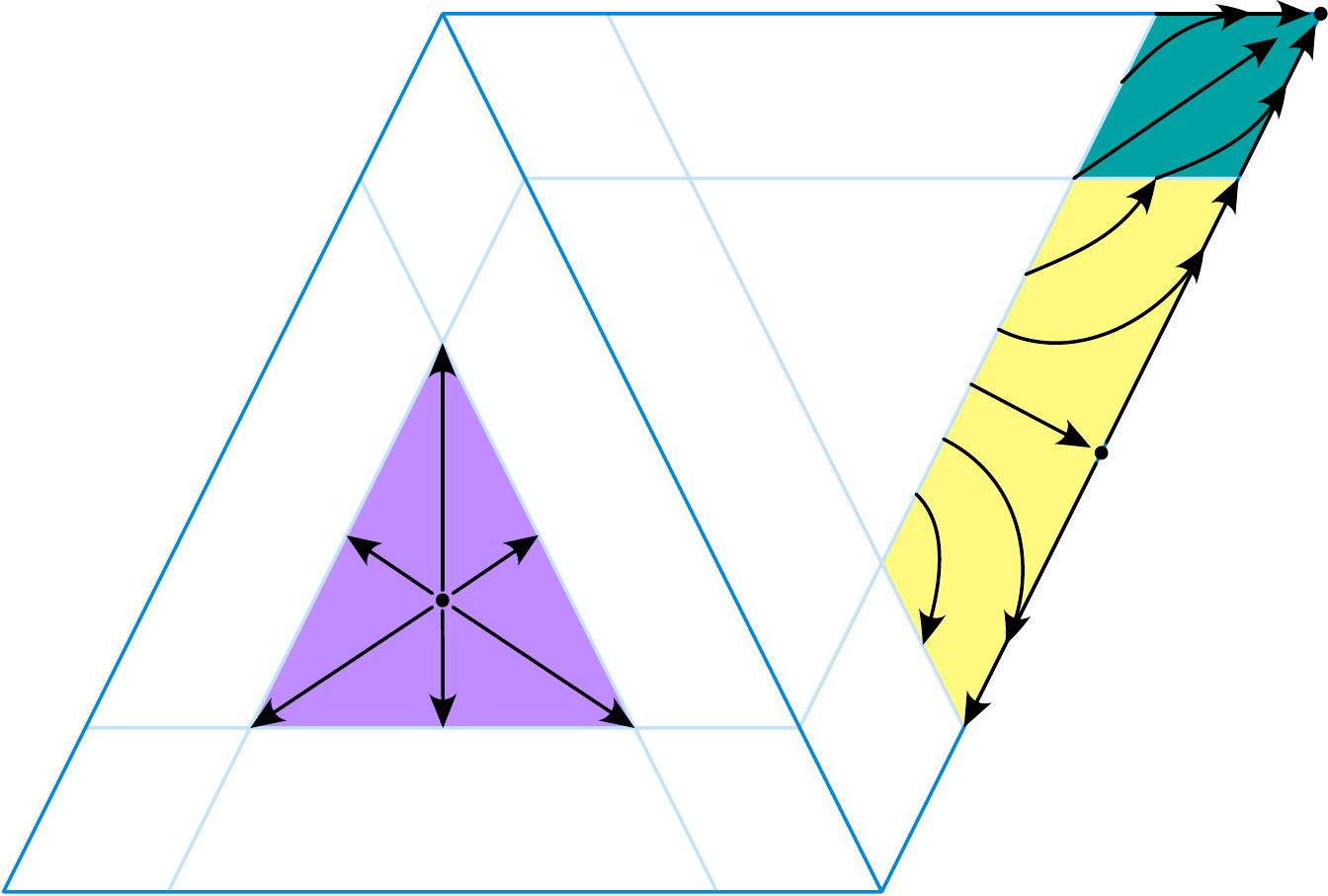}
  \caption{Semiflow induced by~(\ref{def:vectorfield4}) on
           critical flow tiles. For the combinatorial vector
           field~$\cV$ shown in the left panel, the panel on
           the right depicts the flow tiles for the three
           critical cells, together with the strongly admissible
           semiflow induced by~(\ref{def:vectorfield4}) on each
           of these tiles. Note that in all three cases the flow
           moves towards the boundary on the intersection of the
           flow tile and the associated simplex. Solutions which merge
           in finite forward time can be observed in the tiles
           associated with the critical cells of Morse index one and zero.}
  \label{fig:semiflowc}
\end{figure}

Consider now solutions which originate at points~$x \in C_\omega
\setminus \omega$. In this situation, there are vertices
$v \not\in \omega$ with $x_v \neq 0$, and the behavior of
these coordinates is determined by the completely decoupled
scalar ordinary differential equations
\begin{displaymath}
  \dot{x}_v = -\sqrt[3]{\epsilon^2 x_v}
  \qquad\mbox{ for all }\qquad
  v \not\in \omega
  \quad\mbox{ and }\quad
  x \in C_\omega \setminus \omega .
\end{displaymath}
One can easily see that these coordinates decay towards zero ---
and they will in fact reach zero in finite forward time and stay
there from then on. Thus, solutions originating in $C_\omega \setminus
\omega$ are never constant, and they can enter~$C_\omega \cap \omega$
in finite forward time, unless of course they exit the flow tile
before that happens. This implies that the semiflow on~$C_\omega
\setminus \omega$ is attracted towards~$C_\omega \cap \omega$
and roughly follows the flow behavior on~$C_\omega \cap \omega$,
leading to the qualitative
semiflow behavior shown in Figure~\ref{fig:semiflowc}. Notice that
every critical flow tile contains exactly one equilibrium, whose
index is given by~$\dim\omega$. We would like to explicitly point
out, however, that solutions of the semiflow~$\phi$ can reach
lower-dimensional faces of a simplex in finite forward time.

\smallskip\noindent
{\bf (II) The induced semiflow on arrow flow tiles.}
Now suppose that the simplex $\omega \in \cX$ is part of an arrow
for the combinatorial vector field~$\cV$, i.e., we have $ \omega^+ \neq \omega^-$,
and the associated flow tile is given by
$C_\omega = \cl\cse{\omega^-} \cup \cl\cse{\omega^+} \subset X$. Then
the induced semiflow~$\phi$ on~$C_\omega$ is given by the solution
of the ordinary differential equation
\begin{displaymath}
  \begin{array}{rclcl}
    \DS \dot{x}_{v^+} & = & \DS \eta^\omega(x_{v^+},x)
      \left( h(x_{v^+}) + \theta^\omega(x) -
      \sum_{u \not\in \omega^+} x_u \right)
      & \mbox{ for } & \{ v^+ \} = \omega^+ \setminus \omega^-
      , \\[5ex]
    \DS \dot{x}_v & = & \DS x_v - \frac{1}{\card{\omega^-}}
      \left( \sum_{u \in \omega^-} x_u
      + \sum_{u \not\in \omega^-} f^\omega_u(x) \right)
      & \mbox{ for } & v \in \omega^- , \\[5ex]
    \DS \dot{x}_v & = & \DS -g(x_v)
      & \mbox{ for } & v \not\in \omega^+ .
  \end{array}
\end{displaymath}
The last equation describes the evolution of the vector
components~$x_v$ for $v \not\in \omega^+$, and as in the
previous case, these equations are completely decoupled from
the other components. These differential equations have a unique
equilibrium at zero, and they attract nonzero values in finite
forward time. Similarly, the second equation, which describes
the evolution of the vector components~$x_v$ for $v \in \omega^-$
has a form similar to the one in the critical flow tile case, and
it generally leads to flow towards the boundary.
\begin{figure}[tb]
  \centering
  \raisebox{3.0cm}{\parbox{6.3cm}{
    \includegraphics[width=0.40\textwidth]{semiflowV.pdf}}}
  \includegraphics[width=0.55\textwidth]{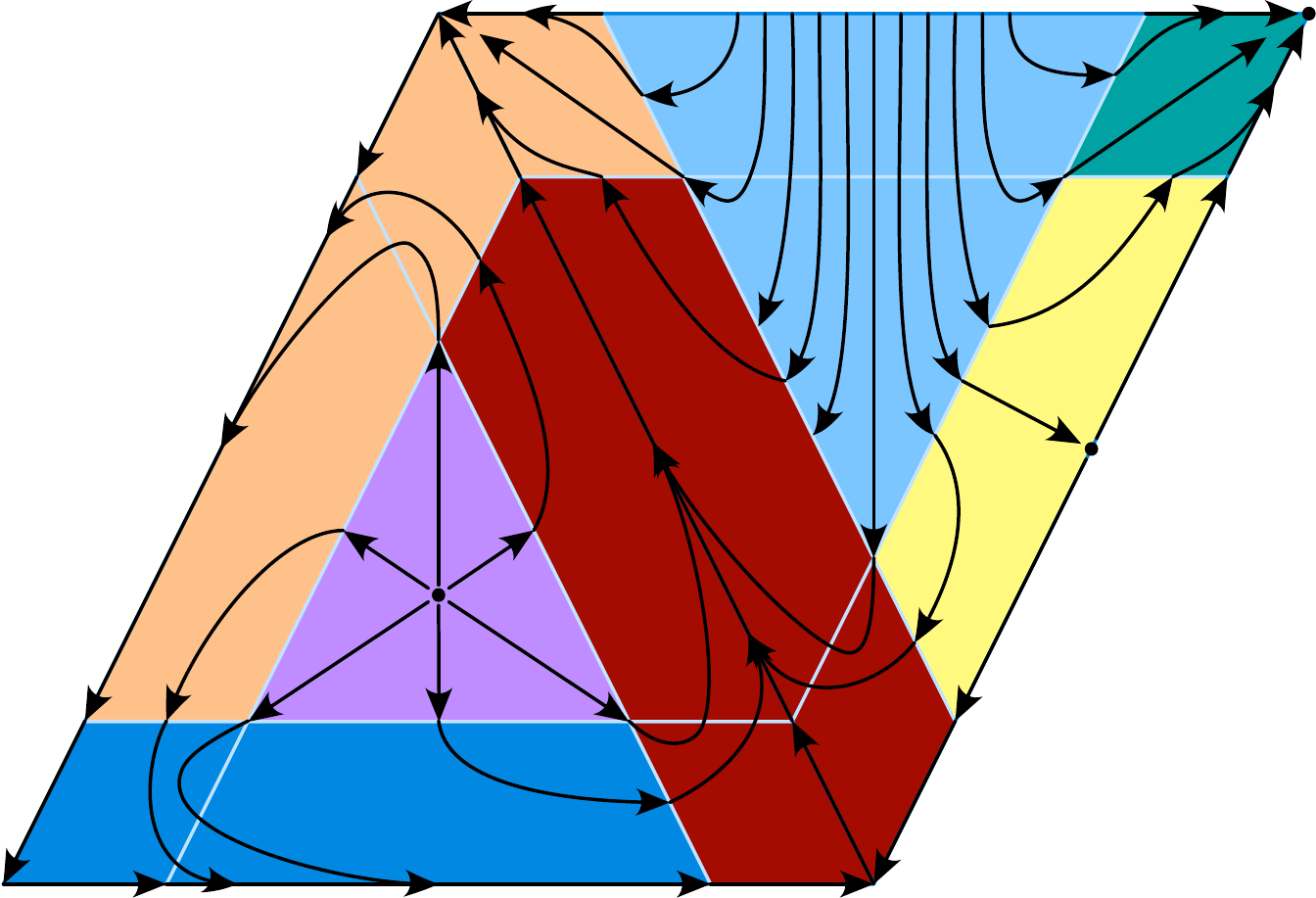}
  \caption{Complete strongly admissible semiflow induced
           by~(\ref{def:vectorfield4}). For the combinatorial vector
           field~$\cV$ shown in the left panel, the panel on
           the right illustrates the strongly admissible semiflow
           induced by~(\ref{def:vectorfield4}) on each of the flow
           tiles. Solutions generally exhibit merging in finite forward
           time, and the semiflow lowers the dimension of the
           (smallest) simplex which contains the solution. This can
           clearly be seen along the four outer edges of the simplicial
           complex. Moreover, significant discontinuous velocity changes
           can be observed in arrow flow tiles along the simplex~$\omega^-$,
           see for example the top left vertex, or the two vertices of
           the lowermost edge.}
  \label{fig:semiflowa}
\end{figure}

In view of these observations, we focus mainly on the first equation,
which describes the flow of the vector component~$x_{v^+}$
where $v^+$ stands for the unique vertex in~$\omega^+ \setminus \omega^-$.
Also, assume for the moment that the location~$x \in C_\omega \subset X$ of
the solution satisfies $\sum_{u \not\in \omega^+} x_u^2 = 0$, i.e.,
we have $x \in C_\omega \cap \omega^+$. In this case, the identity
$\eta^\omega(x_{v^+},x) = 1$ holds, and the equation
for~$\dot{x}_{v^+}$ reduces to the first two terms in
the parentheses. These terms have different responsibilities:
\begin{itemize}
\item The term~$h(x_{v^+})$ provides the general profile for the
velocity~$\dot{x}_{v^+}$, which usually satisfies~$\dot{x}_{v^+} > 0$.
The only exception is a small neighborhood around~$x_{v^+} = \epsilon$,
where the velocity is negative. Note, however, that this can be changed
easily by adding a constant larger than~$\epsilon / 2$ to the first
term, which leads to~$\dot{x}_{v^+} > 0$ even if one has~$x_{v^+}
\approx \epsilon$.
\item Introducing vertical shifts of the function~$h$ is the
responsibility of the second term. As we have mentioned,
this term satisfies $\theta^\omega(x) \in [-\epsilon, \epsilon]$,
and the identity $\theta^\omega(x) = \epsilon$ is satisfied if and
only if~$x_u \ge 2\epsilon$ for all $u \in \omega^-$. Thus, as long
as~$x$ is sufficiently far away from the boundary of the flow tile,
the velocity~$\dot{x}_{v^+}$ is strictly positive. Close to the
boundary of~$C_\omega$ and for~$x_{v^+} \approx \epsilon$ it becomes
negative, which is required for the admissibility of the semiflow.
\end{itemize}
This behavior is illustrated in Figure~\ref{fig:semiflowa} in
the flow tile associated with the vertical arrow whose base is
the top edge of the simplicial complex.

We now turn our attention to points $x \in C_\omega \setminus
\omega^+$. As long as one also has~$x_{v^+} > \epsilon / 4$,
the flow description from above still applies. For this one only
has to realize that in this case we have $\eta(x_{v^+},x)=1$ and
the new shift term $\theta^\omega(x) - \sum_{u \not\in \omega^+} x_u$
is strictly smaller than~$\epsilon / 2$ as long as $x_u \approx \epsilon$
for some vertex~$u \neq v^+$. In other words, close to the
boundary of~$C_\omega$ the flow satisfies the admissibility
condition.

The semiflow behavior changes, however, at points $x \in C_\omega
\setminus \omega^+$ with $x_{v^+} = 0$. At such points, the prefactor
$\eta^\omega(x_{v^+},x)$ is zero, i.e., it keeps the $v^+$-component
of~$x$ fixed at zero until the solution has entered the set~$C_\omega
\cap \omega^+$. This leads to the ``running start'' of solutions
originating at points with~$x_{v^+} = 0$, i.e., to the significant
velocity changes which can be observed in all vertices except at the
top right vertex in Figure~\ref{fig:semiflowa}. Notice further that
the above definition of the flow on arrow flow tiles rules out
any equilibria in~$C_\omega$. In fact, we will see in the next
section that the largest invariant subset of arrow flow tiles
is the empty set. Finally, we would like to note that solutions
through a point~$x \in X$ with $0 < x_u \le \epsilon$ for some
vertex $u \not\in \omega^+$ will always lead to the
identity~$x_u = 0$ in finite forward time, even if
the solution leaves~$C_\omega$ before this happens.
\subsection{Semiflows Induced by the Individual Vector Fields}
\label{sec52}
We now turn our attention to showing that the vector fields defined
in the previous section do in fact generate a continuous, strongly
admissible semiflow on the underlying polytope~$X \subset \R^d$
of the simplicial complex~$\cX$. As a first step, in the present section
we let~$\omega \in \cX$ denote an arbitrary but fixed simplex, and
we show that for every $\xi \in \R^d$ the initial value problem
\begin{equation} \label{def:ivpfomega1}
  \dot{x} = f^\omega(x)
  \qquad\mbox{ with }\qquad
  x(0) = \xi
\end{equation}
has a unique solution~$\phi^\omega(\cdot,\xi) : \R_0^+ \to \R^d$,
which depends continuously on the initial condition~$\xi$.
In this context, a function~$\nu : [0,T) \to \R^d$ is called a
{\em solution of the initial value problem~(\ref{def:ivpfomega1})\/}
if~$\nu$ is continuous and satisfies the integral identity
\begin{equation} \label{def:ivpfomega2}
  \nu(t) = \xi + \int_0^t f^\omega(\nu(\tau)) \, d\tau
  \qquad\mbox{ for all }\qquad
  0 \le t < T .
\end{equation}
We would like to point out that due to the fact that~$f^\omega$ is
not everywhere continuous, we cannot directly apply standard
existence and uniqueness results. We begin with a simple lemma
which allows us to solve part of the system~(\ref{def:ivpfomega1}).
\begin{lemma}[Solution of the Decoupled Components]
\label{lem:ivpdecoupled}
Consider the scalar function~$g$ defined in~(\ref{def:vectorfield1}).
Then for every $\zeta \in \R$ the initial value problem
\begin{equation} \label{lem:ivpdecoupled1}
  \dot{s} = -g(s)
  \qquad\mbox{ with }\qquad
  s(0) = \zeta
\end{equation}
has the unique forward solution
\begin{equation} \label{lem:ivpdecoupled2}
  \psi(t,\zeta) \; := \;
  \left\{ \begin{array}{ccl}
    \DS \mathrm{sgn} \, \zeta \left( \zeta^{2/3} - \frac{2}{3}
      \epsilon^{2/3} t \right)^{3/2} & \mbox{ for } &
      \DS 0 \le t \le \frac{3}{2} \left( \frac{\zeta}{\epsilon}
      \right)^{2/3} , \\[3ex]
    \DS 0 & \mbox{ for } & \DS t \ge \frac{3}{2}
      \left( \frac{\zeta}{\epsilon} \right)^{2/3}
      .
  \end{array} \right.
\end{equation}
Furthermore, the mapping $\psi : \R_0^+ \times \R \to \R$
is a continuous semiflow.
\end{lemma}
\begin{proof}
One can easily verify that the formula in~(\ref{lem:ivpdecoupled2})
is a differentiable solution of the one-dimensional initial value
problem~(\ref{lem:ivpdecoupled1}) which is defined on~$\R_0^+$.
Furthermore, since the right-hand side of this problem is continuously
differentiable on~$\R \setminus \{ 0 \}$, nonuniqueness can only
occur once the solution hits zero.

Suppose therefore that~$\nu : [t_0, T) \to \R$ is a solution of
$\dot{s} = g(s)$ with $\nu(t_0) = 0$. The result follows if we
can show that this implies $\nu(t) = 0$ for all $t \in [t_0,T)$.
For this, assume that there exists a $t_1\in(t_0,T)$ such that
$\nu(t_1)\neq 0$. We first consider the case $\nu(t_1) > 0$. Then
the supremum $t_* := \sup\{ 0 \le t \le t_1 \; \mid \; \nu(t) \le 0 \}$
exists and satisfies $t_* \in [t_0,t_1)$ and $\nu(t_*) = 0$,
as well as $\nu(t) > 0$ for all $t \in (t_*,t_1]$. This implies
\begin{displaymath}
  \nu(t_1) =
  \nu(t_*) - \int_{t_*}^{t_1} \underbrace{g(\nu(\tau))}_{\ge 0}
    \, d\tau \le
  0 + \int_{t_*}^{t_1} 0 \, d\tau = 0 ,
\end{displaymath}
a contradiction. The case $\nu(t_1) < 0$ can be treated
analogously, which proves the lemma.
\end{proof}
\medskip

The scalar differential equation discussed in the above lemma
is precisely the one that describes the evolution of the
$v$-components of the solution of~(\ref{def:ivpfomega1}) for
all $v \not\in \omega^+$ with $f_v^\omega$ defined
by~(\ref{def:vectorfield4}). Notice further that the
discontinuities of the vector field~$f^\omega$ are restricted
to the subspace of~$\R^d$ in which all of these components
are zero. Thus, we consider the decomposition $\R^d =
Y_\omega \oplus Z_\omega$ with
\begin{eqnarray}
  Y_\omega & := & \left\{ x \in \R^d \; \mid \; x_v = 0
    \mbox{ for all } v \not\in \omega^+ \right\}
    \quad\mbox{ and} \label{def:ivpfomega3} \\[1.5ex]
  Z_\omega & := & \left\{ x \in \R^d \; \mid \; x_v = 0
    \mbox{ for all } v \in \omega^+ \right\} ,
    \nonumber
\end{eqnarray}
and we define
\begin{equation} \label{def:ivpfomega4}
  \tau^\omega(x) \; := \;
  \max\left\{ \frac{3}{2} \left( \frac{x_v}{\epsilon}
    \right)^{2/3} \; \mid \; v \not\in \omega^+ \right\} .
\end{equation}
We point out that $\tau^\omega(x) = 0$ if and
only if we have $x \in Y_\omega$.
Thus, in view of
Lemma~\ref{lem:ivpdecoupled} these definitions imply that
a solution $x(t)$ to the initial
value problem~(\ref{def:ivpfomega1}) satisfies
\begin{displaymath}
  x(t) \not\in Y_\omega \;\mbox{ for all }\;
    0 \le t < \tau^\omega(\xi)
  \qquad\mbox{ and }\qquad
  x(t) \in Y_\omega \;\mbox{ for all }\;
    t \ge \tau^\omega(\xi) .
\end{displaymath}
This simple observation lies at the heart of the proof of the
following central result of this section.
\begin{proposition}[Existence of the Simplex-Induced Semiflow]
\label{prop:semiflowexsimplex}
For each $\xi \in \R^d$ a unique solution $\phi^\omega(\cdot,\xi) :
\R_0^+ \to \R^d$ of the initial value problem~(\ref{def:ivpfomega1}) exists.
It is differentiable everywhere with the exception of at most one point in
time, and~$\phi^\omega(t,\xi)$ is contained in the subspace~$Y_\omega$
defined in~(\ref{def:ivpfomega3}) if and only if~$t \ge \tau^\omega(\xi)$,
as introduced in~(\ref{def:ivpfomega4}). In particular, the
map $\phi^\omega : \R_0^+ \times \R^d \to \R^d$ is a well-defined
continuous semiflow for every simplex $\omega \in \cX$.
\end{proposition}
\begin{proof}
At first glance it seems impossible to apply standard existence and
uniqueness results for ordinary differential equations to the construction
of the semiflow~$\phi^\omega$. Note, however, that our discussion
leading up to the formulation of the proposition pointed out two
major points. On the one hand, the anticipated dynamics is divided into
two clear-cut regimes --- one outside the subspace~$Y_\omega$, and one
inside of it. On the other hand, the discontinuity of the vector field
and its accompanying velocity jumps happen only upon entering this
subspace. This allows us to construct the semiflow in three stages.

\smallskip\noindent
{\em (i) Semiflow inside the subspace~$Y_\omega$.\/}
We begin by considering only initial conditions~$\xi \in Y_\omega$.
On this subspace, the $v$-components~$f_v^\omega$ of the vector
field vanish for all~$v \not\in \omega^+$, and therefore we have
$f^\omega(x) \in Y_\omega$ for all $x \in Y_\omega$. This
immediately implies that the solution of the initial value
problem~(\ref{def:ivpfomega1}) satisfies the reduced differential
equation
\begin{displaymath}
  \begin{array}{rclcl}
    \DS \dot{x}_{v^+} & = & \DS h(x_{v^+}) + \theta^\omega(x)
      & \mbox{ for } & \{ v^+ \} = \omega^+ \setminus \omega^-
      , \\[2ex]
    \DS \dot{x}_v & = & \DS x_v - \frac{1}{\card{\omega^-}}
      \left( \sum_{u \in \omega^-} x_u
      + f^\omega_{v^+}(x) \right)
      & \mbox{ for } & v \in \omega^- .
  \end{array}
\end{displaymath}
The right-hand side of this system is clearly globally Lipschitz
continuous on~$Y_\omega$, since in view of~(\ref{def:vectorfield3})
the prefactor~$\eta^\omega(x_{v^+},x)$ reduces to~$1$. Thus,
standard existence and uniqueness results for ordinary differential
equations imply that the solutions of this system generate a
continuous flow $\Phi^\omega : \R \times Y_\omega \to Y_\omega$.

\smallskip\noindent
{\em (ii) Semiflow outside the subspace~$Y_\omega$.\/}
We now turn our attention to the semiflow outside~$Y_\omega$.
To study this case we need a new vector field~$\bar{f}^\omega:\R^d\to\R^d$
which is a slight modification of the vector field~\eqref{def:vectorfield4}.
It is given by the formula
\begin{equation} \label{def:vectorfield4-bar}
  \bar{f}^\omega_v(x) := \left\{
  \begin{array}{rcl}
    \DS -g(x_v)
      & \mbox{ for } & v \not\in \omega^+ , \\[2ex]
    \DS \min\left\{ 1, \, \frac{4}{\epsilon} \left| x_{v}
      \right|\right\} \left( h(x_v) + \theta^\omega(x) -
      \sum_{u \not\in \omega^+} x_u \right)
      & \mbox{ for } & v \in \omega^+ \setminus \omega^- , \\[5ex]
    \DS x_v - \frac{1}{\card{\omega^-}} \left( \sum_{u \in \omega^-} x_u
      + \sum_{u \not\in \omega^-} \bar{f}^\omega_u(x) \right)
      & \mbox{ for } & v \in \omega^- .
  \end{array}
  \right.
\end{equation}
Essentially, the only change is the first factor in the product
defining~$\bar{f}^\omega_v(x)$ for the unique vertex $v\in \omega^+
\setminus \omega^-$ plus the resulting modification in the second
sum of the formula defining~$\bar{f}^\omega_v(x)$ for $v\in \omega^-$.
In particular, $\bar{f}^\omega$ coincides with~$f^\omega$ on the
set~$\R^d \setminus Y_\omega$ and one easily verifies that
\begin{displaymath}
  \bar{f}^\omega(x) \; = \;
  \lim_{Y_\omega \not\ni y \to x} f^\omega(y).
\end{displaymath}
While the new vector field is continuous, it is still not Lipschitz
continuous due to the presence of the root function~$g$ in the
$v$-components for $v \not\in \omega^+$. Notice, however, that
these components are completely decoupled from the rest of the
vector field. Thus, the part of the initial value problem
\begin{equation} \label{prop:semiflowexsimplex1}
  \dot{x} = \bar{f}^\omega(x)
  \qquad\mbox{ with }\qquad
  x(0) = \xi
\end{equation}
which corresponds to $v$-components with $v \not\in \omega^+$
can be solved individually using Lemma~\ref{lem:ivpdecoupled}.
For the initial condition $\xi^z \in Z_\omega$ this leads to the semiflow
$\Psi : \R_0^+ \times Z_\omega \to Z_\omega$ given by
\begin{displaymath}
  \Psi_v\left( t, \xi^z \right) \; = \; \psi(t, \xi_v)
  \quad\mbox{ for all }\quad
  v \not\in \omega^+ ,
\end{displaymath}
where~$\psi$ is defined in~(\ref{lem:ivpdecoupled2}). Having
solved for the solution components in~$Z_\omega$, one can now
see that the initial value
problem~(\ref{prop:semiflowexsimplex1}) is equivalent
to solving the nonautonomous ordinary differential equation
problem
\begin{equation} \label{prop:semiflowexsimplex2}
  \dot{y} = \bar{f}^{\omega,y}(y + \Psi(t,\xi^z))
  \qquad\mbox{ with }\qquad
  y(0) = \xi^y ,
\end{equation}
where we decompose the vector field in the form $\bar{f}^\omega(x)
= \bar{f}^{\omega,y}(x) + \bar{f}^{\omega,z}(x) \in Y_\omega
\oplus Z_\omega$. Note that~(\ref{prop:semiflowexsimplex2}) is
an initial value problem which is only defined for $t \ge 0$,
because $\Psi(t,\xi^z)$ is defined only for $t\ge 0$.
Moreover, it depends both on the initial value~$\xi^y \in Y_\omega$
and on the parameter $\xi^z \in Z_\omega$. In addition, the
right-hand side of the nonautonomous differential equation is
continuous with respect to~$(t,y,\xi^z)$, as well as globally
Lipschitz continuous with respect to~$y$ for every fixed
combination of~$t$ and~$\xi^z$. For example, the
$v^+$-component of the right-hand side for $\{ v^+ \} = \omega^+
\setminus \omega^-$ is given explicitly by
\begin{equation} \label{prop:semiflowexsimplex3}
  \bar{f}^\omega_{v^+}(y + \Psi(t,\xi^z)) \; = \;
  \min\left\{ 1, \, \frac{4}{\epsilon} \left| y_{v^+} \right|
    \right\} \left( h(y_{v^+}) + \theta^\omega(y + \Psi(t,\xi^z)) -
    \sum_{u \not\in \omega^+} \psi(t,\xi_u) \right) ,
\end{equation}
and the nondifferentiable functions~$g$ appear only in
the $v$-components of the vector field for the vertices
$v \in \omega^-$ and are evaluated at the
functions~$\psi(t,\xi_v)$.

While recasting~(\ref{prop:semiflowexsimplex1}) in the
form~(\ref{prop:semiflowexsimplex2}) might seem a technicality
at first, the nonautonomous form of the new equation isolates
the non-Lipschitz part of the vector field~$\bar{f}^\omega$
in the $t$-dependent part. This approach only works because
we can solve for the $Z_\omega$-component of the solution
ahead of time and independently from the rest. Furthermore,
the nonautonomous parameter-dependent initial value
problem~(\ref{prop:semiflowexsimplex2}) satisfies all the
assumptions of~\cite[Theorem~2.4]{aulbach:wanner:96a}.
This implies the existence of a unique
solution~$\Xi(\cdot,\xi^y,\xi^z) : \R_0^+ \to Y_\omega$
of~(\ref{prop:semiflowexsimplex2}), and it also shows that the
map~$\Xi : \R_0^+ \times Y_\omega \times Z_\omega \to Y_\omega$
is continuous with respect to all variables. Finally, the
mapping~$\Pi:(t,\xi^y,\xi^z)\mapsto\left(\Xi(t,\xi^y,\xi^z),
\Psi(t,\xi^z)\right)$ satisfies~(\ref{prop:semiflowexsimplex1})
for the initial condition~$\xi = \xi^y + \xi^z$, and this
solution is continuously differentiable on~$\R_0^+$.
Moreover, since \eqref{prop:semiflowexsimplex1} is autonomous,
$\Pi$ is a semiflow on~$\R^d = Y_\omega \oplus Z_\omega \cong
Y_\omega \times Z_\omega$.

\smallskip\noindent
{\em (iii) Constructing the combined semiflow.\/}
While the solution~$\Xi$ constructed in the last part solves
the initial value problem~(\ref{prop:semiflowexsimplex1}) for
all times $t \ge 0$, this initial value problem is different
from the one we set out to solve. Based on the discussion
leading up to this proposition, we can however use it to find
the solution to~(\ref{def:ivpfomega1}). For this, recall the
definition of the time~$\tau^\omega(\xi)$ in~(\ref{def:ivpfomega4}).
Due to the specific form of~$f^\omega$, any solution
of~(\ref{def:ivpfomega1}) has to solve~(\ref{prop:semiflowexsimplex1})
on the interval~$[0,\tau^\omega(\xi)]$, and it has to satisfy the
autonomous differential equation studied in~{\em (i)\/} on the
interval~$[\tau^\omega(\xi),\infty)$. Thus, the unique forward
solution of~(\ref{def:ivpfomega1}) is given by the composition
\begin{equation} \label{prop:semiflowexsimplex4}
  \phi^\omega(t,\xi) \; = \;
  \Phi^\omega \left( \max\left\{ 0, t - \tau^\omega(\xi) \right\} ,
    \; \Xi\left( \min\left\{ t, \tau^\omega(\xi) \right\} ,
    \xi^y , \xi^z \right) \right) + \Psi (t, \xi^z) .
\end{equation}
According to the previous two parts of this proof, the mapping
$\phi^\omega : \R_0^+ \times \R^d \to \R^d$ is continuous. In
addition, the solution~$\phi^\omega(\cdot,\xi)$ is differentiable
everywhere except possibly at~$t = \tau^\omega(\xi)$, and it
satisfies~(\ref{def:ivpfomega2}). Since the vector field~$f^\omega$
is autonomous, this last fact together with the fact
that~$\phi^\omega(\cdot,\xi)$ is the unique forward solution
of~(\ref{def:ivpfomega1}) finally shows that~$\phi^\omega$ is
a continuous semiflow, which concludes the proof of the
proposition.
\end{proof}
\medskip

The above result shows that the vector field~$f^\omega$ generates
a continuous semiflow on~$\R^d$, despite its discontinuities. All
of its solutions are uniquely defined in forward time, but they can
merge in finite time.

For us, the semiflows~$\phi^\omega : \R_0^+ \times \R^d \to \R^d$
are just a first step towards the construction of a strongly admissible
semiflow on the underlying polytope~$X \subset \R^d$ of the given
simplicial complex~$\cX$, and with respect to the combinatorial vector
field~$\cV$. As a second step, we need to show that~$\phi^\omega$
leaves appropriate parts of~$X$ invariant. This
is the subject of the following corollary.
\begin{corollary}[Relative Forward Invariance of the Semiflow on Flow Tiles]
\label{cor:semiflowexsimplex}
Let $\omega \in \cX$ be an arbitrary simplex  and consider the
semiflow~$\phi^\omega$ on~$\R^d$ guaranteed by
Proposition~\ref{prop:semiflowexsimplex}.
Consider the flow tile associated with~$\omega$, that is
\begin{displaymath}
  C_\omega \; = \;
  \cl\cse{\omega^-} \cup \cl\cse{\omega^+}
  \; \subset \; X \; \subset \; \R^d
\end{displaymath}
as introduced in~(\ref{def:flowtiles2}). Then for every~$\xi \in C_\omega$
the solution~$\phi^\omega(\cdot,\xi)$ stays in~$C_\omega$ until
it reaches a point~$\xi^* \in C_\omega$ which satisfies
$\xi_v^* = \epsilon$ for at least one vertex $v \in \omega^-$.
Such a point necessarily lies on the boundary of the flow
tile~$C_\omega$.
\end{corollary}
\begin{proof}
The semiflow~$\phi^\omega$ constructed in
Proposition~\ref{prop:semiflowexsimplex} generates a semiflow
on the whole space~$\R^d$, and we have already seen that
the underlying polytope~$X$ is only a subset of~$\R^d$ of
measure zero. In view of Lemma~\ref{lem:propepscells}
and~\eqref{def:flowtiles2} we can characterize the flow
tile~$C_\omega$ as the set of vectors $x\in\R^d$
satisfying the following four conditions
\begin{itemize}
    \item[(a)] $\sum_{v \in \cX_0} x_v = 1$ ,
    \item[(b)] $0 \le x_v \le \epsilon$ for all $v \not\in \omega^+$,
    \item[(c)] $x_v \ge 0$ for $v \in \omega^+ \setminus \omega^-$,
    \item[(d)] $x_v \ge \epsilon$ for all $ v \in \omega^-$ .
\end{itemize}
Thus, in order to establish the corollary we only have to
show that a solution~$\phi^\omega(\cdot,\xi)$ originating at
$\xi \in C_\omega$ cannot exit this flow tile by violating
conditions~(a), (b), or~(c).
We verify this claim for each condition separately.

\smallskip\noindent
{\em (i) Along the solution, condition~(a) cannot be violated.\/}
One can easily see that for every point $x \in \R^d$ we have the
identity
\begin{equation}
\label{eq:sum-f-omega-v}
\begin{aligned}
  \sum_{v \in \cX_0} f^\omega_v(x) & =
    \sum_{v \not\in \omega^-} f^\omega_v(x) +
    \sum_{v \in \omega^-} \left( x_v - \frac{1}{\card{\omega^-}}
      \left( \sum_{u \in \omega^-} x_u + \sum_{u \not\in \omega^-}
      f^\omega_u(x) \right) \right) \\
  & =  \sum_{v \not\in \omega^-} f^\omega_v(x) +
    \sum_{v \in \omega^-} x_v -
      \left( \sum_{u \in \omega^-} x_u + \sum_{u \not\in \omega^-}
      f^\omega_u(x) \right) \; = \; 0 .
\end{aligned}
\end{equation}
Since every solution $\phi^\omega(\cdot,\xi)$ of \eqref{def:ivpfomega1}
satisfies the integral equality \eqref{def:ivpfomega2}, we obtain
from~\eqref{eq:sum-f-omega-v} that $\sum_{v \in \cX_0} \phi^\omega_v(t,\xi)$
is independent of~$t$ for $t\geq 0$. In particular, for any $\xi \in C_\omega$
the sum of the components of a solution  has to remain equal to one for all
times~$t \ge 0$.

\smallskip\noindent
{\em (ii) Along the solution, condition~(b) cannot be violated.\/}
In view of $\xi_v \in [0,\epsilon]$ for all $v \not\in \omega^+$,
the formulas~(\ref{prop:semiflowexsimplex4})
and~(\ref{lem:ivpdecoupled2}) immediately imply
that~$\phi^\omega_v(t,\xi) \in [0,\epsilon]$
for all $t \ge 0$ and $v \not\in \omega^+$.

\smallskip\noindent
{\em (iii) Along the solution, condition~(c) cannot be violated.\/}
Recall that $\omega^+ \setminus \omega^-$ is either empty or a singleton.
If it is empty, there is nothing to be verified.
Thus, assume that $v^+$ is the unique vertex in $\omega^+ \setminus \omega^-$.
According to the proof of Proposition~\ref{prop:semiflowexsimplex},
the system~(\ref{prop:semiflowexsimplex1}), which was solved in the
nonautonomous form~(\ref{prop:semiflowexsimplex2}), has the invariant
hyperplane
\begin{displaymath}
  H_\omega = \left\{ x \in \R^d \; \mid \; x_{v^+} = 0 \right\} .
\end{displaymath}
In other words, the
solution~$\Xi$ preserves the inequality $x_{v^+} \ge 0$.
Once the solution hits the set~$Y_\omega \cap C_\omega$, part~{\em (i)\/}
of the proof of Proposition~\ref{prop:semiflowexsimplex} shows
that~$\dot{x}_{v^+} > 0$, i.e., the $v^+$-component
of~$\phi^\omega(\cdot,\xi)$ becomes strictly positive. This
proves the result.
\end{proof}
\subsection{Dynamics of the Individual Semiflows on Flow Tiles}
\label{sec53}
In the last section we have constructed a semiflow~$\phi^\omega$
on~$\R^d$ for every simplex $\omega \in \cX$. We have also seen
that the associated flow tile $C_\omega = \cl\cse{\omega^-} \cup
\cl\cse{\omega^+}$ is forward invariant under this semiflow,
until a solution reaches a well-defined subset of its boundary
in~$X$. But how exactly this boundary is reached, what the vector
field~$f^\omega$ looks like on the boundary, and what other properties
forward solutions of~$\phi^\omega$ have in~$C_\omega$ has been left
unexplored. This gap is closed in the present section.

We begin our discussion with the behavior of~$\phi^\omega$ on the
flow tile~$C_\omega$ near its boundary. Recall that according to
Definition~\ref{def:charactcells} and Lemma~\ref{lem:altcharepscell}
a point~$x \in X$ lies on the boundary of at least two $\epsilon$-cells,
if there exists at least one vertex~$v \in \cX_0$ such that
$x_v = \epsilon$. In fact, the set of all vertices~$v$ for
which this identity is satisfied is given by the vertices in
$\sigma^\epsilon_{\max}(x) \setminus \sigma^\epsilon_{\min}(x)$.
In order to guarantee that~$x$ lies on the boundary of at least two
flow tiles, one needs the somewhat stronger condition that there exists
at least one vertex~$v \in \cX_0$ such that $x_v = \epsilon$ and $v \not\in
\omega^+ \setminus \omega^-$. With these observations in mind, we obtain
the following result in which, as everywhere in this section,
$d := \card{\cX_0}$ and $X \subset \R^d$  denotes
the underlying polytope of the standard geometric realization of~$\cX$.
\begin{proposition}[Vector Field Bounds near Flow Tile Boundaries]
\label{prop:vectorfieldbound}
Consider the vector fields~$f^\omega$
defined in~(\ref{def:vectorfield1}) through~(\ref{def:vectorfield4}),
the flow tiles~$C_\omega$ defined in~(\ref{def:flowtiles2}), and
suppose that \eqref{prop:vectorfieldbound1} is satisfied.
Then for every simplex $\omega \in \cX$ the following is true.
\begin{itemize}
\item[(a)] For all $v \in \omega^-$ and all $x \in C_\omega$ with
$|x_v - \epsilon| \le \epsilon$ we have $f^\omega_v(x) \le -1/(4d) < 0$.
\item[(b)] For all $v \not\in \omega^+$ and all $x \in C_\omega$
with $|x_v - \epsilon| \le \epsilon/2$ we have $f^\omega_v(x)
\le -\epsilon/2 < 0$.
\item[(c)] If $\omega^+ \neq \omega^-$ and $\{ v^+ \} = \omega^+
\setminus \omega^-$, then for all $x \in C_\omega$ with
$|x_{v^+} - \epsilon| \le \epsilon^2/(8+4\epsilon)$ and for which
there is a vertex $v \neq v^+$ with $|x_v - \epsilon| \le \epsilon/8$
we have $f^\omega_{v^+}(x) \le -\epsilon/8 < 0$.
\end{itemize}
\end{proposition}
\begin{proof}
The proof of the proposition is divided into four separate parts.

\smallskip\noindent
{\em (i) Verification of~(a) for critical cells.\/}
Suppose first that the simplex~$\omega \in \cX$ is a critical
cell and let $v \in \omega$ be a fixed vertex. Furthermore,
let~$x \in C_\omega$ be given with $|x_v - \epsilon| \le \epsilon$.
Then Lemma~\ref{lem:propepscells} implies $\epsilon \le x_v \le
2\epsilon$, as well as $0 \le x_u \le \epsilon$ for all vertices
$u \not\in \omega$. The definition of the function~$g$ then further
yields $0 \le g(x_u) \le \epsilon$ for all $u \not\in \omega$.
Let $m := \#\omega$. We deduce from both $d = \card{\cX_0}$
and $\sum_{u \in \cX_0} x_u = 1$, in combination
with~(\ref{prop:vectorfieldbound1}), the estimate
\begin{eqnarray*}
  f_v^\omega(x) & = & x_v - \frac{1}{\card{\omega}} \left(
    1 - \sum_{u \not\in \omega} x_u + \sum_{u \not\in \omega}
    f_u^\omega(x) \right)
  \; \le \;
    2\epsilon - \frac{1}{m} +
    \frac{1}{m} \sum_{u \not\in \omega} x_u +
    \frac{1}{m} \sum_{u \not\in \omega} g(x_u) \\[2ex]
  & \le & 2\epsilon - \frac{1}{m} +
    \frac{(d-m)\epsilon}{m} +
    \frac{(d-m)\epsilon}{m}
  \; = \;
    -\frac{1}{m} + \frac{2d \epsilon}{m}
  \; < \; -\frac{2}{3m}
  \; \le \; -\frac{2}{3d} \;< \; -\frac{1}{4d} ,
\end{eqnarray*}
which proves~{\em (a)\/} for critical cells.

\smallskip\noindent
{\em (ii) Verification of~(a) for arrow cells.\/}
Suppose that the simplex~$\omega \in \cX$ is part of an arrow,
i.e., we have $\omega^+ \neq \omega^-$. Let $v \in \omega^-$ be a
fixed vertex, and let~$x \in C_\omega$ be given with $|x_v - \epsilon|
\le \epsilon$. Then Lemma~\ref{lem:propepscells} implies again
$\epsilon \le x_v \le 2\epsilon$, as well as $0 \le x_u \le \epsilon$
and $0 \le g(x_u) \le \epsilon$ for all vertices $u \not\in \omega^+$.
Setting $m := \#\omega^-$, we now deduce from both $d = \card{\cX_0}$
and $\sum_{u \in \cX_0} x_u = 1$, in combination
with~(\ref{prop:vectorfieldbound1}) and $\{ v^+ \} = \omega^+
\setminus \omega^-$, the estimate
\begin{eqnarray}
  f_v^\omega(x) & = & x_v - \frac{1}{m} \left(
    \sum_{u \in \omega^-} x_u + \sum_{u \not\in \omega^-}
    f^\omega_u(x) \right)
  \; \le \;
    2\epsilon + \frac{1}{m} \left( -1 +
    \sum_{u \not\in \omega^-} x_u - \sum_{u \not\in \omega^-}
    f^\omega_u(x) \right) \nonumber \\[1.5ex]
  & = &
    2\epsilon + \frac{1}{m} \left( x_{v^+} - f^\omega_{v^+}(x) - 1 +
    \sum_{u \not\in \omega^+} x_u - \sum_{u \not\in \omega^+}
    f^\omega_u(x) \right) \nonumber \\[1.5ex]
  & \le &
    2\epsilon + \frac{1}{m} \left( x_{v^+} - f^\omega_{v^+}(x) - 1
    \right) + \frac{2 (d-m-1) \epsilon}{m} \nonumber \\[1.5ex]
  & < &
    \frac{2d\epsilon}{m} + \frac{1}{m}
    \left( x_{v^+} - f^\omega_{v^+}(x) - 1 \right)
  \; < \;
    \frac{1}{3m} + \frac{1}{m}
    \left( x_{v^+} - f^\omega_{v^+}(x) - 1 \right) .
    \label{prop:vectorfieldbound2}
\end{eqnarray}
We now turn our attention to the term in parentheses
in~(\ref{prop:vectorfieldbound2}). Due to $x \in C_\omega$,
Lemma~\ref{lem:propepscells} implies $x_w \ge \epsilon$
for all $w \in \omega^-$, and the definition of~$\theta^\omega(x)$
in~(\ref{def:vectorfield2}) then yields $\theta^\omega(x) \ge 0$.
Furthermore, in view of~(\ref{def:vectorfield3}) we have $0 \le
\eta^\omega(x_{v^+},x) \le 1$. Now the second equation
in~(\ref{def:vectorfield4}) implies
\begin{eqnarray*}
  -f^\omega_{v^+}(x) & = &
    -\eta^\omega(x_{v^+},x) h(x_{v^+})
    -\eta^\omega(x_{v^+},x) \theta^\omega(x)
    +\eta^\omega(x_{v^+},x) \sum_{u \not\in \omega^+} x_u \\
  & \le & -\eta^\omega(x_{v^+},x) h(x_{v^+})
    + \sum_{u \not\in \omega^+} x_u
  \; \le \;
    -\eta^\omega(x_{v^+},x) h(x_{v^+}) + (d-m-1) \epsilon \\
  & < & \frac{1}{6} - \eta^\omega(x_{v^+},x) h(x_{v^+}) ,
\end{eqnarray*}
which together with~(\ref{prop:vectorfieldbound2}) gives
\begin{equation} \label{prop:vectorfieldbound3}
  f_v^\omega(x) \; < \;
  \frac{1}{2m} + \frac{1}{m}
    \left( x_{v^+} - \eta^\omega(x_{v^+},x) h(x_{v^+}) - 1
    \right) .
\end{equation}
A glance at the graph of $\eta^\omega(x_{v^+},x) h(x_{v^+})$ in
the right panel of Figure~\ref{fig:ghfunction}, which could either
be the blue or the red curve, readily shows that for~$0 \le x_{v^+}
\le 1$ the distance between~$\eta^\omega(x_{v^+},x) h(x_{v^+})$
and~$x_{v^+} - 1$ is minimal for $x_{v^+} = \epsilon$, and one
obtains
\begin{displaymath}
  \eta^\omega(x_{v^+},x) h(x_{v^+}) - \left( x_{v^+} - 1 \right)
  \; \ge \; 1 - \frac{3\epsilon}{2} \; > \;
  \frac{3}{4}
  \quad\mbox{ for all }\quad
  0 \le x_{v^+} \le 1 ,
\end{displaymath}
since we have $\epsilon < 1 / (6d) \le 1/6$. Together
with~(\ref{prop:vectorfieldbound3}) we finally get the
estimate
\begin{displaymath}
  f_v^\omega(x) \; < \;
  \frac{1}{2m} + \frac{1}{m}
    \underbrace{\left( x_{v^+} - \eta^\omega(x_{v^+},x) h(x_{v^+}) - 1
    \right)}_{< -3/4} \; < \;
  -\frac{1}{4m} \; \le \; -\frac{1}{4d} ,
\end{displaymath}
which completes the proof of {\em (a)\/} for arrow cells.

\smallskip\noindent
{\em (iii) Verification of~(b).\/}
For arbitrary $x \in C_\omega$ and $v \not\in \omega^+$ we have
$0 \le x_v \le \epsilon$, i.e., the assumption in~{\em (b)\/}
implies $\epsilon/2 \le x_v \le \epsilon$. Together
with~(\ref{def:vectorfield1}) and the first equation
in~(\ref{def:vectorfield4}) one then obtains
$f_v^\omega(x) = -g(x_v) \le -g(\epsilon/2) < -\epsilon/2$.

\smallskip\noindent
{\em (iv) Verification of~(c).\/}
Finally, suppose that~$\omega^- \neq \omega^+$ and $\{ v^+ \} =
\omega^+ \setminus \omega^-$. Let~$x \in C_\omega$ be arbitrary
with $|x_{v^+} - \epsilon| \le \epsilon^2/(8+4\epsilon)$. Then
the definitions of~$h$ and~$\eta^\omega$ in~(\ref{def:vectorfield1})
and~(\ref{def:vectorfield2}), respectively, imply both $h(x_{v^+})
\le -\epsilon/4$ and $\eta^\omega(x_{v^+},x) = 1$, where for the
latter identity we use the inequality $\epsilon^2/(8+4\epsilon)
< \epsilon/2$. According to~(\ref{def:vectorfield4}) these
statements imply
\begin{equation} \label{prop:vectorfieldbound4}
  f_{v^+}^\omega(x) \; = \;
  \eta^\omega(x_{v^+},x) \left( h(x_{v^+}) + \theta^\omega(x)
    -\sum_{u \not\in \omega^+} x_u \right) \; < \;
  -\frac{\epsilon}{4} + \theta^\omega(x)
    -\sum_{u \not\in \omega^+} x_u .
\end{equation}
Consider first the case when there exists a vertex $v \not\in \omega^+$
such that $|x_v - \epsilon| \le \epsilon/8$. Then
Lemma~\ref{lem:propepscells} yields $7\epsilon/8 \le x_v
\le \epsilon$, and~(\ref{prop:vectorfieldbound4}) implies
in combination with $\theta^\omega(x) \le \epsilon$ the
estimate
\begin{displaymath}
  f_{v^+}^\omega(x) \; < \;
  -\frac{\epsilon}{4} + \theta^\omega(x) - x_v \; \le \;
  -\frac{\epsilon}{8} ,
\end{displaymath}
which establishes~{\em (c)\/}. In the other case we
have $|x_u - \epsilon| > \epsilon/8$ for all $u \not\in
\omega^+$. Then, there has to be a vertex $v \in \omega^-$
for which $|x_v - \epsilon| \le \epsilon/8$. In view of
Lemma~\ref{lem:propepscells} this furnishes the inequalities
$\epsilon \le x_v \le 9\epsilon/8$, and therefore we have
$\theta^\omega(x) \le \epsilon/8$. Now~(\ref{prop:vectorfieldbound4})
implies
\begin{displaymath}
  f_{v^+}^\omega(x) \; < \;
  -\frac{\epsilon}{4} + \frac{\epsilon}{8}
    -\sum_{u \not\in \omega^+} x_u \; \le \;
  -\frac{\epsilon}{8} ,
\end{displaymath}
since $x_u \ge 0$ for all $u \not\in \omega^+$.
Thus, also in the other case property (c) holds.
This completes the proof of the proposition.
\end{proof}
\medskip

For later reference, we formulate an easy corollary of this result,
which describes the vector field behavior on the boundary of flow
tiles.
\begin{corollary}[Vector Field Direction along Flow Tile Boundaries]
\label{cor:vectorfieldbound}
Consider the vector fields~$f^\omega$
defined in~(\ref{def:vectorfield1}) through~(\ref{def:vectorfield4}),
the flow tiles~$C_\omega$ defined in~(\ref{def:flowtiles2}), and
suppose that~(\ref{prop:vectorfieldbound1}) holds. Then for every
simplex $\omega \in \cX$ and every $x \in C_\omega$ which lies on
the boundary of~$C_\omega$ in~$X$, the vector~$f^\omega(x)$ points into the interior
of the flow tile~$C_{\sigma^\epsilon_{\min}(x)}$, while~$-f^\omega(x)$
points into the interior of~$C_{\sigma^\epsilon_{\max}(x)}$, where
$\sigma^\epsilon_{\min}(x)$ is given by~(\ref{def:charactcells1})
and $\sigma^\epsilon_{\max}(x)$ by~(\ref{def:charactcells2}).
\end{corollary}
\begin{proof}
Suppose that $x \in C_\omega$ lies on the boundary of two flow
tiles. Then there is at least one vertex $v \in \cX_0$ for which
$x_v = \epsilon$, and in fact all vertices for which this identity
holds are  the vertices in $V_x := \sigma^\epsilon_{\max}(x)
\setminus \sigma^\epsilon_{\min}(x)$. In addition, if $\omega^+ \neq
\omega^-$ and if~$V_x$ contains the unique vertex~$v^+$ in
$\omega^+ \setminus \omega^-$, then~$V_x$ has to contain at least
one more vertex $v \neq v^+$. Using Proposition~\ref{prop:vectorfieldbound}
one can then show that
\begin{displaymath}
  f_v^\omega(x) < 0
  \qquad\mbox{ for all }\qquad
  v \in V_x = \sigma^\epsilon_{\max}(x)
    \setminus \sigma^\epsilon_{\min}(x) ,
\end{displaymath}
which establishes the corollary.
\end{proof}
\medskip

Proposition~\ref{prop:vectorfieldbound} describes in detail the behavior
of solutions of~$\phi^\omega$ in the flow tile~$C_\omega$ near its boundary.
As we will see in the next section, this result is crucial both for the
definition of the final semiflow on~$X$, as well as for its admissibility.
In contrast, the next result will be used for establishing strong admissibility.
\begin{proposition}[Solution Exit from Arrow Flow Tiles]
\label{prop:arrowtileexit}
Consider the vector fields~$f^\omega$
defined in~(\ref{def:vectorfield1}) through~(\ref{def:vectorfield4}),
as well as the associated semiflow~$\phi^\omega : \R_0^+ \times
\R^d \to \R^d$ guaranteed by Proposition~\ref{prop:semiflowexsimplex},
and suppose that~(\ref{prop:vectorfieldbound1}) holds.
Let $\omega \in \cX$ be a simplex which is part of an arrow
of the combinatorial vector field~$\cV$, and let~$x \in C_\omega$ be
contained in the associated flow tile as defined in~(\ref{def:flowtiles2}).
Then the forward solution of~$\phi^\omega$ which originates in~$x$ exits the
tile~$C_\omega$ in finite forward time, and every solution through~$x$
which exists for all negative times exits the flow tile~$C_\omega$
in finite backward time.
\end{proposition}
\begin{proof}
We argue by contradiction. Suppose that there exists a solution
of~$\phi^\omega$ which stays in the compact set~$C_\omega$ for all
$t \ge 0$ or for all $t \le 0$. Then its $\omega$- or its
$\alpha$-limit set has to be nonempty, and standard results
for semiflows imply that there exists a full solution
$\gamma : \R \to \R^d$ of~$\phi^\omega$ which stays in~$C_\omega$ for
all times, see for example~\cite[Proposition~2.6, p.~204]{diekmann:etal:95a}.
We will prove the following four statements:
\begin{itemize}
\item[{\em (i)\/}] For all $t \in \R$ we have
$\gamma(t) \in C_\omega \cap \omega^+$.
\item[{\em (ii)\/}] For every $v \in \omega^-$ the set
$F_v := \{ x \in C_\omega \cap \omega^+ \; \mid \; x_v \le 2\epsilon \}$
is positively invariant under~$\phi^\omega$ relative
to~$C_\omega \cap \omega^+$. Furthermore, every solution starting
in one of these sets will exit~$C_\omega$ in finite forward time.
\item[{\em (iii)\/}]
Every solution originating
in~$F^* := \{ x \in C_\omega \cap \omega^+ \; \mid \; x_w \ge 2\epsilon
  \;\mbox{ for all }\; w \in \omega^- \}$
has to enter a set~$F_v$
for some~$v \in \omega^-$ in finite forward time.
\item[{\em (iv)\/}]
Every solution originating in~$F^*$ has to
exit~$C_\omega$ in finite forward time.
\end{itemize}
\smallskip\noindent
{\em Proof of (i):\/} Let~$x \in \gamma(\R)$ be arbitrary and
suppose that there exists a vertex $u \not\in \omega^+$ such
that $x_u \neq 0$. Then according to~$x \in X$ we have $x_u > 0$,
and the definition of the vector field shows that~$\gamma_u$
solves the differential equation~$\dot{y} = -g(y)$ with the
positive initial value~$x_u$. Solutions of this initial value
problem are uniquely determined in backward time, and one can
easily see that they become unbounded as $t \to -\infty$, which
contradicts the fact that~$\gamma$ lies in the compact
set~$C_\omega$. Thus, we have to have~$x_u = 0$ for all
$u \not\in \omega^+$ and~{\em (i)\/} follows.

\smallskip\noindent
{\em Proof of (ii):\/} Let $v \in \omega^-$ be arbitrary but fixed.
Then Proposition~\ref{prop:vectorfieldbound}{\em (a)\/} immediately
implies
\begin{displaymath}
  f_v^\omega(x) \; \le \; -\frac{1}{4d} \; < \; 0
  \qquad\mbox{ for all }\qquad
  x \in F_v \subset C_\omega \cap \omega^+ .
\end{displaymath}
This leads to the following two
observations. On the one hand, since the semiflow~$\phi^\omega$
on~$C_\omega$ is generated by the vector field~$f^\omega$, it shows
that the $v$-component of any solution originating in~$F_v$
is decreasing, i.e., the solution stays in~$F_v$ for as long as it
stays in~$C_\omega \cap \omega^+$. On the other hand, since the
$v$-component has a velocity which is bounded away from zero, any
such solution has to reach the hyperplane $x_v = \epsilon$ in finite
forward time, at which point it will exit~$C_\omega$ due to
Proposition~\ref{prop:vectorfieldbound} --- unless of course the
solution exits earlier. This completes the proof of~{\em (ii)\/}.

\smallskip\noindent
{\em Proof of (iii):\/} Suppose that there exists a point
$x \in F^*$ with $f_{v^+}^\omega(x) \le \epsilon / 4$, where~$v^+$
denotes the unique vertex in~$\omega^+ \setminus \omega^-$. Then
the definition of~$f_{v^+}^\omega(x)$ in~(\ref{def:vectorfield4}),
together with the fact that for all $y \in C_\omega \cap \omega^+$
we have both $\eta^\omega(y_{v^+},y) = 1$ and $y_w = 0$ for all
$w \not\in \omega^+$, implies the estimate
\begin{displaymath}
  \frac{\epsilon}{4} \; \ge \;
  f_{v^+}^\omega(x) \; = \;
  h(x_{v^+}) + \theta^\omega(x) \; \ge \;
  -\frac{\epsilon}{2} + \theta^\omega(x) ,
  \quad\mbox{ and thus }\quad
  \theta^\omega(x) \; \le \; \frac{3\epsilon}{4} .
\end{displaymath}
According to~(\ref{def:vectorfield2}) this yields a vertex
$w \in \omega^-$ with $x_w - \epsilon \le 3\epsilon / 4 < \epsilon$,
i.e., one has to have the inequality $x_w \le 7\epsilon / 4 < 2\epsilon$.
This clearly contradicts our choice of $x \in F^*$.

In view of the last paragraph, we therefore have $f_{v^+}^\omega(x)
> \epsilon / 4$ for all $x \in F^*$. This in turn implies that any
solution of~$\phi^\omega$ which starts in~$F^*$ either has to reach one of the
sets~$F_v$ as desired, or its $v^+$-component has to reach a point
$y \in F^*$ with $y_{v^+} > 1 - 2\epsilon \cdot \card{\omega^-}$ in finite
forward time. Due to $y \in \omega^+$ this yields $\sum_{u \in \omega^-}
y_u = 1 - y_{v^+} < 2\epsilon \cdot \card{\omega^-}$, and therefore there has
to be a vertex $v \in \omega^-$ with $y_v < 2\epsilon$. This shows
that also in this case the solution enters a set~$F_v$. This completes
the proof of {\em (iii)\/}.

\smallskip\noindent
{\em Proof of (iv):\/} This is an immediate consequence of {\em~(ii)\/}
and {\em (iii)\/}.

Finally, since we clearly have $C_\omega \cap \omega^+ = F^* \cup
\bigcup_{v \in \omega^-} F_v$, we see that the statement~{\em (i)\/}
contradicts statements~{\em (ii)\/} and~{\em (iv)\/},
which in turn establishes the result.
\end{proof}
\medskip

To close this section we take a closer look at solutions of~$\phi^\omega$
which eventually exit the associated flow tile~$C_\omega$. While the
previous result demonstrates that every solution in an arrow flow tile
has to exit, every critical flow tile contains points~$x$ for which
the forward solution~$\phi^\omega(\R_0^+,x)$ is contained in~$C_\omega$.
Yet, as the following result shows, the set of initial conditions
which lead to an exit from the flow tile is always open in~$C_\omega$,
and the time it takes to exit~$C_\omega$ varies continuously with the
initial condition. This fact will be crucial in the next section.
\begin{lemma}[Continuity of the Exit Time from Flow Tiles]
\label{lem:exittime}
Consider the continuous
semiflows~$\phi^\omega:\R_0^+ \times \R^d \to \R^d$
guaranteed by Proposition~\ref{prop:semiflowexsimplex},
and suppose that~(\ref{prop:vectorfieldbound1}) holds.
For every simplex~$\omega \in \cX$ we define the exit time
\begin{equation} \label{lem:exittime1}
  T^\omega(x) \; := \;
  \inf\left\{ t > 0 \; \mid \; \phi^\omega(t,x) \not\in C_\omega
    \right\}
  \quad\mbox{ for all }\quad
  x \in C_\omega ,
\end{equation}
where~$C_\omega$ is defined in~(\ref{def:flowtiles2}) and the
infimum of the empty set is assumed to be~$+\infty$. In addition, we let
\begin{equation} \label{lem:exittime2}
  E_\omega \; := \;
  \left\{ x \in C_\omega \; \mid \; T^\omega(x) < +\infty \right\}
\end{equation}
denote the set of all initial conditions which lead to domain
exit from the flow tile~$C_\omega$. Then the set~$E_\omega$
is open in~$C_\omega$, and the
map $T^\omega : E_\omega \to \R_0^+$ is continuous. Finally,
$T^\omega(x) = 0$ if and only if there exists a
vertex $v \in \omega^-$ with $x_v = \epsilon$.
\end{lemma}
\begin{proof}
We first show that~$E_\omega$ is open in~$C_\omega$. For this,
let~$x \in E_\omega$ be arbitrary. Then there exists a time
$\tau > 0$ such that~$\phi^\omega(\tau,x) \not\in C_\omega$.
Since~$C_\omega$ is closed and~$\phi^\omega : \R_0^+ \times
\R^d \to \R^d$ is continuous, there exists an open
neighborhood~$U \subset \R^d$ of~$x$ such that for all points
$y \in U$ we have $\phi^\omega(\tau,y) \not\in C_\omega$.
This immediately implies $U \cap C_\omega \subset E_\omega$,
and establishes the openness of~$E_\omega$ in~$C_\omega$.
Moreover, the characterization of all~$x \in E_\omega$ which
satisfy $T^\omega(x) = 0$ follows directly from
Corollary~\ref{cor:semiflowexsimplex} and
Proposition~\ref{prop:vectorfieldbound}{\em (a)\/}.

Now let~$\delta > 0$ and $x \in E_\omega$ be arbitrary. Furthermore,
choose a time $\tau \in (T^\omega(x), T^\omega(x) + \delta)$ with
$\phi^\omega(\tau,x) \not\in C_\omega$. Then there exists an open
neighborhood~$U \subset \R^d$ of~$x$ such that for all points
$y \in U$ we have $\phi^\omega(\tau,y) \not\in C_\omega$.
This in turn implies $T^\omega(y) \le \tau < T^\omega(x) + \delta$
for all elements $y \in U \cap C_\omega$, and therefore $T^\omega :
E_\omega \to \R_0^+$ is upper semicontinuous.

Next we choose an $x \in E_\omega$ with $T^\omega(x) > 0$, and
we let $\delta\in(0,T^\omega(x)]$ be arbitrary. In view of
Proposition~\ref{prop:vectorfieldbound}{\em (a)\/} we have the
strict inequality $\phi_u^\omega(t,x) > \epsilon$ for all vertices
$u \in \omega^-$ and all times $t \in [0,T^\omega(x)-\delta]$.
The compactness of the latter interval and the continuity
of~$\phi^\omega$ then imply the existence of a
constant~$\rho > \epsilon$ such that
\begin{displaymath}
  \phi_u^\omega(t,x) \ge \rho > \epsilon
  \quad\mbox{ for all }\quad
  u \in \omega^-
  \;\mbox{ and }\;
  t \in [0,T^\omega(x)-\delta] .
\end{displaymath}
Since~$\phi^\omega$ is a continuous semiflow, there now
exists an $\eta > 0$ such that
\begin{displaymath}
  \left| \phi^\omega(t,y) - \phi^\omega(t,x) \right|
    \; < \; \rho - \epsilon
  \quad\mbox{ for all }\quad
  |y - x| < \eta
  \;\mbox{ and }\;
  t \in [0,T^\omega(x)-\delta] ,
\end{displaymath}
where~$|\cdot|$ denotes the Euclidean norm in~$\R^d$.
This uniform estimate implies for all $u \in \omega^-$, all
$t \in [0,T^\omega(x)-\delta]$, and all $y \in C_\omega$
with $|y - x| < \eta$ the bound
\begin{displaymath}
  \phi_u^\omega(t,y) \; = \;
  \left| \phi_u^\omega(t,y) \right| \; \ge \;
  \left| \phi_u^\omega(t,x) \right| -
    \left| \phi_u^\omega(t,x) - \phi_u^\omega(t,y) \right| \; > \;
  \rho - \left( \rho - \epsilon \right) \; = \; \epsilon .
\end{displaymath}
But then Corollary~\ref{cor:semiflowexsimplex} yields the
inclusion $\phi^\omega(t,y) \in C_\omega$ for all times
$t \in [0,T^\omega(x)-\delta]$ and all points $y \in C_\omega$
with $|y - x| < \eta$, and therefore
\begin{displaymath}
  T^\omega(y) \ge T^\omega(x) - \delta
  \quad\mbox{ for all }\quad
  y \in C_\omega
  \;\mbox{ with }\;
  |y - x| < \eta .
\end{displaymath}
Since this last estimate is trivially satisfied
if~$T^\omega(x) = 0$, this shows that the map~$T^\omega$
is lower semicontinuous. This completes the proof of the
lemma.
\end{proof}
\subsection{Gluing and the Final Strongly Admissible Semiflow}
\label{sec54}
After the preparations of the last two sections, we are finally in a
position to construct a strongly admissible semiflow for a given combinatorial vector field
$\cV$ on the underlying polytope~$X \subset \R^d$ of a simplicial complex~$\cX$.
So far, for every simplex $\omega\in\cX$ we have constructed a semiflow~$\phi^\omega$
on~$X$ which behaves as intended inside the flow tile~$C_\omega$.
In order to construct a stronlgy admissible semiflow $\phi : \R_0^+
\times X \to X$, we only have to concatenate solution pieces from the different
semiflows~$\phi^\omega$. We first construct the solution through a fixed point $x\in X$.
\begin{proposition}[Auxiliary Map]
\label{prop:finalsemiflow}
Consider the semiflows~$\phi^\omega : \R_0^+ \times \R^d \to \R^d$
guaranteed by Proposition~\ref{prop:semiflowexsimplex}, the associated
flow tiles~$C_\omega$ defined in~(\ref{def:flowtiles2}), and suppose
that~(\ref{prop:vectorfieldbound1}) holds. Let~$x \in X$ be fixed
and for $s\in\R^+_0\cup\{\infty\}$ set
\begin{displaymath}
  \Delta(s):=\begin{cases}
               [0,\infty) & \text{ if $s=\infty$,}\\
               [0,s] & \text{ otherwise.}
             \end{cases}
\end{displaymath}
There exist sequences~$(t_k)_{k\in\N}$ of non-negative real numbers,
$(x_k)_{k\in\N}$ of points in~$X$, and~$(\chi_k)_{k\in\N}$ of
partial maps $\chi_k:\R^+_0\pto X$ with $t_0 = 0$ and $x_0 = x$, and
such that for arbitrary $k\in\N$ and $\omega_k:=\sigma^\epsilon_{\min}(x_k)$
we have
\begin{itemize}
\item[(i)] $\dom \chi_k=\Delta(t_{k}+T^{\omega_k}(x_k))$, and
\item[(ii)] $\chi_k(t)=\phi^{\omega_k}(t-t_k,x_k)$ for all
$t\in\dom\chi_k$ with $t\geq t_k$.
\end{itemize}
Moreover, for arbitrary $k\geq 1$ we have
\begin{itemize}
\item[(iii)] $\left.\left(\chi_{k}\right)\right|_{\dom \chi_{k-1}}=\chi_{k-1}$,
as well as
\item[(iv)] $t_{k} = t_{k-1} + T^{\omega_{k-1}}(x_{k-1})$ and $x_{k} =
\phi^{\omega_{k-1}}(T^{\omega_{k-1}}(x_{k-1}),x_{k-1})$
if $T^{\omega_{k-1}}(x_{k-1})<\infty$.
\end{itemize}
In particular, $\chi_x:=\bigcup_{k\in\N}\chi_k$ is a well-defined partial
map with domain $\dom\chi_x=[0,T_x)$, where $T_x=\sup\setof{t_k\,:\,k\in\N}$, if the
sequence~$(t_{k})$ is strictly increasing, and $T_x=\infty$ otherwise.
\end{proposition}
\begin{proof}
We define the three sequences~$(t_k)$, $(x_k)$, and~$(\chi_k)$ recursively.
It is straightforward to verify that~$t_0 := 0$, $x_0 := x$ and
\begin{displaymath}
  \chi_0:\Delta(T^{\omega_0}(x_0))\ni t\mapsto \phi^{\omega_0}(t,x_0)\in X
\end{displaymath}
satisfy properties (i-iv). Assume now that~$t_k$, $x_k$, and~$\chi_k$ are
already defined in such a way that properties (i-iv) are satisfied. Let
\begin{displaymath}
  s_k:=\begin{cases}
         T^{\omega_k}(x_k) & \text{ if $T^{\omega_k}(x_k)<\infty$,} \\
         0 & \text{ otherwise.}
       \end{cases}
\end{displaymath}
Set $t_{k+1}:=t_k+s_k$, $x_{k+1}:=\phi^{\omega_k}(s_k,x_k)$, and define
$\chi_{k+1}:\Delta(t_{k+1}+T^{\omega_{k+1}}(x_{k+1}))\to X$ by
\begin{displaymath}
  \chi_{k+1}(t):=\begin{cases}
                   \chi_{k}(t) & \text{ if $t\in\dom\chi_{k}$,}\\
                   \phi^{\omega_{k+1}}(t-t_{k+1},x_{k+1})& \text{ otherwise.}
                 \end{cases}
\end{displaymath}
Again, it is straightforward to verify that~$t_{k+1}$, $x_{k+1}$,
and~$\chi_{k+1}$ satisfy properties (i-iv).

Note that due to $\omega_k = \sigma^\epsilon_{\min}(x_k)$ the exit
time~$T^{\omega_k}(x_k)$ defined in~(\ref{lem:exittime1}) is either~$+\infty$
or finite and strictly positive. If for some $k\in\N$ we have
$T^{\omega_k}(x_k)=+\infty$, then we get $s_k=0$, and therefore $t_{k+1}=t_k$,
$x_{k+1}=x_k$, and $\dom\chi_{k+1}=[0,\infty)$. Consequently,
all these sequences are constant for $k'>k$ and $T_x=\infty$.
If $T^{\omega_k}(x_k)<\infty$ for all $k\in\N$ , then the sequence~$(t_k)$
is strictly increasing and $T_x=\sup\setof{t_k\,\mid\,k\in\N}$.
\end{proof}
\begin{definition}[Construction of the Final Semiflow]
\label{def:finalsemiflow}
Under the assumptions of Proposition~\ref{prop:finalsemiflow} we set
$U:=\setof{(t,x)\in\R^+_0 \times X \,\mid\, 0\leq t<T_x}$, and we define
$\phi : U \to X$ for $(t,x)\in U$ by $\phi(t,x):=\chi_x(t)$.
\end{definition}
The construction of $\phi$ is visualized in Figure~\ref{fig:finalsemiflow}.
\begin{figure}[tb]
  \centering
  \raisebox{3.0cm}{\parbox{5.5cm}{
    \includegraphics[width=0.35\textwidth]{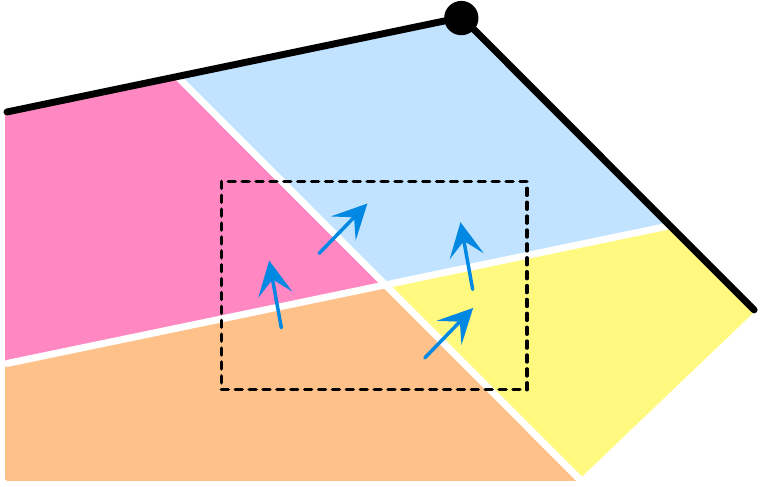}}}
  \hspace{0.5cm}
  \includegraphics[width=0.57\textwidth]{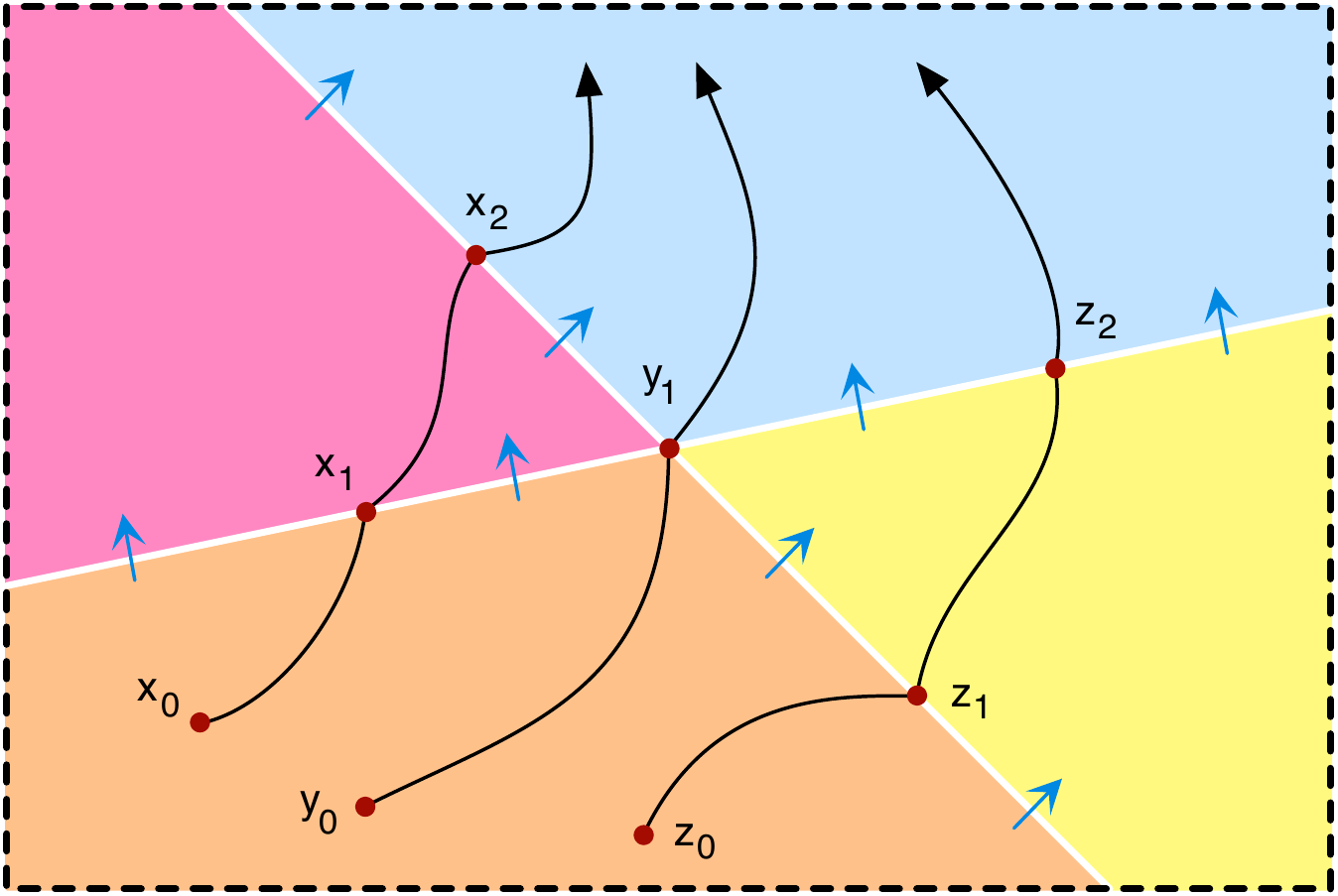}
  \caption{Construction of the final semiflow. Since in points~$\xi$
           on flow tile boundaries the semiflows~$\phi^\omega$ have
           to move from $\sigma^\epsilon_{\max}(\xi)$ to
           $\sigma^\epsilon_{\min}(\xi)$, the forward continuation
           is uniquely determined. The left image shows the region
           close to a triangle vertex in a two-dimensional simplex;
           part of its boundary is shown in black. The white lines
           indicate the $\epsilon$-cell boundaries. The image on the
           right is a larger version of the dashed region on the left.
           In it one can see three solution segments which all start in the
           lower left cell and end in the upper right cell, but in each
           case the cells visited along the way differ.}
  \label{fig:finalsemiflow}
\end{figure}
While the intuition for the above definition is straightforward,
its precise formulation may seem daunting at first. Therefore,
a few comments are in order:
\begin{itemize}
\item As it stands, we make no claim yet that
Definition~\ref{def:finalsemiflow} leads to a strongly admissible
semiflow. We still have to show that $T_x = \infty$ for all
$x \in X$, and that~$\phi : \R_0^+ \times X \to X$ is continuous.
Both of these statements are far from obvious.
\item It is, however, straightforward to verify that for all
$x \in X$ we have
\begin{equation} \label{def:finalsemiflow1}
  \phi(t, \phi(s,x)) = \phi(t+s,x)
  \quad\mbox{ for all }\quad
  s \in [0,T_x)
  \;\mbox{ and }\;
  t \in [0,T_{\phi(s,x)}) .
\end{equation}
This follows directly from the definition of~$\phi$ and
the semiflow properties of the involved semiflows~$\phi^{\omega}$.
\item By choosing the simplices~$\omega_k$ in
Definition~\ref{def:finalsemiflow} as~$\sigma^\epsilon_{\min}(x_k)$,
we have $t_{v}(x_k) > \epsilon$ for all vertices~$v \in \omega_k$,
and therefore this inequality is satisfied for all~$v \in \omega_k^-
\subseteq \omega_k$. In view of Lemma~\ref{lem:exittime} this is the
reason for the strict inequality $T^{\omega_k}(x_k) > 0$.
\item Notice further that the choice $\omega_{k+1} =
\sigma^\epsilon_{\min}(x_{k+1})$ is the only possible choice
for extending the semiflow definition in view of
Corollary~\ref{cor:vectorfieldbound}.
\end{itemize}
The remainder of this section is devoted to showing that
Definition~\ref{def:finalsemiflow} does indeed define a continuous
semiflow on~$X$ which is strongly admissible. This will be accomplished
by localizing the problem and using the compactness of the underlying
space. Our approach is based on the following two lemmas.
\begin{lemma}[Local Semiflow in the Interior of Flow Tiles]
\label{lem:localflowint}
Consider the mapping~$\phi$ introduced in
Definition~\ref{def:finalsemiflow}. Furthermore, let $\omega \in \cX$
denote an arbitrary simplex and let~$x \in \inte_X C_\omega$.
Then there exists a neighborhood~$U$
of~$x$ in~$C_\omega$ and a time~$\tau > 0$ such that for all $y \in U$ we have
$T_y \ge \tau$, and the map $\phi : [0,\tau] \times U \to X$
is continuous.
\end{lemma}
\begin{proof}
Let~$V$ denote an open neighborhood of~$x$ in~$C_\omega$.
Choose an open set~$U \subset V$ such that its closure~$\cl U$
is still contained in~$V$. Since~$X$ is compact, the same is true
for~$\cl U$. Due to the continuity of~$\phi^\omega : \R_0^+ \times
\R^d \to \R^d$, which was established in
Proposition~\ref{prop:semiflowexsimplex}, there exists a
time~$\tau > 0$ such that $\phi^\omega([0,\tau],\cl U) \subset V$.
According to $V \subset C_\omega$, this further implies the identity
$\phi(t,y) = \phi^\omega(t,y)$ for all $t \in [0,\tau]$ and
$y \in U$, and the result follows.
\end{proof}
\medskip

Lemma~\ref{lem:localflowint} shows that at an interior point of a flow tile,
there is a neighborhood~$U$ of the point and a positive time~$\tau$
such that~$\phi$ is defined on $[0,\tau] \times U$, as well as continuous
on this set. Notice, however, that both the neighborhood and the time do
depend on the chosen base point~$x$. The next result establishes an
analogous result for points~$x$ which lie on the boundary of at
least two flow tiles.
\begin{lemma}[Local Semiflow near the Boundary of Flow Tiles]
\label{lem:localflowbnd}
Consider the mapping~$\phi$ introduced in
Definition~\ref{def:finalsemiflow}. Furthermore, let~$x \in X$
be contained in at least two flow tiles. Then there exists a
neighborhood~$U \subset X$ of~$x$ and a time~$\tau > 0$ such that
for all $y \in U$ we have $T_y \ge \tau$, and the map $\phi :
[0,\tau] \times U \to X$ is continuous.
\end{lemma}
\begin{proof}
Since the point is contained in at least two flow tiles, the
vertex set
\begin{displaymath}
  \cE = \left\{ v \in \cX_0 \; \mid \; x_v = \epsilon \right\}
\end{displaymath}
is nonempty, and the simplex set~$\cX^\epsilon(x)$ defined
in~(\ref{def:charactcells3}) contains at least the two distinct
simplices~$\sigma^\epsilon_{\min}(x)$ and~$\sigma^\epsilon_{\max}(x)$
defined in~(\ref{def:charactcells1}) and~(\ref{def:charactcells2}),
respectively. For example, Figure~\ref{fig:finalsemiflow} illustrates
the situation in which~$\cE$ contains exactly two vertices, the
cell associated with~$\sigma^\epsilon_{\max}(x)$ is on the lower
left, and the upper right cell is associated
with~$\sigma^\epsilon_{\min}(x)$.

According to Proposition~\ref{prop:vectorfieldbound} there exists
a neighborhood~$V$ of~$x$ in~$X$ and a positive constant~$\alpha > 0$
such that
\begin{equation} \label{lem:localflowbnd1}
  f_v^\omega(y) \le -\alpha < 0
  \quad\mbox{ for all }\quad
  \omega \in \cX^\epsilon(x) , \;\;
  y \in V \cap \cl\cse{\omega} , \;\mbox{ and all }\;
  v \in \cE .
\end{equation}
Since $\tilde{V}:=\setof{y\in X \; \mid \; v\not\in\cE \implies y_v\neq\epsilon}$
is open in $X$ and clearly $x\in \tilde{V}$, by replacing~$V$ with the
intersection~$\tilde{V} \cap V$ we may also assume that
\begin{equation} \label{lem:localflowbnd0}
  y \in V
  \;\mbox{ and }\;
  y_v = \epsilon
  \quad\Rightarrow\quad
  v \in \cE.
\end{equation}
In addition, despite them not being continuous, all of the vector
fields~$f^\omega$ are uniformly bounded on~$X$, i.e., there exists
a constant $\beta > 0$ such that
\begin{equation} \label{lem:localflowbnd2}
  \left| f^\omega(y) \right| \le \beta
  \quad\mbox{ for all }\quad
  y \in X .
\end{equation}
Based on these estimates and the definition of~$\phi$, which in
turn is based on the definition of the semiflows~$\phi^\omega$
as almost everywhere solutions of differential equations with
right-hand side~$f^\omega$, we see that
\begin{equation} \label{eq:macximum-speed}
  \text{the solutions~$\phi(\cdot,y)$ have a maximum speed
        while they move in~$V$.}
\end{equation}
We now turn to selecting an open neighborhood $U \subset V$
of~$x$ and a time $\tau > 0$ as in the formulation of the
lemma. We claim that
\begin{displaymath}
  T^\omega(x) > 0
  \quad\mbox{ for all }\quad
  \omega \in \cX^\epsilon(x) \cap
     \left\{ \sigma^\epsilon_{\min}(x)^- , \; \sigma^\epsilon_{\min}(x)^+
     \right\} ,
\end{displaymath}
as well as
\begin{displaymath}
  T^\omega(x) = 0
  \quad\mbox{ for all }\quad
  \omega \in \cX^\epsilon(x) \setminus
  \left\{ \sigma^\epsilon_{\min}(x)^- , \; \sigma^\epsilon_{\min}(x)^+
    \right\} .
\end{displaymath}
To see this assume first that $\omega \in \cX^\epsilon(x) \cap \left\{
\sigma^\epsilon_{\min}(x)^- , \; \sigma^\epsilon_{\min}(x)^+ \right\}$.
Then in both cases $\omega^-\subset \sigma^\epsilon_{\min}(x)$, and we get
$T^\omega(x) > 0$ from the definition and properties of the exit
times~$T^\omega$ guaranteed by Lemma~\ref{lem:exittime}.
Consider now $\omega \in \cX^\epsilon(x) \setminus \left\{
\sigma^\epsilon_{\min}(x)^- , \; \sigma^\epsilon_{\min}(x)^+ \right\}$.
We will show that this assumption excludes the inclusion $\omega^- \subset \sigma^\epsilon_{\min}(x)$.
Otherwise,  we have $\omega^- \subset \sigma^\epsilon_{\min}(x) \subset \omega
\subset \omega^+$, which in turn implies either $\sigma^\epsilon_{\min}(x)
= \omega^-$ or $\sigma^\epsilon_{\min}(x) = \omega = \omega^+$. In either
case one immediately obtains $\omega \in \left\{ \sigma^\epsilon_{\min}(x)^- , \;
\sigma^\epsilon_{\min}(x)^+ \right\}$, which is a contradiction. Thus,
we have $\omega^- \not\subset \sigma^\epsilon_{\min}(x)$, and therefore
there exists a vertex $v\in\omega^- \setminus \sigma^\epsilon_{\min}(x)$.
Together with $\omega \subset \sigma^\epsilon_{\max}(x)$ we finally get
$T^\omega(x)=0$, again from Lemma~\ref{lem:exittime}.

Notice further that at least one of~$\sigma^\epsilon_{\min}(x)^-$
and~$\sigma^\epsilon_{\min}(x)^+$ are contained in~$\cX^\epsilon(x)$, namely
the simplex~$\sigma^\epsilon_{\min}(x)$, but not necessarily both.
Due to the continuity of the exit times, we can therefore find a
neighborhood~$U \subset V$ of~$x$ in~$X$ and a time $\tau > 0$ such that
\begin{equation} \label{lem:localflowbnd3}
  T^\omega(y) \ge \tau > 0
  \quad\mbox{ for all }\quad
  \omega \in \cX^\epsilon(x) \cap \left\{ \sigma^\epsilon_{\min}(x)^- ,
    \; \sigma^\epsilon_{\min}(x)^+ \right\}
  \;\mbox{ and all }\;
  y \in U \cap \cl\cse{\omega} ,
\end{equation}
as well as
\begin{equation} \label{lem:localflowbnd4}
  T^\omega(y) < +\infty
  \quad\mbox{ for all }\quad
  \omega \in \cX^\epsilon(x) \setminus \left\{ \sigma^\epsilon_{\min}(x)^- ,
    \; \sigma^\epsilon_{\min}(x)^+ \right\}
  \;\mbox{ and all }\;
  y \in U \cap \cl\cse{\omega} .
\end{equation}
We will prove that by possibly shrinking~$U$ further and decreasing
the value of~$\tau$, we can also achieve that
\begin{equation} \label{lem:localflowbnd5}
  \mbox{no solution of~$\phi$ which originates in~$U$ can reach the
  boundary of~$V$ within time~$\tau$.}
\end{equation}
Indeed, we get this easily from \eqref{eq:macximum-speed}
if we choose~$U$ in such a way that its boundary
has a positive smallest distance to the boundary of~$V$, and we then
choose~$\tau > 0$ smaller than the ratio of this minimum distance
and~$\beta$.

With these choices we will prove that
\begin{equation} \label{lem:localflowbnd6}
  \phi(\cdot,y) \;\mbox{ exists at least on }\; [0,\tau]
  \;\mbox{ for all }\; y \in U .
\end{equation}
To see this, we just need to follow the construction of this
solution in Definition~\ref{def:finalsemiflow} based on Proposition~ \ref{prop:finalsemiflow}.
Let $y\in U$ be arbitrary. We have to prove that $\tau<T_y$.
Assume the contrary. Then $T_y<\infty$ and the sequence $(t_k)$ associated with $y$
as in Proposition~ \ref{prop:finalsemiflow} is strongly increasing. In particular, for every $m\in\N$
we have $t_m<\tau$, which means that,
due to~(\ref{lem:localflowbnd5}), the solution always stays in~$V$.
Note, however, that, by~\eqref{lem:localflowbnd0},
within~$V$ only solution components associated
with vertices in~$\cE$ can cross~$\epsilon$.
There are only  finitely many of these thresholds and,
once the $v$-component drops below $\epsilon$,
it cannot increase again while in $V$ due to \eqref{lem:localflowbnd1}.
Thus, there exists a $k\in\N$ such that
$y_k:=\varphi(t_k,y)\in \cse{\sigma^\epsilon_{\min}(x)} $.
Let $\omega_k:=\sigma^\epsilon_{\min}(y_k)$.
It follows from \eqref{lem:localflowbnd3} that $T^{\omega_k}(y_k)\geq \tau$
and, in consequence, $t_{k+1}=t_k+T^{\omega_k}(y_k)\geq \tau$,
a contradiction proving \eqref{lem:localflowbnd6}.

To summarize, so far we have constructed a neighborhood~$U$
of~$x$ and a time $\tau > 0$ such that all the numbered
statements in this proof are satisfied. In order to complete
the proof of the lemma, we only have to show that~$\phi :
[0,\tau] \times U \to V$ is continuous. Notice first that
in view of~(\ref{lem:localflowbnd2}) and
Definition~\ref{def:finalsemiflow} one has for all
$(t,y), (t_0,y_0) \in [0,\tau] \times U$ the estimate
\begin{eqnarray*}
  \left| \phi(t,y) - \phi(t_0,y_0) \right| & \le &
    \left| \phi(t,y) - \phi(t_0,y) \right| +
    \left| \phi(t_0,y) - \phi(t_0,y_0) \right| \\[1ex]
  & \le &
    \beta \left| t - t_0 \right| +
    \left| \phi(t_0,y) - \phi(t_0,y_0) \right| .
\end{eqnarray*}
This immediately implies that it suffices to establish the
continuity of~$\phi(t,\cdot) : U \to V$ for every fixed
$t \in [0,\tau]$.

Thus, from now on we let $\tau^* \in [0,\tau]$ be arbitrary,
but fixed. Furthermore, let~$(y^{(k)})_{k \in \N}$ denote
a sequence of points in~$U$ which converges to~$y \in U$.
Each solution~$\phi(\cdot,y^{(k)})$ exists on the
interval~$[0,\tau^*]$, and as outlined in
Definition~\ref{def:finalsemiflow} based on Proposition~ \ref{prop:finalsemiflow},
as~$t$ increases through this interval the solution visits a well-defined sequence of
flow tiles~$C_\omega$. More precisely, in view
of~(\ref{lem:localflowbnd0}) and~(\ref{lem:localflowbnd1}),
for each~$k$ this sequence is a finite chain of nested proper
faces which is contained in the simplex face poset interval
with maximal simplex~$\sigma^\epsilon_{\max}(x)$ and minimal
simplex~$\sigma^\epsilon_{\min}(x)$. Clearly, there are only
finitely many possibilities for such chains. Therefore, the
sequence~$(y^{(k)})_{k\in\N}$ may be split into a finite
number of subsequences such that for each of these subsequences
the associated sequence of simplex sequences is constant. Thus,
it suffices to prove that~$\phi(\tau^*,y_{(k)})$ converges
to~$\phi(\tau^*,y)$ under the additional assumption that the
associated simplex sequence is the same for each~$k$.

For illustration purposes, we refer the reader again to
Figure~\ref{fig:finalsemiflow}. In this situation, if we consider
$y^{(k)} \to y_0$, we would separately discuss two subsequences
of points --- one which takes the route taken by the rightmost
solution, and another one which takes the route of the leftmost.
Note that the simplex sequence associated with the limit initial
point is not required to be the same as the sequences for the
solutions starting at~$y^{(k)}$.

Thus, there exists a simplex chain of nested proper faces
\begin{displaymath}
  \sigma^\epsilon_{\max}(x) \supseteq
  \omega_0 \supsetneq \omega_1 \supsetneq \ldots
  \supsetneq \omega_m \supseteq \sigma^\epsilon_{\min}(x)
  ,
\end{displaymath}
and for every $k \in \N$ an associated sequence of times
\begin{displaymath}
  0 = t_0^{(k)} < t_1^{(k)} < \ldots t_m^{(k)} \le
      t_{m+1}^{(k)} = \tau^*
\end{displaymath}
such that for all $\ell = 0,\ldots,m$
\begin{displaymath}
  \phi(t, y^{(k)}) \in C_{\omega_\ell}
  \quad\mbox{ for }\quad
  t_\ell^{(k)} \le t \le t_{\ell+1}^{(k)},
\end{displaymath}
and the points $y_\ell^{(k)} = \phi(t_\ell^{(k)}, y^{(k)})$
for $\ell = 0,\ldots,m-1$ satisfy
\begin{equation} \label{eq:phi-t-ell}
  y_{\ell+1}^{(k)} =
    \phi^{\omega_\ell}(T^{\omega_\ell}(y_\ell^{(k)}), y_\ell^{(k)})
  \quad\mbox{ and }\quad
  t_{\ell+1}^{(k)}=t_{\ell}^{(k)}\;+\;T^{\omega_\ell}(y_\ell^{(k)}) ,
\end{equation}
as guaranteed by Definition~\ref{def:finalsemiflow}.

We will prove by induction in $\ell$ that for every $\ell\in\{1,2,\ldots,m\}$
the limits $t_\ell:=\lim_{k\to\infty}t_\ell^{(k)}$ and
$y_\ell:=\lim_{k\to\infty}y_\ell^{(k)}$ exist, $y_\ell=\phi(t_\ell,y)$
and $\ell<m$ implies
\begin{equation} \label{eq:tl_T}
  t_{\ell+1}=t_\ell+T^{\omega_\ell}(y_\ell).
\end{equation}
Since $t_0^{(k)}=0$ we see that $\lim_{k\to\infty}t_0^{(k)}$ exists
and $t_0=0$. Due to our choice of the sequence~$(y^{(k)})_{k \in \N}$
we know that~$y_0^{(k)} = y^{(k)} \to y$ as $k \to \infty$.
Hence, $y_0=y=\phi(0,y)=\phi(t_0,y)$, which means that our claim
is satisfied for $\ell=0$. Now assume that our claim holds for some
index $\ell \in \{ 0, 1, \ldots, m-1 \}$. In view of $y_\ell^{(k)} \in
C_{\omega_\ell}$ and the closedness of the flow tile, one obtains the
inclusion $y_\ell \in C_{\omega_\ell}$. The continuity of~$T^{\omega_\ell}$
then implies $T^{\omega_\ell}(y_\ell^{(k)}) \to T^{\omega_\ell}(y_\ell)$.
Therefore, property \eqref{eq:phi-t-ell} implies the convergence
of $(t_{\ell+1}^{(k)})_{k\in\N}$ and, consequently, also
property~\eqref{eq:tl_T}. Moreover, the continuity
of~$\phi^{\omega_\ell}$ further yields
\begin{displaymath}
  y_{\ell+1}^{(k)} =
  \phi^{\omega_\ell}(T^{\omega_\ell}(y_\ell^{(k)}),
    y_\ell^{(k)}) \to
  \phi^{\omega_\ell}(T^{\omega_\ell}(y_\ell),
    y_\ell) .
\end{displaymath}
This proves that $(y_{\ell+1}^{(k)}))_{k\in\N}$ is convergent
and we have
\begin{displaymath}
  y_{\ell + 1} = \phi^{\omega_\ell}(T^{\omega_\ell}(y_\ell), y_\ell) =
  \phi^{\omega_\ell}(T^{\omega_\ell}(y_\ell),\phi(t_\ell,y)) =
  \phi(t_\ell+T^{\omega_\ell}(y_\ell),y)=\phi(t_{\ell+1},y),
\end{displaymath}
where the last equality follows from~\eqref{eq:tl_T} and the
definition of~$\phi$. Thus, we have verified our claim
for~$\ell + 1$, which completes our induction argument.
We would like to point out that in this argument it is
certainly possible that $T^{\omega_\ell}(y_\ell) = 0$.
This happens for example in the situation shown in
Figure~\ref{fig:finalsemiflow}.

Now consider the situation when  we finally reach $\ell = m$.
Then according to our setting we have
$T^{\omega_m}(y_m^{(k)}) \ge \tau^* - t_m^{(k)}$, which in
the limit $k \to \infty$ implies $T^{\omega_m}(y_m) \ge \tau^* - t_m$.
This finally gives
\begin{displaymath}
  \phi(\tau^*,y^{(k)}) = \phi^{\omega_m}(\tau^* - t_m^{(k)},
    y_m^{(k)}) \to
  \phi^{\omega_m}(\tau^* - t_m, y_m) ,
\end{displaymath}
as well as
\begin{displaymath}
  \phi^{\omega_m}(\tau^* - t_m, \phi(t_m,y)) =
  \phi(\tau^* - t_m, \phi(t_m,y)) =
  \phi(\tau^*,y) ,
\end{displaymath}
where we also used~(\ref{def:finalsemiflow1}). This establishes
the continuity of~$\phi(\tau^*,\cdot) : U \to V$, and the proof
of the lemma is complete.
\end{proof}
\medskip

After these preparations, we can now prove the main result
of this paper. It shows that the mapping~$\phi$ constructed at
the beginning of this section is indeed a strongly admissible
semiflow, which adheres to our design decisions from the
introduction.
\begin{theorem}[Existence of Strongly Admissible Semiflows]
\label{thm:finalsemiflow}
Let~$\cV$ be a combinatorial vector field on the simplicial
complex~$\cX$, and let $X \subset \R^d$, where $d := \card{\cX_0}$, denote
the underlying polytope of the standard geometric realization of $\cX$.
Furthermore, suppose that we have $\epsilon\in(0, 1/(6d))$
and consider the mapping~$\phi$ introduced in
Definition~\ref{def:finalsemiflow}. Then for every~$x \in X$ the
solution~$\phi(\cdot,x)$ is defined on all of~$\R_0^+$, and
the resulting map $\phi : \R_0^+ \times X \to X$ is continuous.
In addition, we have
\begin{equation} \label{thm:finalsemiflow1}
  \phi(t, \phi(s,x)) = \phi(t+s,x)
  \qquad\mbox{ for all }\qquad
  t,s \in \R_0^+ ,
\end{equation}
and~$\phi$ is strongly admissible in the sense of
Definition~\ref{def:admissibleflow}.
\end{theorem}
\begin{proof}
Let $x \in X$ be arbitrary. According to Lemmas~\ref{lem:localflowint}
and~\ref{lem:localflowbnd} there exists an open neighborhood~$U_x$
of~$x$ and a positive time~$\tau_x > 0$ such that for every $y \in U_x$
the mapping~$\phi(\cdot,y)$ is defined at least on~$[0,\tau_x]$, and
$\phi : [0,\tau_x] \times U_x \to X$ is continuous. Since~$X$ is compact,
we can find a finite set $\{x_1,x_2,\ldots,x_k\}\subset X$ such that
\begin{displaymath}
  X \subset U_{x_1} \cup \ldots \cup U_{x_n} .
\end{displaymath}
If we now let
\begin{displaymath}
  \tau := \min\left\{ \tau_{x_1},  \tau_{x_2},\ldots,  \tau_{x_n}
    \right\},
\end{displaymath}
then clearly $\tau>0$ and one can easily see that for every~$x \in X$ the
mapping~$\phi(\cdot,x)$
is defined at least on~$[0,\tau]$. Iteration of this map now implies that
in fact~$T_x = +\infty$ for all $x \in X$. The continuity of the resulting
map~$\phi : \R_0^+ \times X \to X$ is a direct consequence of the continuity of
the restricted maps, and the semiflow property~(\ref{thm:finalsemiflow1}) follows
from~(\ref{def:finalsemiflow1}). Finally, the strong admissibility of the
semiflow~$\phi$ is a direct consequence of Corollary~\ref{cor:vectorfieldbound}
and Proposition~\ref{prop:arrowtileexit}. This completes the proof of the
theorem.
\end{proof}
%
%
%

%
%
\bibliography{flowfill}
\bibliographystyle{abbrv}
\end{document}